\documentclass[11pt]{article}

\usepackage[margin=2cm]{geometry}
\usepackage{amsmath}
\usepackage{amssymb}
\usepackage{amsthm}
\usepackage{calrsfs}

\usepackage{psfrag}
\usepackage{graphicx,subfigure}
\usepackage{color}

\def\blue{\color{blue}}

\usepackage{enumerate}

\def\bg{{\bar{\gamma}}}
\def\bv{{\bar{v}}}
\def\bth{{\bar{\theta}}}
\def\bs{{\bar{\sigma}}}
\def\bu{{\bar{u}}}

\def\tg{{\tilde{\gamma}}}
\def\tv{{\tilde{v}}}
\def\tth{{\tilde{\theta}}}
\def\ts{{\tilde{\sigma}}}
\def\tu{{\tilde{u}}}

\def\dtg{{\dot{\tilde{\gamma}}}}
\def\dtv{{\dot{\tilde{v}}}}
\def\dtth{{\dot{\tilde{\theta}}}}
\def\dts{{\dot{\tilde{\sigma}}}}

\def\dpp{\dot{p}}
\def\dqq{\dot{q}}
\def\drr{\dot{r}}

\def\BO{{{O}}}

\def\R{\mathbb{R}}

\newcommand{\ga}{\alpha}
\newcommand{\tht}{\theta}

\newcommand{\tcb}{}

\newtheorem{theorem}{Theorem}
\newtheorem{lemma}{Lemma}[section]
\newtheorem{proposition}{Proposition}[section]

\newtheorem{definition}{Definition}[section]

\theoremstyle{remark}
\newtheorem{remark}{Remark}[section]

\begin{document}
\title{Localization in adiabatic shear flow \\via geometric theory of singular perturbations}
\author{  Min-Gi Lee \footnotemark[4] \and Theodoros Katsaounis\footnotemark[1] \footnotemark[2]  \footnotemark[3]
\and Athanasios E. Tzavaras\footnotemark[1] \footnotemark[2] \footnotemark[5]} % \footnotemark[4]}
\date{}

\maketitle
\renewcommand{\thefootnote}{\fnsymbol{footnote}}

\footnotetext[1]{ Computer, Electrical and Mathematical Sciences \& Engineering Division,
King Abdullah University of Science and Technology (KAUST), Thuwal, Saudi Arabia.}
\footnotetext[2]{Institute of Applied and Computational Mathematics, FORTH, Heraklion, Greece}
\footnotetext[3]{Department of Mathematics and Applied Mathematics, University of Crete, Heraklion, Greece}
\footnotetext[4]{Department of Mathematics, Kyungpook National University, Daegu, Korea}
\footnotetext[5]{Corresponding author : \texttt{athanasios.tzavaras@kaust.edu.sa}}
\renewcommand{\thefootnote}{\arabic{footnote}}

\maketitle

%baseline for working
%\baselineskip=18pt

%baseline for submission
\baselineskip=14pt

\begin{abstract}
We study localization occurring during high speed shear deformations of metals leading to the formation of shear bands.
The localization instability results from the competition between  Hadamard instability (caused by softening response) and the stabilizing effects of strain-rate hardening.  We consider a hyperbolic-parabolic system that expresses the above mechanism and construct
self-similar solutions of localizing type that arise as the outcome of the above competition.
The existence of self-similar solutions is turned, via a series of transformations,
 into a problem of constructing a heteroclinic orbit for an induced dynamical system.
The dynamical system is four dimensional but has a fast-slow structure
with respect to a small parameter capturing the strength of strain-rate hardening.
 Geometric singular perturbation theory is applied to construct the heteroclinic orbit as a transversal intersection of two invariant manifolds in the phase space.
\end{abstract}

%\tableofcontents

\section{Introduction}
Shear bands are narrow zones of intensely localized shear that are formed during the high speed plastic deformations of metals \cite{ZH, clifton_rev_1990,wright_survey_2002}.
They often precede rupture and are one of the striking instances of material instability leading to failure.  Considerable attention has been devoted to the
problem of shear band formation in both the mechanics and the applied mathematics literature, and section \ref{mathmodel} is devoted
to a presentation of the problem and a quick derivation of the hyperbolic-parabolic system
\begin{equation}
\label{intro-system1}
\begin{aligned}
v_t &= \big ( \theta^{-\alpha}\gamma^m v_x^n \big )_{x} , \\
 \gamma_t &= v_x ,\\
 \theta_t &= \theta^{-\alpha}\gamma^m (v_x)^{n+1} .\\
% \sigma &=\theta^{-\alpha}\gamma^m u^n.			%\label{eq:tau}
\end{aligned}
\end{equation}
The system describes the plastic shearing deformation of a specimen based on conservation of momentum and energy using
a model in thermoviscoplasticity:
\begin{equation}
\label{consti}
\sigma =\theta^{-\alpha}\gamma^m u^n \, , \qquad \mbox{ where \quad $u := \gamma_t = v_x$ }		 %\label{eq:tau}
\end{equation}
Equation \eqref{consti} is viewed as a yield stress or a plastic flow rule, with the parameters $\alpha$, $m$ and $n > 0$ describing respectively the
degree of thermal softening, strain hardening and strain-rate sensistivity.
We refer to section \ref{mathmodel} for a derivation of  \eqref{intro-system1} and a review of earlier work useful in understanding
the localization problem and its relevance to the present study; references \cite{clifton_rev_1990,shawki_shear_1989,wright_survey_2002,KT09}
can be consulted for further information on the mechanical aspects of the model.

The model \eqref{intro-system1} admits a special class of time-dependent solutions describing uniform shear (see \eqref{uss}) and the problem
of shear band formation is initially posed as a problem of stability for the uniform shearing solutions. As these are time-dependent, it leads to
analysis of non-autonomous problems and presents challenges even for linearized stability. We refer to section \ref{mathmodel}
for information on this aspect of the problem. Here, we focus on the regime of linearized instability and pose the problem of
understanding the behavior in the nonlinear regime.  A conjecture on the threshold of instability is offered by the asymptotic analysis in  \cite{KT09},
devising an effective equation that changes type along a threshold from forward to backward parabolic.
 It leads one to expect instability when $-\alpha+m+n$ changes sign from positive to negative value.

\tcb{
It is expedient to reformulate the problem \eqref{intro-system1}, in terms of the variables
$(u, \gamma, \theta)$, as a parabolic system where the diffusion coefficient is controlled by ordinary differential equations
\begin{equation}
\label{intro-system2}
\begin{aligned}
u_t &= \big ( \theta^{-\alpha}\gamma^m u^n \big )_{xx} , \\
 \gamma_t &= u ,\\
 \theta_t &= \theta^{-\alpha}\gamma^m u^{n+1}  \, . \\
% \sigma &=\theta^{-\alpha}\gamma^m u^n.			%\label{eq:tau}
\end{aligned}
\end{equation}
The systems \eqref{intro-system1} or \eqref{intro-system2} are  considered for $x \in \R$, $t > 0$.
}
The goal of this work is to construct a class of self-similar solutions for systems \eqref{intro-system1} (or \eqref{intro-system2})
of the form
\begin{equation} \label{intro-sols}
\begin{aligned}
 \gamma (t,x) &= t^a\Gamma\big( x \, t^\lambda \big), & v (t,x) &= t^bV\big( x \, t^\lambda \big), & \theta (t,x) &= t^c\Theta\big( x \, t^\lambda \big),\\
 {\sigma}(t,x) &= t^d\Sigma\big( x \, t^\lambda \big), & u(t,x) &= t^{a-1}U\big( x \, t^\lambda \big),
\end{aligned}
\end{equation}
where $\lambda > 0$ and the parameters $(\alpha, m, n)$ take values in the expected instability regime $-\alpha+m+n<0$.
Usually parabolic systems (such as \eqref{intro-system2}) admit diffusing self similar solutions constant on lines $\xi = \frac{x}{t^\rho}$.
By insisting on $\lambda >0$,
the solutions \eqref{intro-sols} will  propagate information on lines $x t^\lambda = const$ that focus around the origin.
The existence of such solutions explores the invariance of the system \eqref{intro-system1} under rescalings and we look for profiles with
$U(\xi)$, $\Gamma(\xi)$, $\Theta(\xi)$ even functions and $V(\xi)$ odd function.

We further demand that these profiles are localizing. We will call a self-similar function
\begin{equation}
\label{deflocal1}
f(t,x) = t^b F(x t^\lambda) \, , \quad \mbox{with $F(-\xi) = F(\xi)$ and $\lambda > 0$}
\end{equation}
{\it localizing} if it has the asymptotic behavior
\begin{equation}
\label{deflocal2}
F(\xi) = \BO (\xi^p)    \quad \mbox{ as $\xi \to \infty$ }
\end{equation}
and satisfies that $p < 0$ when $b > 0$ while $p > 0$ when $b < 0$. Under this definition, when $f(t,0)$ grows then $f(t,x)$ grows at a slower
rate when $x \ne 0$, while when $f(t,0)$ decays then $f(t,x)$ decays at a slower rate at $x \ne 0$. We will call a self-similar function with an odd-profile
$F(-\xi) = -F(\xi)$ localizing when its derivative $f_x( t,x)$ has the aforementioned behavior.

Applying the {\it ansatz} \eqref{intro-sols} leads
to the system of  singular ordinary differential equations
\begin{equation} \label{intro:ss-odes}
\begin{aligned}
 V'(\xi)&=U(\xi),\\
 \Sigma' (\xi) &=  b V(\xi) + \lambda \xi U(\xi) , \\
 a \Gamma(\xi) + \lambda \xi \Gamma'(\xi) &= U(\xi), \\
 c \Theta(\xi) + \lambda \xi \Theta'(\xi)&=\Sigma(\xi) U(\xi),\\
 \Sigma(\xi) &= \Theta(\xi)^{-\alpha} \Gamma(\xi)^m U(\xi)^n, \\
 \Gamma(0)=\Gamma_0>0, \quad U(0)&=U_0>0, \quad \text{$\xi \in [0,\infty)$}.
\end{aligned}
\end{equation}
This is viewed as a system of singular ordinary differential equations for $(V, \Sigma, \Gamma, \Theta)$ with $U$
defined by inverting \eqref{intro:ss-odes}$_5$.
%Our goal is to compare solutions of \eqref{intro:ss-odes} with simple solutions
%of nonlinear parabolic equations, like the fundamental solution of the heat equation
With the objective to compare these self-similar profiles to the fundamental solution of the heat or porous media equation,
we look for solutions so that $(\Gamma,\Theta)$  and $U$ is even; in turn implying that
$\Sigma$  is even while $V$ is odd (see section \ref{sec:scale} for details).
We impose
\begin{equation}
 V(0)=U'(0)=\Gamma'(0)=\Sigma'(0)=\Theta'(0)=0  \label{intro:bdry0}
\end{equation}
so that the symmetric extensions are smooth self-similar profiles.
The main result of this article is the construction of profiles solving \eqref{intro:ss-odes}, \eqref{intro:bdry0}.
It turns out that the induced self-similar solutions \eqref{intro-sols}
exhibit localizing in space behavior as time evolves, see sections \ref{sec:localization}
and \ref{sec:numerics}.

The idea of constructing self-similar localizing solutions for problems of shear band formation is introduced in \cite{KOT14} for the system
\begin{equation}
\label{explaw}
\begin{aligned}
v_t &=  ( e^{-\alpha \theta} v_x^n )_x
\\
\theta_t &= e^{-\alpha \theta} v_x^{n+1}
\end{aligned}
\end{equation}
modeling a non-Newtonian fluid with temperature dependent viscosity.
Due to special properties of \eqref{explaw}, \tcb{ the construction of self-similar solutions is reduced to finding a heteroclinic orbit
for a planar system of autonomous differential equations, which is achieved through phase space analysis.} A second step is taken in \cite{LT16}
where \eqref{intro-system1} is studied for parameters $\alpha = 0$ and $m < 0$ when the system simplifies to a system of two conservation laws. The problem is reduced
using geometric singular perturbation theory to the flow of a planar dynamical system in a two-dimensional invariant manifold. The present work addresses
the full model  \eqref{intro-system1} and due to the higher dimensionality of the problem a more elaborate version of geometric singular perturbation theory is needed.

The self-similar localizing solutions emerge as the combined outcome of Hadamard instability (that characterizes the system \eqref{intro-system1} for $n=0$ in the regime
$-\alpha + m < 0$) and the regularizing effect of momentum diffusion when $n > 0$.
This feature can be clearly seen in the linearized analysis of uniform shearing solutions for the simplified model \eqref{explaw} which indicates that the
combined effect of the two mechanisms amounts to Turing instability, see \cite{KOT14}. Moreover, existing linearized and nonlinear stability analyses that are available
for special instances of \eqref{intro-system1} and are outlined in section 2 corroborate this point.

The article is organized as follows:
Sections \ref{sec:scale} and  \ref{sec:formulation} deal with the formulation of the problem leading to \eqref{intro:ss-odes}, \eqref{intro:bdry0}.
The system \eqref{intro:ss-odes} is singular (at $\xi=0$) and non-autonomous and it does not fit under a general existence theory.
The singularity can be resolved and \eqref{intro:ss-odes} is desingularized using again the scale-invariance properties.
Furthermore, upon introducing a series of nonlinear transformations,  the construction of profiles for \eqref{intro:ss-odes}, \eqref{intro:bdry0} is accomplished by the
construction of a heteroclinic orbit for the four-dimensional dynamical system for $(p,q,r,s)$
\begin{equation}\label{intro:slow}\tag{S}
 \begin{aligned}
 \dot{p} &=p\Big(\frac{1}{\lambda}(r-a) + 2- \lambda p r -q\Big), \\
 \dot{q} &=q\Big(1 -\lambda p r -q\Big) + b p r,\\
 n\dot{r} &=r\Big(\frac{\alpha-m-n}{\lambda(1+\alpha)}(r-a) + \lambda pr + q +\frac{\alpha}{\lambda}r\big(s- \frac{1+m+n}{1+\alpha}\big) + \frac{n\alpha}{\lambda(1+\alpha)}\Big),\\
 \dot{s} &=s\Big(\frac{\alpha-m-n}{\lambda(1+\alpha)}(r-a) + \lambda pr + q - \frac{1}{\lambda}r\big(s- \frac{1+m+n}{1+\alpha}\big) - \frac{n}{\lambda(1+\alpha)}\Big),
 \end{aligned}
\end{equation}
parametrized by $(\lambda,\alpha,m,n)$. The initial conditions are transmitted to asymptotic conditions for the heteroclinic
as $ \eta(=\log\xi) \to -\infty$ while the behavior as $\eta \to \infty$ will capture the asymptotic behavior of the profiles.

The existence of solutions to \eqref{intro:ss-odes}, \eqref{intro:bdry0} is achieved in sections \ref{sec:scale} - \ref{sec:proof}. Their construction is reduced
to obtaining a heteroclinic orbit for \eqref{intro:slow} with prescribed asymptotic behavior as $\eta \to -\infty$.
\tcb{At the end of section \ref{sec:scale}, the reader will find an outline on how the construction of the profiles is reduced to obtaining a heteroclinic orbit for \eqref{intro:slow}.}
 The existence of a heteroclinic  orbit for \eqref{intro:slow} (with prescribed asymptotic behavior) is obtained in Theorem \ref{thm1} using the geometric theory of singular perturbations \cite{fenichel_persistence_1972,fenichel_asymptotic_1974,fenichel_asymptotic_1977,fenichel_geometric_1979,Jones_1995,KUEHN_2015}, exploiting the smallness of the parameter $n$.
Section \ref{sec:proof} contains the main part of the proof, motivated by the geometric arguments of \cite{Sz1991} and adapted to the present system through somewhat cumbersome
computations detailed in sections \ref{sec:singorb} and \ref{sec:thmproof}. The proof is based on  a more elaborate argument than the simple invariant manifold argument
for obtaining the corresponding result for the simplified model in \cite{LT16}.
In the present case, the finer structure inside the manifold is needed along with the persistence of the unstable and stable manifolds, see section \ref{sec:thmproof}.

The constructed self-similar solutions depend on two parameters $(U_0, \Gamma_0)$ describing the initial nonuniformity; the rate of localization $\lambda$ is
determined from $(U_0, \Gamma_0)$ via \eqref{eq:lambda}. Due to the construction necessities the rate has to obey the bound \eqref{eq:lambda}.
The solutions \eqref{intro-sols} provide an example of instabilty resulting in localization. Their localizing behavior is investigated in section \ref{sec:localization}, see
Proposition \ref{prop:ss} and section \ref{sec:emloc}.
In section \ref{sec:numerics}, the heteroclinic orbit is computed numerically using the standardized continuation software AUTO, \cite{Doedel_1981,DK_1986,DCFKSW_1999},
what leads to graphs of the profiles and the corresponding localizing solutions for various examples of material parameters.

To our knowledge, the localizing self-similar solutions are the first instance of depicting localizing behavior for a sufficiently broad model \eqref{intro-system1}
that embodies the basic  shear band formation mechanism proposed by Zener and Hollomon \cite{ZH} and Clifton \cite{clifton_rev_1990}, and encompasses
all the contributing factors  of thermal softening, strain hardening and strain-rate hardening.
They complement \cite{Tz_1987}, where shear bands are induced by energy supplied via the boundary.
Some of the key predictions of stress-collapse are common, but the present result has the conceptual
advantage to capture the emergence of localization as the combined result
of Hadamard instability with small viscosity effects.
It would be very interesting to study the stability of the solutions that are constructed here; this
appears a challenging problem.

%It is not clear at this point how the self-similar solutions will play a role in a composite deformation
%or will interact with boundary effects. This will be the subject of a future work. It is however, notable that the two independent parameters that enter in
%the solution the size of the nonuniformity $(\Gamma_0, U_0)$ and the rate $\lambda$ are interconnected through

%The model \eqref{intro-system1} describes a specimen deforming under simple shear; it describes the evolution of  the velocity $v$ (in the shear direction),
%the plastic strain $\gamma$,  and the temperature $\theta$, and the equations  stand respectively for the conservation of momentum,
%kinematic compatibility, and conservation of energy.  It is completed by using a thermoviscoplastic constitutive law,  obeying an empirical power law
%\begin{equation}
%\label{consti}
%\sigma =\theta^{-\alpha}\gamma^m u^n \, , \qquad \mbox{ where \quad $u := \gamma_t = v_x$ }		 %\label{eq:tau}
%\end{equation}
%Equation \eqref{consti} is viewed as a yield stress or a plastic flow rule, and the system \eqref{intro-system1}
%is widely considered  as an adequate simplified model to capture the phenomenon
%of localization and formation of shear bands;
%We refer to section \ref{mathmodel} for a derivation of  \eqref{intro-system1},  to  \cite{KT09}  for further details, and \cite{clifton_rev_1990,shawki_shear_1989,wright_survey_2002} can be consulted for information on the mechanical aspects
%of the model.

A preliminary report of these results, concerning the case with no strain hardening ($m = 0$), has been presented in the
Proceedings article \cite{KLT_HYP2016}.

\section{Description of the shear band formation problem}
\label{mathmodel}

The formation of shear bands \cite{CCHD,ZH}  is a phenomenon occuring during
high strain-rate plastic deformations of certain steels and other metal alloys. Instead of distributing evenly across the loaded region,
the shear strain concentrates in a narrow band with a concurrent elevation of the temperature in the interior of the band, \cite{ZH,CCHD,HDH}.
Shear bands are often precursors to rupture and their study has attracted considerable
attention including  experimental works  \cite{CCHD,HDH}, mechanical modeling
and  linearized analysis studies  ({\it e.g.} \cite{CDHS,FM,MC,wright_survey_2002} and references therein) and
nonlinear analysis investigations  \cite{DH_1983,Tz_1987,bertsch_effect_1991}.

\subsection{Modeling shear bands}
Shear bands appear and propagate as one dimensional structures (up to interaction times), and
many investigations focus on the study of one-dimensional, simple shear.
A specimen located in the $xy$-plain undergoes shear motion in the $y$-direction. The motion is described by the (plastic) shear strain
$\gamma(t,x)$, the strain rate $u(t,x) = \gamma_t (t,x)$, the velocity $v(t,x)$ in the shear direction, the temperature $\theta(t,x)$ and the shear stress
$\sigma (t,x)$ all defined in $(t,x)\in \mathbb{R}^+ \times \mathbb{R}$. It is described by the equations
\begin{equation} \label{mechmodel}
\begin{aligned}
 \gamma_t &= v_x  \\	
  v_t &= \sigma_x
 \\
 \theta_t &= \kappa \theta_{xx} + \sigma v_x, 	\\
\end{aligned}
\end{equation}
which stand respectively for the kinematic compatibility equation, the balance of momentum and
the balance of energy equation. Here, the elastic effects are neglected and all strain is considered to be plastic, and
a Fourier heat conduction is considered with $\kappa$ the thermal diffusivity.

Under shearing most materials deform in a uniform fashion until they break. By contrast, in high strain-rate deformations of
certain steels, it is observed that nonuniformities develop in the  plastic strain and  localize in a narrow region, called shear band; see
 Fig. \ref{ShearFlow} for a caricature of shear band forming.
 \tcb{
 Shear bands correspond to material instabilities and are usually observed in the interior of specimens. Typically, the maximum temperature is measured in the interior of the band \cite{CCHD}.
 }
\begin{figure}
\centering
\vspace{-0.1cm}
\includegraphics[scale=0.5]{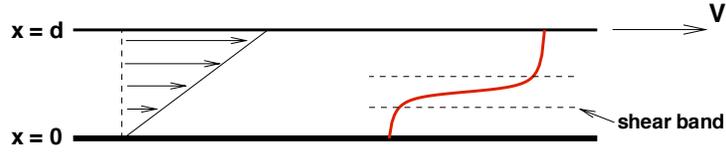}
\vspace{0.1cm}
\caption{Uniform shear versus shear band}.
\label{ShearFlow}
\end{figure}
It was recognized by Zener and Hollomon \cite{ZH} that the  high deformation speed has two effects:
First, an increase in the deformation speed changes the deformation conditions from isothermal
to nearly adiabatic. Under such conditions the combined effect of thermal softening and strain hardening
tends to produce net softening response. (Indeed, experimental observations of shear bands
are  typically associated  with strain softening response -- past a critical strain --
of the measured stress-strain curve \cite{CDHS}.)
Second, strain rate has an effect {\it per se}, and needs to be included in the constitutive modeling.

Both effects are captured by modeling shear band formation via constitutive models
within the framework  of thermoviscoplasticity:
\begin{equation}
\label{thermovpcl}
\sigma = f(\theta, \gamma, \gamma_t) \quad \mbox{where} \quad f_p (\theta, \gamma, p)  > 0 \, .
\end{equation}
The constitutive relation \eqref{thermovpcl} may be viewed as a yield surface or,  upon inverting it, as a plastic flow rule.
This suggests the terminology:
the material exhibits thermal softening at state variables $(\theta, \gamma, p)$
where $f_\theta(\theta, \gamma, p) < 0$, strain hardening at state variables where $f_\gamma(\theta, \gamma, p) > 0$, and strain softening when $f_\gamma(\theta, \gamma, p) < 0$.
The slopes of $f$ -- $f_\theta$, $f_\gamma$ or $f_p$ -- measure respectively the degree of thermal softening, strain hardening (or softening)
and strain-rate sensitivity, respectively.
The difficulty of performing high strain-rate experiments causes uncertainty as to the specific form of the constitutive form of the stress.  Here, we will
use two constitutive laws to describe the stress $\sigma$:
\begin{align}
%&  \sigma =  \theta^{-\alpha} \gamma^{m} \gamma_{t}^{n}, \quad & &  \text{ Power Law }, \label{PL0}\\
%& \sigma = e^{-\ga\tht} v_x^n, \quad & & \text{ Exponential Law} \label{ARL0}  \, .
&  \sigma =  \theta^{-\alpha} \gamma^{m} \gamma_{t}^{n}, \quad & &  \text{ power law }, \label{PL0}\\
& \sigma = e^{-\ga\tht} v_x^n, \quad & & \text{ exponential law} \label{ARL0}  \, .
\end{align}
The power law \eqref{PL0}  characterizes the response of the material. The parameter $\alpha>0$
measures the degree of thermal softening, $m>0$ measures the degree of strain hardening (or $m<0$ in case of a softening plastic flow), while $n>0$ measures strain-rate hardening and is typically small, $n \ll 1$, \cite{CDHS, clifton_rev_1990}. It is an empirical law and the parameters are determined by fitting experimental data.

We summarize the equations describing the model. For the power law
the resulting system reads
\begin{equation}
  \label{PLS}
  \begin{aligned}
    & v_{t} =  \sigma_{x},\\
    & \theta_{t} = \kappa \theta_{ x x}  +  \sigma \gamma_{t}, \\
    & \gamma_{t} = v_{x},  \\
    & \sigma  = \theta^{-\alpha}\gamma^{m}\gamma_{t}^n \, .
  \end{aligned}
\end{equation}
The system \eqref{PLS} captures the simplest mechanism proposed for shear localization in
high-speed deformations of metals \cite{ZH, clifton_rev_1990}, and an (isothermal) variant appears in early studies of necking \cite{HN77}.
Very often attention is restricted to the adiabatic model $\kappa = 0$ which is appropriate for the initial
development of shear bands under very fast deformations.

The exponential law does not exhibit any strain hardening and thus \eqref{mechmodel} decouples and
leads to the simplified system
\begin{equation}
  \label{ARS}
  \begin{aligned}
    & v_{t} = \sigma_{x},\\
    & \theta_{t} = \kappa \theta_{ x x}  +  \sigma v_x \\
    & \sigma  = e^{-\ga\theta} v_x^n \, .
  \end{aligned}
\end{equation}
The exponential law  can be interpreted as a temperature dependent non-Newtonian fluid and {is exactly} \eqref{explaw} for adiabatic deformations ($\kappa =0$).

\subsection{Uniform shearing solutions} \label{sec:uss}

In the study of shear bands,  a special problem is often considered  where an infinite slab of material
is  sheared by prescribed constant velocity $V=1$ at the upper plate while the lower plate is held fixed.
This is described by setting the plates at $x=0, 1$ and imposing prescribed (normalized) velocities $v(t,0) = 0$,  $v(t,1) = 1$, respectively.
The plates are thermally insulated:  $ \theta_{x}(t,0) = 0$,  $\theta_{x}(t,1) = 0$.
For the heat flux $Q$ one either uses the adiabatic assumption $Q = 0$ (equivalently $\kappa = 0$)
or alternatively a Fourier  law, $Q = \kappa \theta_{x}$ with thermal diffusivity
parameter $\kappa$.
Imposing adiabatic conditions projects the belief that, at high strain rates,
heat diffusion operates at a slower time scale than the time-scale of the development
of a shear band. It appears a plausible assumption for the shear band initiation process,
but not necessarily for the evolution of a developed band,
 due to the high temperature differences involved.

The model \eqref{PLS} admits  a special class of solutions describing uniform shearing: They emanate
from spatially uniform initial data $\gamma_0$ and $\theta_0$, and are obtained by the {\it ansatz} $\gamma_s (t) = t + \gamma_0 $ and $v_s (x) = x$
for the strain and velocity respectively. They are obtained
upon solving the ordinary differential equation
\begin{equation}
\label{uss2}
\frac{d \theta_s }{dt} = \sigma_s = \theta_s^{-\alpha} (t + \gamma_0)^m \, , \quad \theta_s(0) = \theta_0 \, ,
\end{equation}
and read
\begin{equation} \label{uss}
\begin{aligned}
v_s (x)  &=x \, ,   \quad  \gamma_s(t) = t+\gamma_0,  \quad
\\
\theta_s(t) &=  \left( \tfrac{1+\alpha}{1+m }\right )^{\frac{1}{1+\alpha}}  (t+\gamma_0)^{\frac{1 + m}{1+\alpha}}
 \left( 1 +  \tfrac{1+m}{1+\alpha} \big (  \theta_0^{1+\alpha}  - \tfrac{1+\alpha}{1+m} \gamma_0^{1+m} \big ) \tfrac{1}{(t+\gamma_0)^{m+1}}  \right)^{\frac{1}{1+\alpha}}
 \\
%&=
%\left(\tfrac{1+\alpha}{1+m} (t+\gamma_0)^{1+m}  +
% \left (  \theta_0^{1+\alpha} - \tfrac{1+\alpha}{1+m} \gamma_0^{1+m}  \right )   \right)^{\frac{1}{1+\alpha}}\\
  \sigma_s(t)&=
   \left( \tfrac{1+\alpha}{1+m}\right )^{-\frac{\alpha}{1+\alpha}}  (t+\gamma_0)^{\frac{-\alpha + m}{1+\alpha}}
   \left( 1 +  \tfrac{1+m}{1+\alpha} \big (  \theta_0^{1+\alpha}  - \tfrac{1+\alpha}{1+m} \gamma_0^{1+m} \big ) \tfrac{1}{(t+\gamma_0)^{m+1}}  \right)^{\frac{-\alpha}{1+\alpha}} .
\end{aligned}
\end{equation}
Equation \eqref{uss}$_3$ describes the stress-strain curve $\sigma_s$ versus $\gamma_s$ for uniform shear.
The stress-strain curve is increasing when $\alpha < m$ but it is decreasing for large times when $\alpha > m$.
Here, we are interested in the regime $\alpha > m$ where thermal softening dominates strain hardening and produces net softening.

\subsection{On the stability of the uniform shearing solution}

The system \eqref{intro-system1} for $n=0$ is a first-order system. When $\alpha > m$, the initial value problem  has two purely imaginary eigenvalues in a regime
of strain beyond the maximum of the stress-strain curve (see Appendix \ref{append:hadamard}).
Accordingly, the linearized system (for $n=0$) around the uniform shearing solution \eqref{uss} exhibits {\it Hadamard instability}, see Appendix \ref{append:hadamard}.

The stability of the uniform shear solution for $n > 0$ has been the objective of many investigations. Since \eqref{uss} is time dependent this
leads to investigations of non-autonomous systems. A natural way to define stability is to consider
\begin{equation} \label{ussgr}
\begin{aligned}
\gamma_s^* (t) = t+\gamma_0,  \quad
\theta_s^* (t) &=  \left( \tfrac{1+\alpha}{1+m }\right )^{\frac{1}{1+\alpha}}  (t+\gamma_0)^{\frac{1 + m}{1+\alpha}} \, ,
%\quad   \sigma_s^*(t)&=
%   \left( \tfrac{1+\alpha}{1+m}\right )^{-\frac{\alpha}{1+\alpha}}  (t+\gamma_0)^{\frac{-\alpha + m}{1+\alpha}}
\end{aligned}
\end{equation}
the functions capturing the growth of the uniform shearing solution, and to study the relative perturbations
\begin{equation}
u (t,x) = v_x (t,x) \, , \quad \hat\Gamma (t,x) = \frac{\gamma (t,x)}{\gamma_s^* (t)} \, , \quad \hat\Theta (t,x) = \frac{\theta (t,x)}{\theta_s^* (t)}.
\end{equation}
\begin{itemize}
\item
The uniform shear solution is {\it asymptotically stable} when the solution emanating from small perturbations
of \eqref{uss} satisfies that $(u, \hat\Gamma, \hat\Theta) \to (1,1,1)$ as time goes to infinity.
\item
The uniform shear solution is {\it unstable} if for small perturbations of \eqref{uss} the relative perturbations
$(u, \hat\Gamma, \hat\Theta)$ drift away from $(1,1,1)$ as time increases.
\end{itemize}

This notion of stability is used in nonlinear stability studies of shear bands \cite{DH_1983,Tz_1986} as well as
in linearized stability analyses by  Molinari and Clifton \cite{MC,FM}
who coined the name {\it stability analysis of relative perturbations}. The problem of stability is presently resolved only for the special cases
$m=0$ or $\alpha = 0$ for \eqref{intro-system1};  these are cases that the system decouples and reduces to simpler models:
\begin{itemize}
\item[(i)] Case $m=0$: The uniform shear is linearly stable when $-\alpha + n > 0$ and linearly unstable when $-\alpha + n < 0$ \cite{MC,FM};
it is nonlinearly stable in the region $-\alpha + n > 0$, \cite{Tz_1986}.
\item[(ii)] Case $\alpha =0$, $m > -1$: The uniform shear is linearly stable when $m + n > 0$ and linearly unstable when $m+ n < 0$, \cite{FM,tzavaras_nonlinear_1992};
it is nonlinearly stable in the region $m + n > 0$, \cite{tzavaras_plastic_1986}.
\end{itemize}

Understanding of the nature of the instability is offered in \cite{KOT14} for the model \eqref{explaw}, which has the special property
that both the nonlinear and the linearized analysis of relative perturbations is reduced to studying {\it autonomous systems}.
In particular, linearized stability (or instability) can be accessed  via analyzing Fourier modes; see \cite{KOT14}.
For $n=0$, the linearized stability analysis predicts  exponential growth of the high frequency modes, leading to what is usually termed as {\it Hadamard instability}.
By contrast, when $n > 0$ the linear modes are still unstable and their growth rates are increasing with frequency but they are
uniformly bounded by a bound independent of the frequency.  The behavior of the linearized system around the uniform shearing solution
for the full system \eqref{intro-system1} is at present open;  the conjecture is that it has the same structure as described above for relative perturbations of \eqref{explaw} when $n>0$ is small,
and it is stable past a certain threshold.
This is corroborated by linearized analysis for the special case \eqref{intro-system1} with $\alpha =0$, $m =-1$, $n <1$, which again
leads to the study of autonomous systems for relative perturbations, \cite{KLT_2016}.

\subsection{The nonlinear regime}

\tcb{ In the unstable parameter regime,
at the initial stage unstable modes start to grow and this process can be captured by the linearized problem.}
The second stage of localization lies within the realm of nonlinear analysis. The  question arises
 how the high frequency oscillations resulting from Hadamard instability interact with
the nonlinearity and the viscosity to form a coherent structure.
An asymptotic criterion accounting for the nonlinear aspects of localization is derived in \cite{KT09}.
Based on ideas from the theory of relaxation system and the Chapman Enskog expansion,  an effective equation is derived for the nonlinear dynamics \eqref{intro-system1}.
It predicts stability in the regime $- \alpha + m + n > 0$ and instability in the regime $-\alpha + m + n < 0$, see \cite{KT09}.

Insight on how coherent structures form can be offered by investigating self-similar solutions \eqref{intro-sols}.
It is customary in studies of parabolic systems (like \eqref{intro-system2}) to investigate diffusing self-similar
solutions corresponding to the parameter selection $\lambda < 0$. By contrast, self-similar solutions with  $\lambda > 0$ tend to propagate information
along the lines $t^\lambda x = const.$ and thus to localize around the point $x=0$.
Self-similar localizing solutions were established in \cite{KOT14}  for the model \eqref{explaw}  using a phase-plane analysis for the resulting two-dimensional
system. They will be pursued also here for the power law
 \eqref{intro-system1}.

\section{Self-similar solutions} \label{sec:scale}

%We express \eqref{intro-system1} in terms of the quantity $u=v_x$ in the form
%\begin{equation}
%\label{intro-system2}
%\begin{aligned}
%u_t &= \big ( \theta^{-\alpha}\gamma^m u^n \big )_{xx} , \\
% \gamma_t &= u ,\\
% \theta_t &= \theta^{-\alpha}\gamma^m u^{n+1} ,\\
%% \sigma &=\theta^{-\alpha}\gamma^m u^n.			%\label{eq:tau}
%\end{aligned}
%\end{equation}
%and consider this system for $x \in \R$, $t > 0$.
We consider the  system \eqref{intro-system1}  (or the system \eqref{intro-system2}) in the domain $x \in \R$, $t > 0$
and note that both systems are invariant under a  family of scaling transformations:
if $(\gamma,u,v,\theta,\sigma)$ satisfy \eqref{intro-system1}, with $u$, $\sigma$ connected via \eqref{consti}, then for any $\lambda \in \mathbb{R}$ and $\rho > 0$ the rescaled functions
$(\gamma_\rho, u_\rho,v_\rho,\theta_\rho,\sigma_\rho)$ defined by
\begin{equation}\label{eq:scale}
\begin{aligned}
 \gamma_\rho(t,x) &= \rho^a\gamma(\rho^{-1}t,\rho^\lambda x), &
 v_\rho(t,x) &= \rho^bv(\rho^{-1}t,\rho^\lambda x),\\
 \theta_\rho(t,x) &= \rho^c\theta(\rho^{-1}t,\rho^\lambda x), &
 \sigma_\rho(t,x) &= \rho^d\sigma(\rho^{-1}t,\rho^\lambda x),\\
 u_\rho(t,x) &= \rho^{b+\lambda}\gamma(\rho^{-1}t,\rho^\lambda x)
\end{aligned}
\end{equation}
also satisfies \eqref{intro-system1}, provided
\begin{equation} \label{eq:exponents}
\begin{aligned}
 a&:= a_0 + a_1 \lambda=\frac{2+2\alpha-n}{D} + \frac{2(1 + \alpha)}{D}\lambda, & b&:=b_0 + b_1\lambda=\frac{1+m}{D} + \frac{1+m+n}{D}\lambda ,\\
 c&:=c_0 + c_1\lambda=\frac{2(1+m)}{D} + \frac{2(1+m+n)}{D}\lambda, & d&:=d_0 + d_1\lambda=\frac{-2\alpha + 2m +n}{D} + \frac{2(-\alpha+m+n)}{D}\lambda,
\end{aligned}
\end{equation}
and
\begin{equation}
\label{defD}
D = 1+2\alpha-m-n \, .
\end{equation}
\tcb{
The same scaling trasformation leaves invariant solutions of \eqref{intro-system2}.
We note there are two independent scaling parameters in \eqref{eq:scale}, $\rho$ and $\lambda$, while the
remaining parameters are
determined by the relations \eqref{eq:exponents}, \eqref{defD}.
}
Throughout this work, the material parameters $(\alpha,m,n)$ will be restricted to the range
\begin{equation}
\label{eq:paramrange}
 \begin{aligned}
  \alpha>0\quad&\text{(thermal softening)},\\
  m>-1 \quad&\text{(strain softening/hardening)}, \\%\label{eq:a1}\\
  n>0 \quad&\text{(strain rate sensitivity)},\\ %\label{eq:a2}\\
  -\alpha+m+n<0 \quad&\text{(unstable regime)}.  \\%\label{eq:a3}\\
%   0< \lambda < \frac{2(\alpha-m-n)}{1+m+n}\left(\frac{1+m}{1+m+n}\right) \quad&\text{(localizing rate bound)}. %\label{eq:a4}
\end{aligned}
\end{equation}
Observe that \eqref{eq:paramrange}$_4$ implies that $-\alpha + m <  0$ and thus we are in the regime of net softening, where
the associated hyperbolic system with $n=0$ loses hyperbolicity, see Appendix \ref{append:hadamard}.  Moreover,
$D>1+\alpha>1$   {  while $1 + \alpha - n > 1 + m > 0$. }

Solutions of \eqref{intro-system1} or \eqref{intro-system2} that are self-similar with respect to the scaling transformation
 \eqref{eq:scale}, \eqref{eq:exponents}, \eqref{defD} have the form
\begin{equation}\label{eq:ORItoCAP}
\begin{aligned}
 \gamma(t,x) &= t^a\Gamma(t^\lambda x), & v(t,x) &= t^b V(t^\lambda x), &\theta(t,x) &= t^c \Theta(t^\lambda x),\\
 \sigma(t,x) &= t^d \Sigma(t^\lambda x), & u(t,x) &= t^{b+\lambda} U(t^\lambda x) \, ,
\end{aligned}
\end{equation}
and depend on one parameter, $\lambda$. In the sequel, we are interested in constructing solutions \eqref{eq:ORItoCAP}
defined in the domain $x \in \R$, $t > 0$ for values of the parameter $\lambda > 0$.

\tcb{
To motivate the role of self-similar solutions with $\lambda>0$ and some forthcoming selections,
recall that the fundamental solution of the
heat equation $u_t = u_{xx}$ is of self-similar form

$$
u(t, x) = \frac{1}{\sqrt{t}}U ( \frac{x}{\sqrt{t}} )\, .
$$
Moreover, power nonlinear parabolic diffusion equations (such as the porous media) admit self-similar solutions
which correspond to values $\lambda < 0$ and capture the effect of {\it diffusion}.
We are interested here to investigate whether the couplings with the remaining equations in \eqref{intro-system2}
can lead to the opposite behavior, of {\it localization}, and we seek existence of self-similar solutions
with the parameter in the range $\lambda>0$.
Note that profiles of the form \eqref{eq:ORItoCAP} with $\lambda > 0$ are constant on lines $\xi = t^\lambda x$  and are thus expected
to localize in space as time evolves. In order to compare the solutions we intend to construct for $\lambda > 0$,
with the existing self-similar solutions  of nonlinear parabolic equations,
we seek solutions where  $u(t,x)$ admits a maximum located at $x = 0$ for all times.
This imposes for self-similar solutions that
\begin{equation}
\label{bc1}
U'(0) = 0
\end{equation}
and as we will see this induces some symmetry properties to solutions.

Introducing the ansatz \eqref{eq:ORItoCAP} into \eqref{intro-system1}
gives a system of ordinary differential equations
\begin{equation}
\label{eq:ss-odes}
\begin{aligned}
V' &= U
\\
\Sigma' &= b V + \lambda \xi U
\\
 c \Theta + \lambda \xi \Theta' &=\Sigma U
 \\
a \Gamma + \lambda \xi \Gamma'  &= U
\end{aligned}
\end{equation}
together with an algebraic equation,
 \begin{equation}
\label{eqnsigma}
 \Sigma = \Theta^{-\alpha} \Gamma^m U^n \, ,
\end{equation}
obtained from \eqref{consti}.
This is viewed as a first order system for the variable $(V , \Sigma, \Theta , \Gamma)(\xi)$ with
 $U(\xi)$  determined by  inverting \eqref{eqnsigma}.
In principle, solutions
of \eqref{eq:ss-odes} will depend on five data inputs:  the initial data $(V_0, \Sigma_0, \Theta_0, \Gamma_0)$ and the
parameter $\lambda$.
The system \eqref{eq:ss-odes} is non-autonomous and singular at $\xi = 0$, what  imposes compatibility conditions.

For smooth initial data,  solutions of  \eqref{intro-system2} either blow-up in finite time
or they are as smooth as the initial data  \cite[Theorem 1]{Tz_1987} (in fact analytic for analytic initial data).
Thus the self-similar solutions will be sought to be smooth with their maxima fixed at the origin. The latter can be always achieved
due to the translation invariance of  \eqref{intro-system2}.
Next, we discuss the conditions imposed by these requirements:
The  initial conditions
\begin{equation}
\label{eq:ic0}
V(0) = V_0 \in \mathbb{R} \, , \quad  \Sigma (0) = \Sigma_0 >0 ,  \quad \Theta (0) = \Theta_0>0 \, , \quad \Gamma (0) = \Gamma_0>0 \, , \quad U(0) = U_0>0
\end{equation}
are supplemented with \eqref{bc1}.
Since the solution is smooth the
singularity at $\xi = 0$ imposes two compatibility conditions on
the data which together with \eqref{eqnsigma} imply
\begin{equation}
\label{compatibility}
a \Gamma_0 = U_0 \, , \quad
c \Theta_0 = \Sigma_0 U_0 \, , \quad \Sigma_0 = \Theta_0^{-\alpha} \Gamma_0^m U_0^n
\end{equation}
with $a,b, c$ given by \eqref{eq:exponents}. The condition $U'(0) = 0$ together with the smoothness of
the solution yields upon differentiating \eqref{eq:ss-odes} and \eqref{eqnsigma}
\begin{equation}
\label{algsys}
\begin{aligned}
(a + \lambda) \Gamma'(0) &= U' (0) = 0
\\
(c + \lambda) \Theta'(0) &= \Sigma'(0) U(0) + \Sigma (0) U' (0)
\\
&= -\alpha \frac{\Sigma_0 U_0}{\Theta_0} \Theta'(0) + m \frac{\Sigma_0 U_0}{\Gamma_0} \Gamma' (0).
\end{aligned}
\end{equation}
By \eqref{eq:exponents} and \eqref{eq:paramrange},  for $\lambda >0$ we have $a > 0$, $c > 0$, hence
\begin{equation}
\label{dercon1}
\Gamma' (0) = 0 \, , \quad \Theta'(0) = 0 \, .
\end{equation}
Again by \eqref{bc1},  \eqref{eqnsigma}, \eqref{eq:ss-odes}$_2$ \tcb{and $b>0$ },
\begin{equation}
\label{dercon2}
\Sigma'(0) = 0 \, , \quad V(0) = V_0 = 0 \, .
\end{equation}
Finally, since \eqref{eq:ss-odes} is invariant under the change of variables
$$
\xi \to -\xi \, , \quad V \to - V \, , \quad U \to U \, , \quad  \Theta \to \Theta \, , \quad \Gamma \to \Gamma
$$
it admits solutions such that $U$, $\Theta$, $\Gamma$ and $\Sigma$ are even functions of $\xi$, while
$V$ is an odd function of $\xi$.

In summary, we proceed as follows: We first construct a solution $(V(\xi) , \Sigma(\xi), \Theta(\xi), \Gamma(\xi))$
of \eqref{eq:ss-odes} defined for $\xi \in [0,\infty)$ and set $U(\xi)$ by \eqref{eqnsigma}. The solution will be
sought subject to the data
\begin{equation}
U'(0)=\Gamma'(0)=\Sigma'(0)=\Theta'(0)=0 \label{eq:bdry0}
\end{equation}
\begin{equation}
\label{eq:ic}
V(0) = 0, \quad U(0) = U_0>0 \, , \quad \Gamma (0) = \Gamma_0>0 \, , \quad \Theta (0) = \Theta_0>0 \, , \quad \Sigma (0) = \Sigma_0>0 \, .
\end{equation}
satisfying the compatibility conditions  \eqref{compatibility} for some $\lambda > 0$.
It is not a-priori clear that this problem is not overdetermined and we
give a detailed analysis of this point  in Section \ref{sec:char}.
The constructed solution is then extended on $(-\infty, 0]$ by setting
$$
V(-\xi ) = -V(\xi) \, , \quad U(-\xi) = U(\xi) \, , \quad \Theta (-\xi) = \Theta (\xi) \, , \quad \Gamma (-\xi) = \Gamma (\xi) \, ,
$$
that is we use an odd extension for $V$ and even extensions for  $U$, $\Theta$, $\Gamma$ and $\Sigma$.
Given the material constants $(\alpha, m, n)$ there are two independent parameters in the problem,
which may be viewed as $\Gamma_0$, $U_0$ and $\lambda > 0$ subject to the constraint
\begin{equation}
\label{rate}
\frac{U_0}{\Gamma_0} = \frac{2+2\alpha-n}{D} + \frac{2(1 + \alpha)}{D}\lambda \, .
\end{equation}
The remaining constants $\Theta_0$ and $\Sigma_0$ are determined via \eqref{compatibility}.

The profiles are constructed in the forthcoming sections \ref{sec:formulation}-\ref{sec:proof}. Then in section
\ref{sec:localization} we check that the profiles are {\it localizing} according to the definition \eqref{deflocal1}-\eqref{deflocal2}
in the Introduction. This is based on the asymptotic behavior of the constructed profiles $\xi \to \infty$, established
in Proposition \ref{prop:ss}.
}

Note that the uniform shearing solution is achieved as a self-similar profile for  $\lambda = -\frac{1+m}{2(1+\alpha)}<0$ and
\begin{equation*}
 \Gamma(\xi) = U(\xi)=U_0, \quad V(\xi) = U_0\xi, \quad  \Theta(\xi) = \Big( \frac{1+\alpha}{1+m} U_0^{1+m+n}\Big)^{\frac{1}{1+\alpha}}, \quad \Sigma(\xi) = \Big(\frac{1+\alpha}{1+m}\Big)^{\frac{-\alpha}{1+\alpha}} U_0^{\frac{-\alpha+m+n}{1+\alpha}}.
\end{equation*}
The uniform shear should be contrasted to the solutions that are constructed here which exhibit localizing behavior:
the growth of the strain is superlinear at  the origin and the profiles of the solution (at fixed times) localize  as time proceeds, see Section \ref{sec:localization}.

\medskip

We give a short roadmap of how we proceed to construct the solution of \eqref{eq:ss-odes}, \eqref{eq:bdry0}, \eqref{eq:ic} and
determine its properties.
\begin{itemize}
\item[(a)]
In section \ref{sec:formulation} we de-singularize \eqref{eq:ss-odes} and re-formulate it as an autonomous system,
see \eqref{eq:slow}.
\item[(b)]
In section \ref{sec:equil} we determine two equilibria $M_0$ and $M_1$  so that a heteroclinic orbit of \eqref{eq:slow} provides a meaningful, for the localization problem, self-similar profile.
\item[(c)]
Section \ref{sec:char} discusses the behavior of \eqref{eq:ss-odes} near $\xi = 0$ and what it implies for the heteroclinic orbit.
\item[(d)]
Section \ref{sec:proof}  is the core of the proof: the geometric singular perturbation theory is used to construct a heteroclinic
orbit joining $M_0$ to $M_1$ for system \eqref{eq:slow}.
\item[(e)] In Section \ref{sec:localization} we show that the self-similar profiles are localizing in the sense of Definition
\eqref{deflocal1}, \eqref{deflocal2} in the Introduction.
In section \ref{sec:numerics} we outline  a continuation method to compute the heteroclinic orbits via a standard package and provide
numerical examples of the emerging solutions.
\end{itemize}
As an outcome of this construction, it turns out  there is a two parameter family of solutions depending on the data $U_0$ and $\Gamma_0$ with the rate $\lambda$ determined via \eqref{rate}.
The dynamic stability of the solutions is a challenging open problem.

\section{Reduction to the construction of a heteroclinic orbit} \label{sec:formulation}
The goal of this section is to derive an equivalent system \eqref{eq:slow} to \eqref{eq:ss-odes} that is autonomous and to turn the problem of
constructing profiles for \eqref{eq:ss-odes} to the construction of a heteroclinic orbit for \eqref{eq:slow}. We employ techniques from \cite{KOT14} and \cite{LT16}.
The novelty of the present analysis lies in the higher dimensionality of the resulting system especially  with regard to the construction of the heteroclinic orbit.

\subsection{De-singularization}
We regard \eqref{eq:ss-odes} as a boundary-value problem in the right  half-line $\xi \in [0,\infty)$ subject to the boundary conditions \eqref{eq:bdry0} and proceed to
de-singularize it. The system \eqref{eq:ss-odes} is itself scale invariant: Given a solution $\big(\Gamma(\xi), V(\xi), \Theta(\xi), \Sigma(\xi), U(\xi)\big)$ the rescaled
function $\big(\Gamma_\rho(\xi), V_\rho(\xi), \Theta_\rho(\xi), \Sigma_\rho(\xi), U_\rho(\xi)\big)$ defined by
\begin{equation}
\label{selfsimilardef2}
\begin{aligned}
 \Gamma_\rho(\xi)&=\rho^{a_1}\Gamma(\rho\xi), & V_\rho(\xi)&=\rho^{b_1}V(\rho\xi),  \quad  \Theta_\rho(\xi)=\rho^{c_1}\Theta(\rho\xi),\\
 \Sigma_\rho(\xi)&=\rho^{d_1}\Sigma(\rho\xi), & U_\rho(\xi)&=\rho^{b_1+1}U(\rho\xi)=\rho^{a_1}U(\rho\xi)
\end{aligned}
\end{equation}
is again a solution.
The class of functions that remain invariant under this scaling transformation is
 $$\big(\Gamma(\xi), V(\xi), \Theta(\xi), \Sigma(\xi), U(\xi)\big)=\big(A\xi^{-a_1}, B\xi^{-b_1},C\xi^{-c_1},D\xi^{-d_1},E\xi^{-a_1}\big)$$
where $A, B, C, D, E$ constants. Such a function is singular at $\xi =0$ and fails to satisfy \eqref{eq:bdry0}. Nevertheless, it suggests the change of variables
\begin{equation} \label{eq:CAPtoBAR}
\begin{aligned}
 \bg(\xi)&=\xi^{a_1}\Gamma(\xi), &
 \bv(\xi)&=\xi^{b_1}V(\xi), &
 \bth(\xi)&=\xi^{c_1}\Theta(\xi),  & %\\
 \bs(\xi)&=\xi^{d_1}\Sigma(\xi), &
 \bu(\xi)&=\xi^{b_1+1}U(\xi) ,
\end{aligned}
\end{equation}
with $a_1, b_1, c_1$ and $d_1$ as in \eqref{eq:exponents}, in order to de-singularize the problem.
After some cumbersome, but straightforward calculation, we find that $(\bg,\bv,\bth,\bs)$ satisfies
% These variables result in a nice property;
\begin{equation} \label{eq:barsys}
 \begin{aligned}
  a_0\bg + \lambda\xi\bg' &=\bu,\\
  b_0\bv + \lambda\xi\bv' &=-d_1 \bs + \xi\bs',\\
  c_0\bth+ \lambda\xi\bth'&=\bs\bu,\\
  -b_1\bv+\xi\bv' &= \bu,
 \end{aligned}
\end{equation}
and $\bu$ is defined by
$$  \ts =\bth^{-\alpha}\bg^m\bu^n.$$

Next, introduce a new independent variable $\eta = \log\xi$ and define $(\tg,\tv,\tth,\ts,\tu)$ by
\begin{equation} \label{eq:BARtoTIL}
\begin{aligned}
 \tg(\log\xi)&=\bg(\xi), &
 \tv(\log\xi)&=\bv(\xi), &
 \tth(\log\xi)&=\bth(\xi), \\
 \ts(\log\xi)&=\bs(\xi), &
 \tu(\log\xi)&=\bu(\xi).
\end{aligned}
\end{equation}
Noticing that $\frac{d}{d\eta}\tg(\eta) = \xi \frac{d}{d\xi}\bg(\xi)$, we obtain an autonomous system
\begin{equation} \label{eq:tildesys}
 \begin{aligned}
  a_0\tg + \lambda\dtg &=\tu,\\
  b_0\tv + \lambda\dtv &=-d_1 \ts + \dts,\\
  c_0\tth+ \lambda\dtth &= \ts\tu, \\
  -b_1\tv+\dtv &= \tu,
 \end{aligned}
\end{equation}
where the notation $\dot{f}=\frac{df}{d\eta}$ is used, and $\tu$ is defined by
$$
\ts =\tth^{-\alpha}\tg^m\tu^n \, .
$$

%  \ts &=\tth^{-\alpha}\tg^m\tu^n.\\

The system \eqref{eq:tildesys} is autonomous and one might attempt to consider its equilibria.
However, it is easy to conclude that we cannot expect a heteroclinic that  tends to equilibria of \eqref{eq:tildesys}.
Indeed, suppose $\tu \rightarrow \tu_\infty\ge 0$ as $\eta \rightarrow \infty$.
Then from the last equation in \eqref{eq:tildesys},
we conclude that $\tv \rightarrow \infty$. % and thus also $\ts \rightarrow \infty$ by the second equation.
This suggests to enlarge the scope and consider solutions that grow as polynomials (or faster) at infinities.

\subsection{The $(p,q,r,s)$-system derivation}
Next, we attempt to come up with a new choice of variables that tend to equilibria as $\eta \rightarrow \pm \infty$ and accommodate
orbits that have power behavior at infinities. We rewrite \eqref{eq:tildesys} in the form
\begin{equation}
 \label{eq:tildesys2}
\begin{aligned}
\frac{d}{d\eta}{(\ln{\tg})}  &=  \tfrac{1}{\lambda} \big (- a_0 +  \frac{\tu}{\tg} \big ),
\\
 \frac{d}{d\eta}{(\ln{\tv})}  &=   b_1 + \frac{\tu}{\tv} ,
\\
\frac{d}{d\eta}{(\ln{\tth})} &=   \tfrac{1}{\lambda} \big (- c_0 +  \frac{\ts \tu}{\tth} \big ),
\\
\frac{d}{d\eta}{( \ln{\ts}  )} &= d_1 + b \frac{\tv}{\ts} +   \lambda \frac{\tu}{\ts}
\end{aligned}
\end{equation}
and view it as describing the evolution of $(\tg,\tv,\tth,\ts)$ with $\tu$ determined by $\tu = \left ( \frac{\ts}{ \tth^{-\alpha} \tg^m} \right )^\frac{1}{n}$.

This leads us to define
\begin{equation}\label{eq:pqrdef}
 \begin{aligned}
  p :=\frac{\tg}{\ts}, \quad q :=b \frac{\tv}{\ts},  \quad r = \left ( \frac{\ts}{ \tth^{-\alpha} \tg^{m+n}} \right )^\frac{1}{n} = \frac{\tu}{\tg}  , \quad s := \frac{\ts\tg}{\tth} \, .
 \end{aligned}
\end{equation}
The transformation $(p,q,r,s) \leftrightarrow (\tg,\tv, \tth,\ts)$ is a bijection in the positive orthant with the inverse determined by
$$
\tg = p^\frac{1+\alpha}{D} s^\frac{\alpha}{D} r^\frac{n}{D} \, ,  \quad \tth = p^\frac{1+m+n}{D} s^\frac{m+n-1}{D} r^\frac{2n}{D} \, ,
$$
and then
$$
\ts = \frac{1}{\tg} p \, , \quad \tv = \frac{1}{b} \ts \, q = \frac{1}{b} \frac{ p q}{\tg} \, .
$$
Using \eqref{eq:tildesys2} and \eqref{eq:pqrdef}, we write
\begin{align*}
 \frac{\dpp}{p}&=\frac{\dtg}{\tg} - \frac{\dts}{\ts}& &=\left[\frac{1}{\lambda }\Big(\frac{\tu}{\tg}-a_0\Big)\right] & &-\left[d_1 + b\frac{\tv}{\ts} +
 \lambda   \frac{\tu}{\tg} \frac{\tg}{\ts}  \right]
 \\
 \frac{\dqq}{q}&=\frac{\dtv}{\tv} - \frac{\dts}{\ts}& &=\left[b_1 +\frac{\tu}{\tv}\right] & &-\left[d_1 + b\frac{\tv}{\ts} + \lambda  \frac{\tu}{\tg} \frac{\tg}{\ts}    \right]
 \\
 n\frac{\drr}{r}&=-(m+n)\frac{\dtg}{\tg}+\frac{\dts}{\ts} + \alpha\frac{\dtth}{\tth} & &=\left[\frac{-(m+n)}{\lambda}\Big(\frac{\tu}{\tg}-a_0\Big)\right]& &+
 \left[d_1 + b\frac{\tv}{\ts} + \lambda  \frac{\tu}{\tg} \frac{\tg}{\ts}   \right] + \left[\frac{\alpha}{\lambda }\Big(\frac{\ts\tu}{\tth}-c_0\Big)\right]\\
 \frac{\dot{s}}{s} &= \frac{\dtg}{\tg} + \frac{\dts}{\ts} - \frac{\dtth}{\tth} & &=\left[\frac{1}{\lambda }\Big(\frac{\tu}{\tg}-a_0\Big)\right] & &+\left[d_1 + b\frac{\tv}{\ts}
 + \lambda  \frac{\tu}{\tg} \frac{\tg}{\ts}  \right] -\left[\frac{1}{\lambda }\Big(\frac{\ts\tu}{\tth}-c_0\Big)\right].%\\
%  \dot{s} &=\frac{\partial s}{\partial (z-1)} \dot{z} &&= \frac{1+m}{\lambda}\frac{\partial s}{\partial (z-1)} \bigg\{z\big[-r - \frac{n}{D}\Big]+ ru^n\bigg\}.
\end{align*}
We note that
\begin{align*}
 \frac{\ts\tu}{\tth} = rs, \quad \frac{\tu}{\tv} = \frac{bpr}{q}, \quad \frac{\tu}{\ts} = pr,
\end{align*}
and using \eqref{eq:exponents} and \eqref{defD}, after a cumbersome but straightforward calculation, we derive the $(p,q,r,s)$-system:
\begin{equation}\label{eq:slow} \tag{S}
 \begin{aligned}
 \dot{p} &=p\Big(\frac{1}{\lambda}(r-a) + 2- \lambda p r -q\Big),\\
 \dot{q} &=q\Big(1 -\lambda p r -q\Big) + b p r,\\
 n\dot{r} &=r\Big(\frac{\alpha-m-n}{\lambda(1+\alpha)}(r-a) + \lambda pr + q +\frac{\alpha}{\lambda}r\big(s- \frac{1+m+n}{1+\alpha}\big) + \frac{n\alpha}{\lambda(1+\alpha)}\Big),\\
 \dot{s} &=s\Big(\frac{\alpha-m-n}{\lambda(1+\alpha)}(r-a) + \lambda pr + q - \frac{1}{\lambda}r\big(s- \frac{1+m+n}{1+\alpha}\big) - \frac{n}{\lambda(1+\alpha)}\Big).
 \end{aligned}
\end{equation}
In the sequel, we analyze \eqref{eq:slow} as an autonomous system: We begin with sorting its equilibria and analyzing their linear stability. Most importantly, \eqref{eq:slow} possesses the {\it fast-slow} structure because of the small parameter $n$ in the left-hand-side of $\eqref{eq:slow}_3$; the dynamics of $r$ can be distinctively faster than those of the other variables off the nullcline of $r$.

\section{Equilibria and their linear stability} \label{sec:equil}

System \eqref{eq:slow} admits several equilibria listed in Appendix \ref{append:equi_reject}.
Our region of interest is the sector
$$
\mathcal{P} = \{(p,q,r,s) \; | \; p\ge0, q\ge0, r>0, s>0 \}.
$$
That $p,q\ge0$ comes from the requirement that $\tg,\tv,\ts\ge0$. The reason we restrict to $r > 0$, $s > 0$
stems from mechanical considerations: If we transform back to the original variables, then we find
\begin{equation*}
 r(\eta)|_{\eta=t^\lambda x}=t\partial_t\log \gamma(t,x), \quad r(\eta)s(\eta)|_{\eta=t^\lambda x}=t\partial_t \log \theta(t,x).
\end{equation*}
%for the original variables $\theta(t,x)$ and $\gamma(t,x)$.
Shear band initiation is related to conditions of loading where both the plastic strain and the temperature are increasing. This motivates to restrict to self-similar solutions
taking values in the region $r > 0$, $s > 0$.

From the complete set of equilibria for \eqref{eq:slow} listed in Appendix \ref{append:equi_reject} only two reside in
the region $r > 0$, $s >0$, namely
\begin{align*}
 M_0 &= (0,0,r_0,s_0), & r_0 & =a , & s_0&=\frac{1+m+n}{1+\alpha} - \frac{n}{(1+\alpha)r_0},\\
 M_1 &= (0,1,r_1,s_1), & r_1 & = a -\frac{1+\alpha}{\alpha-m-n}\lambda, & s_1&=\frac{1+m+n}{1+\alpha} - \frac{n}{(1+\alpha)r_1} \, .
\end{align*}
\tcb{
Here, we recall \eqref{eq:exponents}, \eqref{defD}:
$$
\begin{aligned}
a = a_0 + a_1 \lambda &= \frac{2+2\alpha-n}{D} + \frac{2+2\alpha}{D}\lambda \, \\
D &= 1+2\alpha-m-n  \, ,
\end{aligned}
$$
and that the parameters $(\alpha, m, n)$ take values in the range \eqref{eq:paramrange}. As a consequence
$r_0>0$, and a simple calculation shows that  $r_0s_0 = \frac{2(1+m)}{D} + \frac{2(1+m+n)}{D}\lambda>0$; hence,
$r_0, s_0>0$ and $M_0$  resides in the region $\mathcal{P}$.
}
By contrast, $M_1$ can be out of the region $r>0$, $s>0$ if $\lambda$ is large enough. Note that $r_1,s_1>0$ only if $\frac{1+m+n}{1+\alpha}r_1 > \frac{n}{(1+\alpha)}$. This reads
$$\frac{1+m+n}{1+\alpha}\Big(\frac{2+2\alpha-n}{D} - \frac{(1+\alpha)(1+m+n)}{D(\alpha-m-n)}\lambda\Big) > \frac{n}{(1+\alpha)},$$
% $\lambda$ is not arbitrarily large but in the range %in this region only under the constraint that
and thus $M_1$ resides in the region $\mathcal{P}$ only under the constraint
\begin{equation} \label{eq:lambda-range}
 0< \lambda < \frac{2(\alpha-m-n)}{1+m+n}\left(\frac{1+m}{1+m+n}\right).
\end{equation}
Henceforth, we restrict attention to rates $\lambda$ satisfying \eqref{eq:lambda-range}.

\begin{figure}
 \centering
  \psfrag{x0}{\scriptsize $M_0$}
  \psfrag{x1}{\scriptsize $M_1$}
  \psfrag{x2}{~~\scriptsize $1$}
  \psfrag{x3}{}
  \psfrag{p}{\scriptsize $p$}%=\frac{\gamma}{\sigma}$}
  \psfrag{q}{\scriptsize~~~$q$}%=n\frac{v}{\sigma}$}
  \psfrag{q*}{}%=\frac{2-n}{\lambda}$}
  \psfrag{r*0}{}%\hskip -15pt$r_0=1+\frac{2\lambda}{2-n}$}
  \psfrag{r*1}{}% \hskip -35pt$r_1=1-\frac{n\lambda}{(2-n)(1-n)}$}
  \psfrag{r*2}{}
  \subfigure[$pqr$-space]{
  \psfrag{r}{\scriptsize$r$}%=\big(\sigma\gamma^{(1-n)}\big)^{\frac{1}{n}}$}
  \includegraphics[width=6cm]{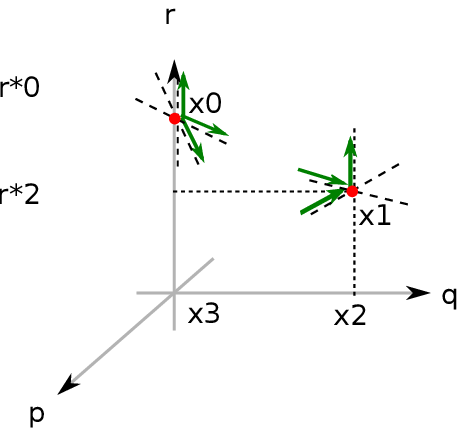}\label{fig:eq1}
  }
  \quad \quad
  \subfigure[$pqs$-space]{
  \psfrag{r}{\scriptsize$s-\frac{1+m}{1+\alpha}$}%=\big(\sigma\gamma^{(1-n)}\big)^{\frac{1}{n}}$}
  \includegraphics[width=6cm]{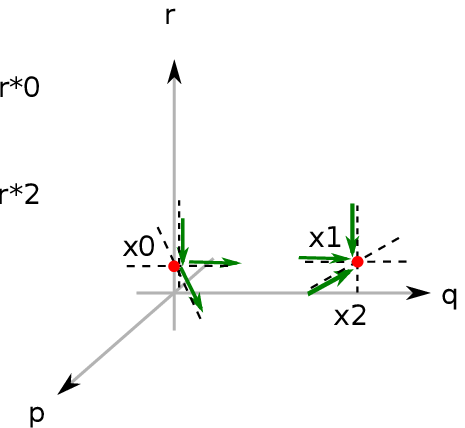}\label{fig:eq2}
  }
  \caption{Eigenvectors around $M_0$ and $M_1$ in $pqr$-space and in $pqs$-space respectively ($\mu_{11}\ne-1$ and $n\ll1$). } \label{fig:equilibria}
\end{figure}

We denote the four eigenvalues and four eigenvectors of the vector field linearized at $M_i$, $i=0,1$,  by $\mu_{ij}$ and $X_{ij}$ with $j=1,2,3,4$.
\begin{itemize}
 \item $M_0$ is a saddle; the matrix of the linearized vector field at $M_0$ has three positive eigenvalues and one negative eigenvalue.
 \begin{equation} \label{eq:eigM0}
  \mu_{01} = 2, \quad \mu_{02}=1, \quad \mu_{03}=\mu_0^+=\BO\Big(\frac{1}{n}\Big)>0, \quad \mu_{04}=\mu_0^{-}<0,
 \end{equation}
where $\mu_0^\pm$ are respectively a positive and a negative solution of the quadratic equation
%$$ \mu^2 - \mu\Big(\frac{r_0(1-s_0)}{n\lambda}-\frac{r_0s_0+1}{\lambda}\Big) - \frac{r_0^2s_0(\alpha-m-n)}{n\lambda^2}=0.$$
 $$
 \Big(\mu - \frac{r_0}{n}\Big(\frac{1-s_0}{\lambda}-\frac{n}{\lambda r_0}\Big)\Big)\Big(\mu + \frac{s_0r_0}{\lambda}\Big) -
 \frac{s_0r_0(1-s_0)(\alpha r_0)}{n\lambda^2} = 0.
 $$
The leading orders of $\mu_0^\pm$ are given by
% $$\mu_0^+ = \frac{\alpha-m}{n\lambda(1+\alpha)}\frac{2(1+\alpha)(1+\lambda)}{(1+2\alpha-m)}+\BO(1), \quad\mu_0^- = -\frac{1+m}{\lambda}\frac{2(1+\alpha)(1+\lambda)}{(1+2\alpha-m)}  + \BO(n).$$
$$\
\mu_0^+ = \frac{1}{n} \tfrac{2(\alpha-m)(1+\alpha)(1+\lambda)}{\lambda(1+\alpha)(1+2\alpha-m)}+\BO(1), \quad
\mu_0^- = -\tfrac{2(1+m)(1+\alpha)(1+\lambda)}{\lambda(1+2\alpha-m)}  + \BO(n).$$
Notice that one of the positive eigenvalue $\mu_{03}$ is $\BO( \frac{1}{n})$, which indicates the separably fast dynamics along the direction $X_{03}$. We will make use of this structure later.
% $$\mu_0^+ = \frac{r_0(1-s_0)}{n\lambda} + \frac{1}{\lambda}\frac{{r_0^2s_0}(\alpha-m-n)}{ {r_0(1-s_0)}-n(1+r_0s_0) } + \BO(n), \quad \mu_0^- = -\frac{1}{\lambda}\frac{{r_0^2s_0}(\alpha-m-n)}{ {r_0(1-s_0)}-n(1+r_0s_0) } + \BO(n).$$
The precise eigenvector components are presented in Appendix \ref{append:lin}, the directions of the eigenvectors are
pointed out in Fig. \ref{fig:equilibria} for $n$ sufficiently small.
 \item $M_1$ is a saddle;  the matrix of the linearized vector field at $M_1$ has one positive eigenvalue and three negative eigenvalues.
\begin{equation} \label{eq:eigM1}
 \mu_{11}=-\frac{1+m+n}{\alpha-m-n}, \quad \mu_{12}=-1, \quad \mu_{13}=\mu_1^+=\BO\Big(\frac{1}{n}\Big)>0, \quad \mu_{14}=\mu_1^{-}<0,
\end{equation}
where $\mu_1^\pm$ is respectively a positive and a negative solution of the quadratic equation
 $$
 \Big(\mu - \frac{r_1}{n}\Big(\frac{1-s_1}{\lambda}-\frac{n}{\lambda r_1}\Big)\Big)\Big(\mu + \frac{s_1r_1}{\lambda}\Big) - %\frac{s_1r_1}{n} \frac{1-s_1}{\lambda}\frac{\alpha r_1}{\lambda} = 0$$
\frac{s_1r_1(1-s_1)(\alpha r_1)}{n\lambda^2} = 0.
$$
The leading orders of $\mu_1^\pm$ are given by
\begin{align*}
\mu_1^+ &= \frac{\alpha-m}{n\lambda(1+\alpha)}\Big(\frac{2(1+\alpha)(1+\lambda)}{(1+2\alpha-m) } - \frac{1+\alpha}{\alpha-m-n}\lambda\Big) + \BO(1), \\
\mu_1^- &= -\frac{1+m}{\lambda}\Big(\frac{2(1+\alpha)(1+\lambda)}{(1+2\alpha-m) } - \frac{1+\alpha}{\alpha-m-n}\lambda\Big) + \BO(n).
\end{align*}
Note that the positive eigenvalue $\mu_{13}$ is $\BO( \frac{1}{n})$.
In constrast to what happens at $M_0$, the eigenvalues of the linearized vector field at $M_1$ may have multiplicity higher than one. Appendix \ref{append:lin} describes  the possible cases and provides the generalized eigenvectors when necessary.
\end{itemize}

\section{Characterization of the heteroclinic orbit} \label{sec:char}
The equilibrium $M_0$ has a three dimensional unstable manifold and a one dimensional stable manifold while the equilibrium $M_1$ has a three dimensional
stable manifold and a one-dimensional unstable manifold. There is one unstable direction for each equilibrium corresponding to a positive
eigenvalue of order $\BO (\frac{1}{n})$.  Due to the high dimensionality, it is difficult to read the complete behavior
of the flow in phase space. This section aims to develop a picture of the flow on the positive orthant $p,q,r,s > 0$ and
to associate the behavior  of the system \eqref{eq:ss-odes} near the singular point $\xi = 0$ with the behavior of the system \eqref{eq:slow} around $M_0$.

\subsection{Behavior near the singular point $\xi =0$}
We begin with the latter point. The following proposition states how \eqref{eq:ic0}, \eqref{bc1} are transmitted to the asymptotic behavior of $(p,q,r,s)$
around the equilibrium $M_0$ as $\eta \to -\infty$.

\tcb{
\begin{proposition} \label{prop1}
    Let $\big(V, \Sigma, \Theta, \Gamma \big)(\xi)$ be a smooth solution of \eqref{eq:ss-odes},  $U (\xi)$ be defined by \eqref{eqnsigma},
    and suppose the solution is defined for $\xi > 0$, is smooth,  takes values in the positive orthant, and assumes the initial conditions
    $$
    V(0) = V_0 \ge 0, \quad \Sigma(0) = \Sigma_0 > 0 \, , \quad \Theta (0) = \Theta_0 > 0 \, \quad \Gamma_0 = \Gamma_0 > 0 \quad \mbox{and} \;\; U'(0) =0 \, .
    $$
    Then $\big(V, \Sigma, \Theta, \Gamma \big)$ and $U$ satisfy at $\xi =0$ the  conditions \eqref{compatibility}, \eqref{eq:bdry0}, \eqref{eq:ic}. Morever,  the orbit defined by the transformations \eqref{eq:CAPtoBAR},
    \eqref{eq:BARtoTIL},  \eqref{eq:pqrdef}, $\chi(\eta) = (p(\eta), q(\eta), r(\eta),s(\eta)) \rightarrow M_0$ as $\eta \rightarrow -\infty$.
    Furthermore, it tends to $M_0$ along the direction of the first eigenvector $X_{01}$, in fact
    \begin{equation} \label{eq:alpha}
     e^{-2\eta}\big(\chi(\eta) - M_0 \big) \rightarrow \kappa X_{01}, \quad \text{for some constant  $\kappa>0$ as $\eta \rightarrow -\infty$.}
    \end{equation}
\end{proposition}
}

\begin{remark} \label{rem:alpha}
The orbit approaches $M_0$ tangent to $X_{01}$ as $\eta \rightarrow -\infty$. Since $M_0$ has a three-dimensional unstable manifold and $\mu_{02}(=1)<\mu_{01}(=2)<\mu_{03}(=\BO(\frac{1}{n}))$, %The unstable manifolds of one, two, and three dimensions are separable from the fastest asymptotic rate.
  the orbits emanating from $M_0$ tangent to $X_{01}$ all lie on a two dimensional manifold that at $M_0$ is tangent to the plane spanned by $X_{01}$ and $X_{03}$. This two dimensional submanifold will be referred to as the Strongly unstable manifold of $M_0$.
\end{remark}

\begin{proof}
Assuming smoothness and boundedness of $\big(V, \Sigma, \Theta, \Gamma \big)$ and $U$ in a neighborhood of $\xi=0$ and the conditions \eqref{bc1}, \eqref{eq:ic0} we deduce first \eqref{compatibility}, \eqref{eq:bdry0}, \eqref{eq:ic} by the argument of
section \ref{sec:scale}. The derivatives of $\big(\Gamma,V,\Theta,\Sigma,U\big)$ at $\xi=0$ are obtained by differentiating the system \eqref{eq:ss-odes} repeatedly.
Re-write \eqref{eq:ss-odes} as
\begin{align*}
  a + \lambda\xi\frac{\Gamma'}{\Gamma} &= \frac{U}{\Gamma}, &
  c + \lambda\xi\frac{\Theta'}{\Theta} &= \frac{\Sigma\Gamma}{\Theta} \frac{U}{\Gamma},\\
  (b+\lambda)U  + \lambda \xi U'(\xi) &= \Sigma^{''} = \Big(\frac{\Sigma\Gamma}{\Theta} \frac{\Theta}{\Gamma}\Big)^{''}, &
  \frac{\Big(\frac{\Sigma\Gamma}{\Theta}\Big)^{''}}{\frac{\Sigma\Gamma}{\Theta}} &= (1+m+n)\frac{\Gamma^{''}}{\Gamma}-(1+\alpha) \frac{\Theta^{''}}{\Theta} + n \frac{ \big(\frac{U}{\Gamma}\big)^{''}}{\frac{U}{\Gamma}} \, ,
\end{align*}
from where after a computation we conclude
\begin{align*}
&\frac{U}{\Gamma}(0) = a = r_0,  & \Big(\frac{U}{\Gamma}\Big)'(0)&=0, & \Big(\frac{U}{\Gamma}\Big)^{''}(0) &= \frac{\Gamma(0)}{\Sigma(0)} \frac{-2(b+\lambda)r_0}{\frac{1-s_0}{\lambda}-\frac{n}{r_0}\Big(\frac{2}{s_0} + \frac{r_0}{\lambda}\Big)\left(\frac{ \frac{1}{\lambda}+2}{ \frac{1+\alpha}{\lambda}r_0 + \frac{2}{s_0}}\right)},\\
&\frac{\Sigma\Gamma}{\Theta}(0) = \frac{c}{a} = s_0,  & \Big(\frac{\Sigma\Gamma}{\Theta}\Big)'(0)&=0, &
\Big(\frac{\Sigma\Gamma}{\Theta}\Big)^{''}(0) &= \frac{n}{r_0} \left(\frac{ \frac{1}{\lambda}+2 }{ \frac{1+\alpha}{\lambda}r_0 + \frac{2}{s_0}}\right)\Big(\frac{U}{\Gamma}\Big)^{''}(0).
\end{align*}
The Taylor expansions of $p(\log\xi)$, $q(\log\xi)$, $r(\log\xi)$ and $s(\log\xi)$ at $\xi=0$ are computed using \eqref{eq:CAPtoBAR}, \eqref{eq:pqrdef}, \eqref{eq:bdry0}
and the relations above,
\begin{align*}
 p(\log\xi) &= \frac{ \bg }{\bs}  = \frac{ \xi^{a_1} \Gamma(\xi)}{\xi^{d_1} \Sigma(\xi)} = \xi^2\frac{\Gamma(\xi)}{\Sigma(\xi)} = \xi^2\frac{\Gamma(0)}{\Sigma(0)} + o(\xi^2) \, , \\
 %&= \xi^2\Big(\frac{U(0)}{\Phi(0)}\Big)^{-n}\Phi(0)^{1+\frac{\alpha-n}{1+\alpha}} + o(\xi^2),\\
 q(\log\xi) &= b\frac{\bv}{\bs} = b\frac{ \xi^{b_1} V(\xi) }{ \xi^{d_1} \Sigma(\xi)} = b\xi\frac{ V(\xi) }{ \Sigma(\xi)} = b\xi^2 \frac{U(0)}{\Sigma(0)}+ o(\xi^2)=\xi^2 ~br_0\frac{\Gamma(0)}{\Sigma(0)} + o(\xi^2) \, ,\\
 %&= \xi^2\Big(b\frac{U(0)}{\Phi(0)}\Big)\Big(\frac{U(0)}{\Phi(0)}\Big)^{-n}\Phi(0)^{1+\frac{\alpha-n}{1+\alpha}} + o(\xi^2),\\
 r(\log\xi) &= \frac{\bu}{ \bg } = \frac{ \xi^{1+b_1}U(\xi) }{ \xi^{a_1}\Gamma(\xi) } = \frac{ U(0) }{ \Gamma(0) }+ \xi \Big(\frac{U}{\Gamma}\Big)'(0) + \frac{1}{2}\xi^2\Big(\frac{U}{\Gamma}\Big)^{''}(0) + o(\xi^2)\\
  &=\frac{ U }{ \Gamma }(0) + \xi^2\frac{\Gamma(0)}{\Sigma(0)} \frac{-(b+\lambda)r_0}{\frac{1-s_0}{\lambda}-\frac{n}{r_0}\Big(\frac{2}{s_0} + \frac{r_0}{\lambda}\Big)\left(\frac{ \frac{1}{\lambda}+2}{ \frac{1+\alpha}{\lambda}r_0 + \frac{2}{s_0}}\right)} ,
  \\
 s(\log\xi) &= \frac{\bs\bg}{\bth} = \frac{ \xi^{a_1+d_1}\Sigma(\xi)\Gamma(\xi) }{\xi^{c_1} \Theta(\xi)} = \frac{ \Sigma\Gamma }{\Theta}(0) + \xi \Big(\frac{ \Sigma\Gamma }{\Theta}\Big)^{'}(0) + \frac{1}{2}\xi^2\Big(\frac{ \Sigma\Gamma }{\Theta}\Big)^{''}(0) + o(\xi^2)\\
 &=\frac{ \Sigma\Gamma }{\Theta}(0) + \xi^2 n \left(\frac{ \big(\frac{1}{\lambda}+2\big) \frac{1}{r_0} }{ \frac{1+\alpha}{\lambda}r_0 + \frac{2}{s_0}}\right)\frac{\Gamma(0)}{\Sigma(0)} \frac{-(b+\lambda)r_0}{\frac{1-s_0}{\lambda}-\frac{n}{r_0}\Big(\frac{2}{s_0} + \frac{r_0}{\lambda}\Big)\left(\frac{ \frac{1}{\lambda}+2}{ \frac{1+\alpha}{\lambda}r_0 + \frac{2}{s_0}}\right)}+ o(\xi^2).
\end{align*}
Therefore, taking note of the eigenvector $X_{01}$ in Appendix \ref{append:lin}, we conclude
\begin{align*}
\chi(\log\xi)-M_0  = \big(p(\log\xi),q(\log\xi),r(\log\xi),s(\log\xi)\big) -M_0 =  \frac{\Gamma(0)}{\Sigma(0)}\xi^2 X_{01} + o(\xi^2),
\end{align*}
which implies \eqref{eq:alpha} since $\eta=\log\xi \to -\infty$ as $\xi \to 0$.
\end{proof}
\begin{remark} \label{rem:signs}
For $n$ small enough, we find that second derivatives have definite signs: $\displaystyle \Big(\frac{U}{\Gamma}\Big)^{''}(0) <0$, $\displaystyle \Big(\frac{\Sigma\Gamma}{\Theta}\Big)^{''}(0) <0$, and
\begin{equation} \label{eq:second_der}
\begin{aligned}
\frac{\Gamma^{''}(0)}{\Gamma(0)} &= \frac{1}{2\lambda}\Big(\frac{U}{\Gamma}\Big)^{''}(0) < 0, &
\frac{\Theta^{''}(0)}{\Theta(0)} &= \frac{1}{2\lambda}\Big(s_0\Big(\frac{U}{\Gamma}\Big)^{''}(0) + r_0\Big(\frac{\Sigma\Gamma}{\Theta}\Big)^{''}(0)\Big)\Big)  < 0,\\
\frac{U^{''}(0)}{U(0)} &=\frac{\Gamma^{''}(0)}{\Gamma(0)} + \frac{ \big(\frac{U}{\Gamma}\big)^{''}(0)}{\frac{U}{\Gamma}(0)}< 0,&
\Sigma^{''}(0)&=(b+\lambda)U(0)>0.
\end{aligned}
\end{equation}
\end{remark}

\subsection{A heteroclinic orbit}
The behavior of \eqref{eq:ss-odes} near the singular point  $\xi =0$ suggests to look for an orbit of \eqref{eq:slow} emanating from $M_0$ in the direction of the Strongly unstable manifold.
At the other end point, as $\eta \rightarrow \infty$, we expect that the variables $p$, $q$, $r$ and $s$
equilibrate to a bounded state as $\eta \rightarrow \infty$.
This would imply that the variables $\tg$, $\tv$, $\tth$, and $\ts$ grow at most exponentially in $\eta$ as seen from \eqref{eq:tildesys2} and in turn polynomially in $\xi$.
Therefore, we look for a heteroclinic orbit joining $M_0$ with $M_1$, the only two equilibria in the sector $r>0$ and $s>0$.
That is we expect $\chi(\eta) \rightarrow M_1$ as $\eta \rightarrow \infty$.
The heteroclinic orbit will be called $\chi(\eta)$ for the rest of this paper. It will be constructed in the next section.

The two end point behaviors can be interpreted geometrically as well: The end point behavior as $\eta \rightarrow -\infty$ specifies a nontrivial submanifold of the unstable manifold
of the equilibrium  $M_0$ from which the  orbit emanates. This submanifold will turn out to intersect the stable manifold of $M_1$ and the intersection of these two manifold is the heteroclinic orbit we search for.

$U_0$, $\Gamma_0$ are selected.

Given the heteroclinic $\chi(\eta)$, this subsection is devoted to adapting the initial data. By  data we refer to
$\big(\Gamma_0,\Theta_0,\Sigma_0,U_0 \big)$.

\subsection{Adapting the initial data } \label{sec:twoparam}
Suppose now that an orbit $\chi(\eta)$ has been constructed that satisfies \eqref{eq:slow}, it emanates from $M_0$ in the Strongly unstable manifold, i.e. it satisfies \eqref{eq:alpha},
and connects to $M_1$. The orbit corresponds to a set of parameters $(\lambda, \alpha, m,n)$. We proceed to show how the data input $\big(V_0,\Gamma_0,\Theta_0,\Sigma_0,U_0 \big)$ fit under the supposed orbit $\chi^{\lambda, \alpha, m,n}(\eta)$. From \eqref{eq:ic}, \eqref{compatibility} we know that $V_0=0$ and
$$\Theta_0 = c^{-\frac{1}{1+\alpha}}\Gamma_0^{\frac{m}{1+\alpha}} U_0^{\frac{1+n}{1+\alpha}}, \quad \Sigma_0 = c^{\frac{\alpha}{1+\alpha}}\Gamma_0^{\frac{m}{1+\alpha}} U_0^{-\frac{\alpha-n}{1+\alpha}}.$$
By \eqref{rate}, the rate of growth $\lambda$ determines the ratio $\frac{U_0}{\Gamma_0}$
\begin{equation} \label{eq:lambda}
 \lambda = \Big(\frac{U_0}{\Gamma_0} - \frac{2(1+\alpha)-n}{D}\Big)\frac{D}{2(1+\alpha)}.
\end{equation}
Note finally that the restriction \eqref{eq:lambda-range} on the growth rate $\lambda$ implies a restriction on the ratio
\begin{equation} \label{eq:restriction}
\begin{aligned}
 \frac{2(1+\alpha) -n}{D} < \frac{U_0}{\Gamma_0} &< \frac{2(1+\alpha) -n}{D} + \frac{4(1+\alpha)(\alpha-m-n)(1+m)}{D(1+m+n)^2}\\
 &=\frac{2(1+\alpha)}{1+m+n} -\frac{n}{D}\left( \frac{4(1+\alpha)(\alpha-m-n)}{(1+m+n)^2} +1\right).
\end{aligned}
\end{equation}

It remains to resolve only one degree of freedom. The orbit $\chi (\eta)$ emanating from $M_0$ in the direction $X_{01}$ satisfies the asymptotic expansion
 \begin{equation}\label{eq:alpha-expan}
  \chi (\eta) - M_0 = \kappa_1 e^{\mu_{01}\eta} X_{01} + \kappa_3 e^{\mu_{03}\eta} X_{03} + \text{higher-order terms as $\eta \rightarrow -\infty$},
 \end{equation}
for two constants $\kappa_1$ and $\kappa_3$.
Any reparametrization $\chi(\eta-\eta_0)$, $\eta_0\in \mathbb{R}$, depicts the same heteroclinic orbit. Define
$\bar \chi (\eta) = \chi (\eta - \eta_0)$ and we proceed to select $\eta_0$ so as to satisfy the data. Then,
$$
\lim_{\eta \rightarrow -\infty}\big( \bar\chi(\eta) - M_0\big)e^{-2 \eta}
= \lim_{\eta \rightarrow -\infty}  \Big ( \big(\chi(\eta-\eta_0) - M_0\big)e^{-2(\eta-\eta_0)} \Big ) e^{-2\eta_0} = e^{ -2\eta_0} \kappa_1 X_{01}.
$$
On the other hand, from the proof of Proposition \ref{prop1},
$$
\lim_{\eta \rightarrow -\infty}\big( \bar\chi(\eta) - M_0\big)e^{ -2\eta}  = \frac{\Gamma_0}{\Sigma_0} X_{01} \, ,
$$
which dictates we fix the last degree of freedom by selecting
\begin{equation}
 \eta_0 = \frac{1}{2}\log \left( \frac{\Gamma_0}{\Sigma_0} \kappa_1\right).%\log \sqrt{\frac{\bar\kappa_1}{\frac{\Gamma(0)}{\Sigma(0)}}}.
\end{equation}

\section{Existence via Geometric theory of singular perturbations} \label{sec:proof}
This section is devoted to proving the existence of a heteroclinic orbit $\chi(\eta)$ with limiting behavior as determined in  section \ref{sec:char}.

\begin{theorem} \label{thm1}
Let $(\alpha,m,n)$ take values in the range \eqref{eq:paramrange}.
Given $\lambda > 0$ satisfying \eqref{eq:lambda-range} with $n=0$, there is
$n_0( \lambda,\alpha,m)$ such that for $n \in [0, n_0)$ and  $\lambda$ satisfying \eqref{eq:lambda-range} the
system \eqref{eq:slow}  admits a heteroclinic orbit $\chi^{\lambda,\alpha,m,n}(\eta)$ joining the equilibrium $M_0^{\lambda,\alpha,m,n}$ to the equilibrium $M_1^{\lambda,\alpha,m,n}$ and satisfying the property
    \begin{align}
     \label{eq:rapid}
%         &\chi(\eta) \rightarrow M_1 \quad \text{as $\eta \rightarrow \infty$ and} \\
        e^{-2\eta}\big(\chi^{\lambda,\alpha,m,n}(\eta) - M_0^{\lambda,\alpha,m,n}\big) \rightarrow \kappa X_{01}^{\lambda,\alpha,m,n} \quad \text{as $\eta \rightarrow -\infty$ for some $\kappa\ne0$}.
    \end{align}
\end{theorem}

The heteroclinic orbit $\chi^{\lambda,\alpha,m,n}(\eta)$ is achieved by applying the {\it geometric singular perturbation theory}. The presence of the small parameter $n>0$ in the left-hand-side of \eqref{eq:slow}$_3$ provides a {\it fast-slow} structure to the system, having $r$ as a fast variable and the rest as slow variables.
In the interest of the reader, we present some preliminary information. Experts on geometric singular
perturbation theory may wish to proceed directly to Sections \ref{sec:singorb}, \ref{sec:thmproof}.

Recall that \eqref{eq:slow} accounts for a family of dynamical systems parametrized by $(\lambda,\alpha,m,n)$; the heteroclinic orbit will be achieved respectively for each admissible $(\lambda,\alpha,m,n)$. To simplify notations we suppress the dependence on $\lambda$, $\alpha$, and $m$ but retain
the dependence on $n$.

\subsection{Invariant manifold theory and geometric singular perturbation theory}\label{sec:singpert}
We state here some rudiments of the geometric singular perturbation theory from \cite{fenichel_asymptotic_1977,fenichel_geometric_1979}. %Slightly different definitions are found in \cite{HPS_1977}. %
Fenichel's persistence theorem is developed in \cite[Theorem 9.1]{fenichel_geometric_1979}. In the present application the versions \cite[Theorem 2.2]{Sz1991} and \cite[Theorem 3.1]{Sz1991} are applied.

Let $X$ be a $C^{r}$ vector field in $\mathbb{R}^d$ with $r\ge2$ and let $\bar{\Lambda}=\Lambda \cup \partial \Lambda$ be a compact, connected $C^{r+1}$ manifold in $\mathbb{R}^d$. $F^t: \mathbb{R}^d \mapsto \mathbb{R}^d$ denotes the time $t$-map associated with the vector field $X$ and $DF^t$ denotes its differential.
$\bar{\Lambda}$ is said to be {\it overflowing invariant} under $X$ if for every $m\in\bar{\Lambda}$ and $t\le0$, $F^t(m)\in \bar{\Lambda}$ and $X$ is pointing strictly outward on $\partial \Lambda$. $T \mathbb{R}^d|\bar\Lambda$ denotes the tangent bundle of $\mathbb{R}^d$ along $\bar\Lambda$ and $T\bar\Lambda$ denotes the tangent bundle of $\bar\Lambda$. A subbundle $E\subset T\mathbb{R}^d|\bar{\Lambda}$ is said to be negatively invariant if $E \supset DF^t(E)$ for all $t\le0$.

Let $E\subset T\mathbb{R}^d|\bar{\Lambda}$ be a subbundle that is negatively invariant and contains $T\bar\Lambda$. Given such $E$,  $T \mathbb{R}^d| \bar\Lambda$ then splits into $T\mathbb{R}^d|\bar{\Lambda} =E\oplus E'= T\bar\Lambda\oplus N\oplus E'$, where $N\subset E$ is any complement of $T\bar\Lambda$ in $E$ and $E'\subset T\mathbb{R}^d|\bar{\Lambda}$ is any complement of $E$ in $T\mathbb{R}^d|\bar{\Lambda}$.
\tcb{ Next, the subundles are distinguished according to the growth or decay rates of the linearized flow as $t \to -\infty$,
%
%When a family of overflowing invariant manifold exits, the following  especially along the one sided limit as $t \rightarrow -\infty$.
%With those introduced, we give a sufficient condition for the separation of three subbundles {\blue distinguished by rates of growth or decay as $t \rightarrow -\infty$.} We check whether as $t \rightarrow -\infty$ $(i)$ the norm of a vector picked from $E_m'$ at a point $m$ diverges; $(ii)$ the norm of a vector picked from $N_m$ vanishes; $(iii)$ the growth or decay rate of a vector picked from $T_m\bar\Lambda$ is dominated above and below by the rate respectively of the two former asymptotics.
following  \cite{fenichel_geometric_1979}:}
Let $m\in \bar{\Lambda}$ and $v^0 \in T_m \Lambda$; $w^0\in N_m$; $x^0\in E'_m$; $v^t = DF^t(m)v^0$; $w^t = \pi^N DF^t(m)w^0$; $x^t = \pi^{E'}DF^t(m)x^0$,
% \begin{align*}
%  v^0 &\in T_m M, \quad w^0\in N_m, \quad x^0\in E'_m,\\
%  v^t &= DF^t(m)v^0, \quad w^t = \pi^N DF^t(m)w^0, \quad x^t = \pi^{E'}DF^t(m)x^0, \quad \text{where $\pi^N$ and $\pi^{E'}$ are bundle projections.}
% \end{align*}
where $\pi^N$ and $\pi^{E'}$ are bundle projections onto $N$ and $E'$ respectively. Define
\begin{align*}
 \nu^s(m) &\triangleq \inf \Big\{\nu>0 \: : \: \frac{1}{|x^{-t}|} = o(\nu^t) \quad \text{as $t \rightarrow \infty$} \quad \forall x^0\in E'_m\Big\}.
 \end{align*}
If  $\nu^s(m)<1$, define
\begin{align*}
 \sigma^s(m) &\triangleq \inf \Big\{\sigma>0 \: : \: |v^{-t}| = o(|x^{-t}|^\sigma) \quad \text{as $t \rightarrow \infty$} \quad \forall x^0\in E'_m, v^0\in T_m\Lambda\Big\}.
 \end{align*}
 Next, define
 \begin{align*}
 \alpha^u(m) &\triangleq \inf \Big\{\alpha>0 \: : \: |w^{-t}| = o(\alpha^t) \quad \text{as $t \rightarrow \infty$}\quad \forall w^0\in N_m\Big\}.
 \end{align*}
If $\alpha^u(m)<1$, define
\begin{align*}
 \rho^u(m) &\triangleq \inf \Big\{\rho>0 \: : \: \frac{|w^{-t}|}{|v^{-t}|} = o(\rho^t) \quad \text{as $t \rightarrow \infty$} \quad \forall w^0\in N_m, v^0\in T_m\Lambda\Big\}.
 \end{align*}
If $\rho^u(m)<1$, define
\begin{align*}
 \tau^u(m) &\triangleq \inf \Big\{\tau>0 \: : \: |\hat{v}^{-t}| = o\left(\Big(\frac{|v^{-t}|}{|w^{-t}|}\Big)^{\tau}\right) \quad \text{as $t \rightarrow \infty$} \quad \forall w^0\in N_m, v^0\in T_m\Lambda,\hat{v}^0\in T_m\Lambda\Big\}.
\end{align*}

\begin{definition} \label{def:over}
Let $\bar{\Lambda}=\Lambda \cup \partial\Lambda$  be an overflowing invariant manifold {\blue such that $T \mathbb{R}^d|\bar\Lambda$} admits a splitting by $E$ as described  above. We say an overflowing invariant manifold $\bar{\Lambda}$ satisfies assumptions \eqref{eq:A} and \eqref{eq:B}, $r'\le r-1$ if for all $m\in \bar{\Lambda}$ the  growth rates hold
\begin{align}
\nu^s(m)&<1, \quad \sigma^s(m)<\frac{1}{r}, \label{eq:A}\tag{$A_r$}   \\
\alpha^u(m)&<1, \quad \rho^u(m)<1, \quad \tau^u(m)<\frac{1}{r'}. \label{eq:B}\tag{$B_{r'}$}
%\quad \text{for all $m\in \bar{\Lambda}$}.
\end{align}
\end{definition}
\begin{remark}
\tcb{
Given the bundle splitting, the conditions  \eqref{eq:A} and \eqref{eq:B} suffice to construct the unstable manifold of $\bar\Lambda$
as well as the finer foliation structure within it; see \cite[Theorem 4]{fenichel_asymptotic_1974} and \cite[Theorem 3]{fenichel_asymptotic_1977}. Moreover, for the special case $E=T\bar\Lambda$, Fenichel  in \cite{fenichel_persistence_1972}  proved the persistence of the overflowing manifold $\bar\Lambda$ of $X_0$  when only \eqref{eq:A} is assumed.
%(While \cite[pp. 84--85]{fenichel_asymptotic_1977} checks different set of rate numbers, \eqref{eq:A} and \eqref{eq:B} suffice.)

Next, suppose that a $C^r$ family of vector fields $X_\epsilon$ is given depending on a small parameter for $\epsilon\in[-\epsilon_0,\epsilon_0]$ is given.
If $\bar\Lambda_\epsilon$ exists as a $C^r$ family of overflowing invariant manifolds for each sufficiently small $\epsilon$,
% a continuity argument shows the satisfactions of \eqref{eq:A} and \eqref{eq:B} for $\epsilon$ sufficiently small and
then the unstable manifold and its foliation structure are persistent in an appropriate sense; see section 16 of \cite{fenichel_geometric_1979}, and the discussion in \cite[p. 90]{fenichel_geometric_1979}.
}
\end{remark}

Next, we introduce the notion of normal hyperbolicity for manifolds without boundary,  \cite[p.89]{fenichel_geometric_1979} and \cite[p.221]{fenichel_persistence_1972}.
\begin{definition}[Normally Hyperbolic Invariant Manifold \cite{fenichel_geometric_1979}]\label{def:nhim}
Let $\Lambda$ be a compact,  invariant under $X$  manifold without boundary.
 Let $E^s$ and $E^u$ be subbundles of $T \mathbb{R}^d|\Lambda$ such that $E^s + E^u = T \mathbb{R}^d|\Lambda$, $E^s\cap E^u=T\Lambda$, $E^u$ is negatively invariant under $X$ and $E^s$ is negatively invariant under $-X$.
 We say $\Lambda$ is $r$-normally hyperbolic if $\Lambda$ is an overflowing invariant manifold with a subbundle $E^u$ satisfying the rate assumptions \eqref{eq:A} and $\Lambda$ is so with $E^s$ under $-X$.
\end{definition}

The geometric singular perturbation theory \cite[Theorem 9.1]{fenichel_geometric_1979} applies the invariant manifold theory
to the {\it fast-slow} structure induced by the dynamical system
%The difference is detailed in \cite{KUEHN_2015}.}
% \noindent
% \tcb{To MGL: Please check red statement below for accuracy. Maybe provide a reference?}
%
% \tcr{ The two notions of invariant manifolds in Definitions \ref{def:over} and \ref{def:nhim} are each related to a type of persistence theorem under perturbations of the vector field. }
% Before we introduce the persistence theorem by Fenichel in the form \cite[Theorem 2.2]{Sz1991},
%The invariant manifold theory is specialized for the dynamical system that has the {\it fast-slow} structure,
\begin{equation} \label{eq:fast-slow}
 \left\{
 \begin{aligned}
  \dot{x}&=f(x,y,\epsilon),\\
  \epsilon\dot{y}&=g(x,y,\epsilon),
 \end{aligned}\right. \quad \text{where $\epsilon \in (-\epsilon_0,\epsilon_0), \epsilon_0>0$ small, $x\in \mathbb{R}^\ell$, $y\in \mathbb{R}^k$, $\ell+k=d$.}
\end{equation}
We say $x$ is a slow variable and $y$ is a fast variable. We assume that $f$ and $g$ are sufficiently smooth,
and the terms slow and fast variable originate from
the two limiting asymptotic problems :
%     \hspace{3em} {Reduced Problem} \hspace{7em} Layer Problem $(\cdot)' = \frac{d}{d(t/\epsilon)} = \frac{d}{d\tilde{t}}.$
\begin{equation*} %\label{eq:reduced}
 \text{(Reduced Problem)}\quad\left\{
 \begin{aligned}
    \dot{x} &= f(x,y,0),\\
    0&= g(x,y,0),
 \end{aligned}\right.
 \hspace{2.5em}
 \text{(Layer Problem)} \quad
 \left\{
 \begin{aligned}
    x'&= 0,\\
    y'&= g(x,y,0), \quad (\cdot)' = \frac{d}{d(t/\epsilon)}.% = \frac{d}{d\tilde{t}}.
 \end{aligned}\right.
\end{equation*}
% where the former describes the dynamics under the limiting assumption the fast variables $y$ have arrived equilibrium and $x$ evolves slowly, and the latter describes that of the fast variables $y$ relaxing towards equilibrium manifold keeping the slow variables $x$ unchanged. Formally the latter describes the initial layer behavior as its name indicates.

The zeroset $\mathcal{S}$ of $g(x,y,0)$ defines a manifold the orbits of the Reduced problem take values. In general, $\mathcal{S}$ is not realized as a graph as it can have many branches. This manifold consists of equilibria of the Layer problem. We consider
\begin{align*}
 \mathcal{S}& = \Big\{ (x,y)\:\Big|\: g(x,y,0)=0\Big\},\\
 \mathcal{S}_R&\subset \Big\{ (x,y)\in \mathcal{S} \:\Big|\: \text{$D_y g(x,y,0)$ has the full rank $k$}\Big\} \quad \text{open},\\
 \mathcal{S}_H&\subset \Big\{ (x,y)\in \mathcal{S}_R \:\Big|\: \text{all eigenvalues of $D_y g(x,y,0)$ have nontrivial real parts}\Big\}\quad \text{open}.
\end{align*}
On $\mathcal{S}_R$, the equation $0=g(x,y,0)$ is locally solvable for $y$ in terms of $x$ and we speak of the reduced vector field $X_R$ on slow variables. (See equation (7.8) in \cite{fenichel_geometric_1979}.) On a compact subset $K \subset \mathcal{S}_H$ Fenichel's persistence Theorem \cite[Theorem 9.1]{fenichel_geometric_1979} applies. 

%We note that the normal hyperbolicity applied to $K$ in that theory is different from Definition \ref{def:nhim}, and the Theorem captures the {\it locally invariant manifolds with boundary} for $\epsilon>0$. By contrast in the application of the theory to our phase space analysis, we will only need Definition \ref{def:nhim}. 
%%and apply the theory in the fiorm explained in the monograph \cite{KUEHN_2015}.  
%\tcr{ Min-Gi check sentence and give a more precise reference}.

%A compact  $K \subset \mathcal{S}_H$ is normally hyperbolic to the Layer problem.

% Next, we use Fenichel's Theorems in the form of \cite[Theorem 2.2]{Sz1991}. In particular this extended version with overflowing invariant manifold needs \cite[Theorem 3]{fenichel_asymptotic_1977}. % and \cite[Theorem 2.2]{Szmolyan}.
% We omit the statements, but the result states the persistence properties of the compact set $K\subset\mathcal{S}_H$ and its stable and unstable manifolds under a small perturbation. The upshot of the theorem is that any $\mathcal{N}\hookrightarrow K$ in the reduced phase space that is normally hyperbolic invariant under the reduced vector field $X_R$ persists under the perturbation in a suitable sense. The stable and unstable manifolds of $\mathcal{N}$ has local  liftings to the unreduced phase space and persist under the perturbation as well. In particular, as far as the perturbations of $\mathcal{N}$ and its unstable manifold are concerned, it is enough to have $\mathcal{N}$ overflowing invariant as in Definition \ref{def:over} with its center-unstable bundle as $E$.

In the sequel, we need the notion of transversal intersections:

% we introduce the notion of the transversal intersection of two submanifolds in a phase space $\mathcal{M}$.
\begin{definition}[Transversal Intersection]  (\cite[Definition 3.1]{Sz1991})
 Let ${\mathcal{M}}_1$ and ${\mathcal{M}}_2$ be submanifolds of a manifold ${\mathcal{M}}$. The manifolds ${\mathcal{M}}_1$ and ${\mathcal{M}}_2$ intersect transversally at a point $m\in{\mathcal{M}}_1\cap {\mathcal{M}}_2$ iff
 $$T_m{\mathcal{M}} =  T_m{\mathcal{M}}_1+T_m{\mathcal{M}}_2$$
 holds, where $T_m\mathcal{M}$ denotes the tangent space of the manifold $\mathcal{M}$ and similarly for $\mathcal{M}_1$ and $\mathcal{M}_2$.
\end{definition}

It is shown in \cite{Sz1991} that when heteroclinic orbits are realized as transversal interesections of a stable and an unstable manifold
they are persistent under perturbations.

%We apply the notion of transversal intersection in sections \ref{sec:singorb} and \ref{sec:thmproof} in order to
%show persistence of a heteroclinic orbit for \eqref{eq:slow}.
%We take a simply connected branch of  $\mathcal{S}_H$ and a compact subset $K$ of the branch. We pick $\mathcal{N}_0$ and $\mathcal{N}_1$ in $K$. We show that for $n =0$ the heteroclinic orbit is achieved by the transversal intersection of the unstable manifold of $\mathcal{N}_0$ and the stable manifold of $\mathcal{N}_1$ in $K$. This transversal intersection in $K$ then lifts to the unreduced phase space and it further persists under perturbation for $n > 0$ small. (see \cite[Theorem 3.1]{Sz1991}.) %, when $\mathcal{N}_0$ and $\mathcal{N}_1$ are taken from the same branch $K$, the transversality in the reduced phase space constitutes the same in the unreduced phase space.

\subsection{Singular orbits for the inviscid system with $n=0$}\label{sec:singorb}

Let us describe how we apply the perturbation theory to the case of \eqref{eq:slow}. We take two normally hyperbolic manifolds %of the Reduced problem \eqref{eq:slow02}
$\mathcal{N}_0$ and $\mathcal{N}_1$, which are simply the equilibrium points $M_0$ and $M_1$. The goal of this section is to establish the transversal intersection of the $\mathcal{N}_0^u$, the unstable manifold of $\mathcal{N}_0$  in $pqrs$-space, and $\mathcal{N}^s_1$, the stable manifold of $\mathcal{N}_1$.

A bunch of symbols and series of notations, following \cite{Sz1991}, are introduced and used for the remaining section.
$K\subset\mathcal{S}_H$ denotes a compact critical manifold to be specified later. The reduced vector field problem is defined in $K$ and $X_R$ denotes the reduced vector field. As the reduced vector field is defined in $K$, $K$ may contain a finer invariant manifold $\mathcal{N}' \hookrightarrow K$ that is normally hyperbolic to the reduced vector field.
As usual $\mathcal{N'}$, $W^u(\mathcal{N}')$, and $W^s(\mathcal{N}')$ will be respectively the manifold, its unstable, and its stable manifolds embedded in $K$. In particular, $W_0^u$ denotes the reduced unstable manifold of $\mathcal{N}_0$ and $W_1^s$ denotes the reduced stable manifold of $\mathcal{N}_1$ embedded in $K$. Inside of $W^u(\mathcal{N}')$ are foliations $\mathcal{F}^u_x$ (see \cite[Theorem 12.2]{fenichel_geometric_1979}), meaning a foliation that passes  through $x\in W^u(\mathcal{N}')$.
$\mathcal{F}^s_x$ stands for the analogous foliation of  $W^s(\mathcal{N}')$.
Finally, when geometric objects are extended via the invariant manifold theory,  those objects for $n>0$  are denoted by a superscript $n$, for example $K^n$, $\mathcal{N}_0^n$, $\mathcal{N}_0^{u,n}$, $W_0^{u,n}$, $M_0^n$, $\cdots$.

Next, we identify the Reduced problem,
\begin{equation}\label{eq:slow0} \tag{R}
 \begin{aligned}
%  r &={r}(p,q,s,n=0) \triangleq \frac{ \frac{\alpha-m}{\lambda(1+\alpha)}a - q }{  \frac{\alpha-m}{\lambda(1+\alpha)} + \lambda p + \frac{\alpha}{\lambda}\big(s- \frac{1+m}{1+\alpha}\big)},\\% \quad \text{$={r}(0)$ for simplicity },\\
 \dot{p} &=p\Big(\frac{1}{\lambda}({r}-a) + 2- \lambda p {r} -q\Big),\\% & &\bigg(= p\Big(\frac{D}{\lambda(1+\alpha)}({r}-a_0) + \frac{\alpha}{\lambda}{r}\big(s- \frac{1+m}{1+\alpha}\big) \Big)\bigg),\\
 \dot{q} &=q\Big(1 -\lambda p {r} -q\Big) + b p {r},\\% & &\bigg(=q\Big(\frac{\alpha-m}{\lambda(1+\alpha)}({r}-r_1) + \frac{\alpha}{\lambda}{r}\big(s- \frac{1+m}{1+\alpha}\big) \Big) + b p {r}\bigg),\\
 0&=r\Big(\frac{\alpha-m}{\lambda(1+\alpha)}(r-a) + \lambda pr + q +\frac{\alpha}{\lambda}r\big(s- \frac{1+m}{1+\alpha}\big)\Big),\\
 \dot{s} &=s\Big(\frac{\alpha-m}{\lambda(1+\alpha)}({r}-a) + \lambda p{r} + q - \frac{1}{\lambda}{r}\big(s- \frac{1+m}{1+\alpha}\big)\Big),% & &\bigg(= -\frac{1+\alpha}{\lambda}{r}s\big(s- \frac{1+m}{1+\alpha}\big)\bigg).
 \end{aligned}
\end{equation}
and the Layer problem for the system \eqref{eq:slow},
\begin{equation} \label{eq:fast0}
 \begin{aligned}
 {p}' &=0, \quad {q}' =0, \quad {r}' =r\Big(\frac{\alpha-m}{\lambda(1+\alpha)}(r-a) + \lambda pr + q +\frac{\alpha}{\lambda}r\big(s- \frac{1+m}{1+\alpha}\big)\Big),  \quad{s}' =0 .
 \end{aligned}
\end{equation}
Here, $(\cdot)'= \frac{d}{d\tilde{\eta}} := \frac{d}{d(\eta/n)}$ denotes differentiation with respect to the fast independent variable $\tilde{\eta}$. The zero-set  of
the function
\begin{equation}
g(p,q,r,s) \triangleq r\Big(\frac{\alpha-m}{\lambda(1+\alpha)}(r-a) + \lambda pr + q +\frac{\alpha}{\lambda}r\big(s- \frac{1+m}{1+\alpha}\big)\Big)\label{eq:zeroset}
\end{equation}
consists of the equilibria of \eqref{eq:fast0}.

\subsubsection{Choice of the critical manifold $K$} \label{sec:choice}

The algebraic equation $g(p,q,r,s)=0$ specifies three dimensional hypersurfaces. Away from the $r\equiv0$ plane, one may obtain the hypersurface as a graph of
the function
\begin{equation*}
r = \hat{r}(p,q,s) = \frac{ \frac{\alpha-m}{\lambda(1+\alpha)}a - q }{  \frac{\alpha-m}{\lambda(1+\alpha)} + \lambda p + \frac{\alpha}{\lambda}\big(s- \frac{1+m}{1+\alpha}\big)} \, ,
\end{equation*}
or implicitly in the form
\begin{equation}
\frac{\alpha-m}{\lambda(1+\alpha)}(\hat{r}-a) + \lambda p\hat{r} + q +\frac{\alpha}{\lambda}\hat{r}\big(s- \frac{1+m}{1+\alpha}\big)=0. \label{eq:implicit}
\end{equation}

\begin{figure}[ht]
 \centering
  \psfrag{p}{\scriptsize $p$}%=\frac{\gamma}{\sigma}$}
  \psfrag{q}{\scriptsize~~~$q$}%=n\frac{v}{\sigma}$}
  \psfrag{s}{\scriptsize $s-\frac{1+m}{1+\alpha}$}%=n\frac{v}{\sigma}$}
  \psfrag{x0}{\scriptsize $M_0$}
  \psfrag{x1}{\scriptsize $M_1$}
  \psfrag{R1}{\scriptsize $\hat{r}(p,q,s)=a$}
  \psfrag{R2}{\scriptsize $\hat{r}(p,q,s)=R_1$}
  \psfrag{R3}{\scriptsize $\hat{r}(p,q,s)=R_2$}
%
%   \subfigure[Domain $G$ of a Graph]{
%   \includegraphics[width=5cm]{trapezoid.eps}\label{fig:flow0b}
%   }
%   \quad\quad
%   \subfigure[Affine level sets $\hat{r}(p,q,s)=R$ in $pqs$-space]{
  \includegraphics[width=6cm]{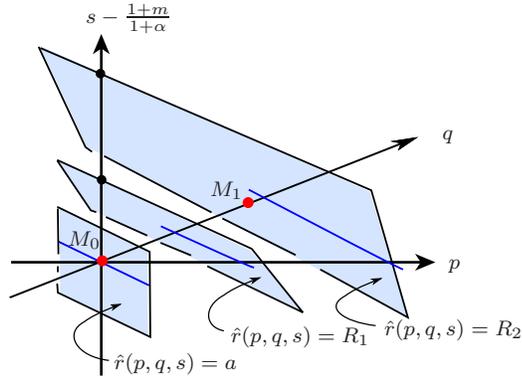}
%   }
  \caption{Affine level sets $\hat{r}(p,q,s)=R$, $0\le R\le a$  in $pqs$-space} \label{fig:affine}
\end{figure}

The level sets of $\hat{r}$ can be used in the $pqs$-space in order to visualize the hypersurface; they are affine due to \eqref{eq:implicit}.
Fig. \ref{fig:affine} illustrates a few marked affine surfaces of $\hat{r}(p,q,s)=R$, in the range $0\le R\le a$. When $R=a(=r_0)$, it passes through $\left(0,0,\tfrac{1+m}{1+\alpha}\right)$, which is the equilibrium $M_0$. As $R$ decreases the affine level sets sweep the positive $p,q$ sector. The surface crosses the other equilibrium $M_1$ when $R=r_1$. Then $R$ continues to decrease until it touches the $r\equiv0$ plane.

The critical manifold $K$ is selected taking account of the properties of $\hat{r}(p,q,r)$. The domain of $\hat{r}$ is a trapezoid in $pqs$-space
\begin{align*}
 D &\triangleq \left\{ \: (p,q,s) \: \Big| \:  p\ge-\epsilon, ~~ |q|\le2, ~~ \left|s-\frac{1+m}{1+\alpha}\right| \le \frac{1}{2}\min\left\{\frac{\alpha-m}{\alpha(1+\alpha)},\frac{1+m}{(1+\alpha)}\right\},\right.  \\
 &\left. \hat{r}(p,q,s)\ge \frac{1}{2}\min\{1,r_1\}\right\}.
\end{align*}
where  $\epsilon$ is a positive parameter selected sufficiently small.  $K$ is then defined by setting $K\triangleq\big(D,\hat{r}(D)\big)$.
% \begin{align*}
%  K &\triangleq \left\{ \: (p,q,r,s) \: \big| \:  p\ge-\epsilon, ~~ |q|\le2, ~~ \left|s-\frac{1+m}{1+\alpha}\right| \le \frac{1}{2}\min\left\{\frac{\alpha-m}{\alpha(1+\alpha)},\frac{1+m}{(1+\alpha)}\right\}, \right. \\
%  &\left. \hat{r}(p,q,s)\ge \frac{1}{2}\min\{1,r_1\},~~ r=\hat{r}(p,q,s), ~~\hat{r}(p,q,s) = \frac{ \frac{\alpha-m}{\lambda(1+\alpha)}a - q }{  \frac{\alpha-m}{\lambda(1+\alpha)} + \lambda p + \frac{\alpha}{\lambda}\big(s- \frac{1+m}{1+\alpha}\big)}\: \right\}.
% \end{align*}
\begin{figure}[ht]
 \centering
  \psfrag{p}{\scriptsize \hskip -2pt $p$}%=\frac{\gamma}{\sigma}$}
  \psfrag{q}{\scriptsize~~~$q$}%=n\frac{v}{\sigma}$}
  \psfrag{s}{\scriptsize $s-\frac{1+m}{1+\alpha}$}%=n\frac{v}{\sigma}$}
  \psfrag{x0}{\scriptsize \hskip -4pt$M_0$}
  \psfrag{x1}{\scriptsize $M_1$}
  \psfrag{K}{\scriptsize \hskip -85pt $\hat{r}(p,q,s)= \frac{1}{2}\min\{1,r_1\}$}
%   \psfrag{R1}{\scriptsize $\hat{r}(p,q,s)=a$}
%   \psfrag{R2}{\scriptsize $\hat{r}(p,q,s)=R_1$}
%   \psfrag{R3}{\scriptsize $\hat{r}(p,q,s)=R_2$}
  \includegraphics[width=5cm]{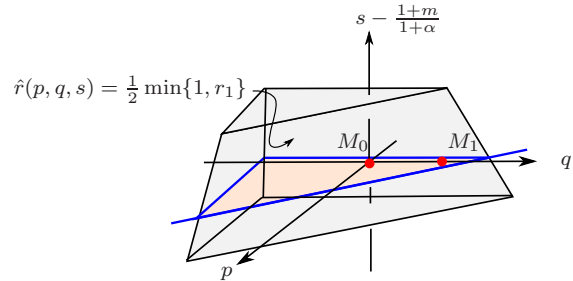}
%   }
  \caption{The trapezoid $D$, the domain of the graph.} \label{fig:D}
\end{figure}
Note that $K$ is chosen so that $(i)$ $M_0$ and $M_1$ are on $K$; $(ii)$ $s$ and $r=\hat{r}(p,q,s)$ have positive lower bound on $K$. See the trapezoid $D$ in Fig. \ref{fig:D}.

Next, we verify that $K\subset \mathcal{S}_H$.
\begin{proposition}
We have that $K\subset \mathcal{S}_H$, i.e., the partial jacobian $\frac{\partial g}{\partial r}(p,q,r,s)|_{r=\hat{r}(p,q,s)} >0 $ for all $(p,q,r,s)\in K$. %$K$ is normally hyperbolic with respect to the Layer problem \eqref{eq:fast0}.
\end{proposition}
\begin{proof}
 %Let $g(p,q,r,s) \triangleq r\Big(\frac{\alpha-m}{\lambda(1+\alpha)}(r-a) + \lambda pr + q +\frac{\alpha}{\lambda}r\big(s- \frac{1+m}{1+\alpha}\big)\Big)$, the right hand side of $\eqref{eq:fast0}_3$.  We need to show that $K\subset \mathcal{S}_H$, i.e., the partial jacobian $\frac{\partial g}{\partial r}$ is a nontrivial real number for all $m\in K$.
 \begin{align*}
 \left.\frac{\partial g}{\partial r}\right|_{K} &= \Big(\frac{\alpha-m}{\lambda(1+\alpha)}(\hat{r}-a) + \lambda p\hat{r} + q +\frac{\alpha}{\lambda}\hat{r}\big(s- \frac{1+m}{1+\alpha}\big)\Big) + \hat{r}\Big(\frac{\alpha-m}{\lambda(1+\alpha)} + \lambda p + \frac{\alpha}{\lambda}\big(s- \frac{1+m}{1+\alpha}\big)\Big)\\
 &= \hat{r}\Big(\frac{\alpha-m}{\lambda(1+\alpha)} + \lambda p + \frac{\alpha}{\lambda}\big(s- \frac{1+m}{1+\alpha}\big) \Big)\ge \frac{1}{2}\min\{1,r_1\}\Big(\frac{\alpha-m}{2\lambda(1+\alpha)} - \lambda \epsilon\Big).
 \end{align*}
 It suffices to take $\epsilon < \frac{\alpha-m}{4\lambda^2(1+\alpha)}$, independently of $n$.
\end{proof}

\subsubsection{Nested invariant manifold structures in $K$}
The flow \eqref{eq:slow0}, strictly restricted on $K$, is further analyzed. The three dimensional flow
\begin{equation}\label{eq:slow02} \tag{$\text{R}^\prime$}
 \begin{aligned}
%   r &=\hat{r}(p,q,s,n=0) \triangleq \frac{ \frac{\alpha-m}{\lambda(1+\alpha)}a - q }{  \frac{\alpha-m}{\lambda(1+\alpha)} + \lambda p + \frac{\alpha}{\lambda}\big(s- \frac{1+m}{1+\alpha}\big)},\\% \quad \text{$=\hat{r}(0)$ for simplicity },\\
 \dot{p} &= p\Big(\frac{D}{\lambda(1+\alpha)}(\hat{r}-a_0) + \frac{\alpha}{\lambda}\hat{r}\big(s- \frac{1+m}{1+\alpha}\big) \Big),\\
 \dot{q} &=q\Big(1 -\lambda p \hat{r} -q\Big) + b p \hat{r},\\% & &\bigg(=q\Big(\frac{\alpha-m}{\lambda(1+\alpha)}(\hat{r}-r_1) + \frac{\alpha}{\lambda}\hat{r}\big(s- \frac{1+m}{1+\alpha}\big) \Big) + b p \hat{r}\bigg),\\
%  0&=r\Big(\frac{\alpha-m}{\lambda(1+\alpha)}(r-a) + \lambda pr + q +\frac{\alpha}{\lambda}r\big(s- \frac{1+m}{1+\alpha}\big)\Big),\\
 \dot{s} &= -\frac{1+\alpha}{\lambda}\hat{r}s\Big(s- \frac{1+m}{1+\alpha}\Big).
 \end{aligned}
\end{equation}
augmented in the $r$-direction by $r=\hat{r}(p,q,s)$ is the flow of the Reduced system \eqref{eq:slow0}.

It is necessary to pinpoint a few finer invariant structures of the reduced flow \eqref{eq:slow02}, on which the invariant manifold theory can  equally well be applied.  This will be crucial ingredient of our arguments. To summarize, in the three dimensional reduced space $K$, we
consider the embeddings
$$
M_0, M_1 \quad \hookrightarrow \quad \text{$p\equiv0$ line on $s\equiv\tfrac{1+m}{1+\alpha}$ plane} \quad \hookrightarrow \quad \text{$s\equiv\tfrac{1+m}{1+\alpha}$ plane} \quad \hookrightarrow \quad K \, ,
$$
which consist of manifolds all of which are invariant under the flow \eqref{eq:slow02}. Indeed, on the plane $s\equiv\tfrac{1+m}{1+\alpha}$, \eqref{eq:slow02} decouples,
\begin{equation}\label{eq:slow03}
 \begin{aligned}
 \dot{p} &= p\Big(\frac{D}{\lambda(1+\alpha)}(\hat{r}-a_0)\Big),\\
 \dot{q} &= q\Big(1 -\lambda p \hat{r} -q\Big) + b p \hat{r},
 \end{aligned}
\end{equation}
where $\hat{r}=\hat{r}\big(p,q,\frac{1+m}{1+\alpha}\big)$. Restricting to $p\equiv0$ we obtain yet another invariant line and importantly this line contains
the equilibrium points $M_0$ and $M_1$.

\begin{figure}[ht]
 \centering
  \psfrag{p}{$p$}
  \psfrag{q}{$q$}
  \psfrag{s}{\hskip -2em $s-\tfrac{1+m}{1+\alpha}$}
  \psfrag{N2}{\scriptsize $N'$: segment in $q$-axis}
  \psfrag{F2}{\scriptsize normal foliations}% $\mathcal{F}(N_2)$}
  \psfrag{M0}{$M_0$}
  \psfrag{M1}{$M_1$}
  \psfrag{F0}{}%{\tiny \hskip -10em $M_0\xhookrightarrow{} N_2\xhookrightarrow{} \underbrace{N_2\oplus \mathcal{F}(N_2)}_{\triangleq N_1} $}
  \psfrag{F1}{}%{\tiny $M_1\xhookrightarrow{} W^s(M_1)\xhookrightarrow{} W^s(M_1)\oplus \mathcal{F}(N_0)$}
 \includegraphics[height=5cm]{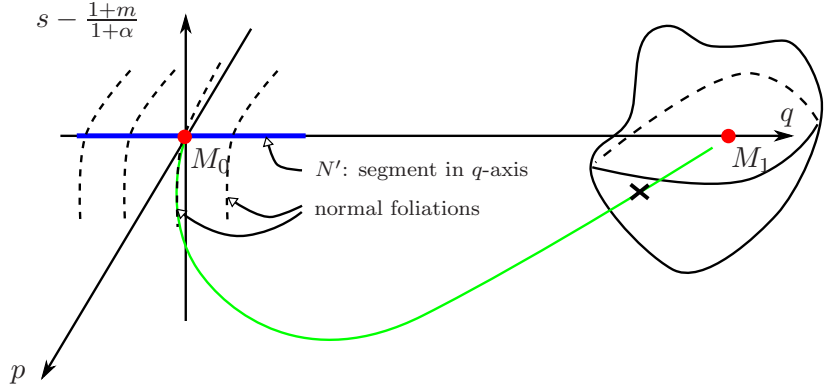}
  \caption{Nested invariant manifold structures} \label{fig:HF}
\end{figure}
Fig. \ref{fig:HF} illustrates the rest of the program. The justification of the following descriptions will be the subject of the next section. $M_1$ is a stable node; the three dimensional volume surrounding $M_1$ in Fig. \ref{fig:HF} depicts its stable manifold. $M_0$ is a saddle; $M_0$ has two unstable dimensions in $s\equiv\tfrac{1+m}{1+\alpha}$, and has one stable dimension in its oblique direction. The vector aligned to $q$-axis and the one to the green orbit are two eigenvectors for the unstable dimensions. This explains how our heteroclinic orbit (the green one) appears in $K$. %The one to the $q$-axis always has the smaller eigenvalue $1$ to the other eigenvalue $2$. It is common idea in the invariant manifold theory to conceive the rapid orbit with greater eigenvalue as one of the foliations onto the slow orbit with smaller eigenvalue. This shall be clarified in the below.

Not all of the manifolds appearing are {\it normally hyperbolic} to the reduced vector field \eqref{eq:slow02}: $M_0$ and $M_1$ are hyperbolic equilibrium points; $\hat{\mathcal{N}}^0$, a segment of $p\equiv0$ line (the blue portion in Fig. \ref{fig:HF}) will be identified as an overflowing manifold satisfying the rate assumptions \eqref{eq:A} and \eqref{eq:B}. However,
the plane $s\equiv \tfrac{1+m}{1+\alpha}$, in general, does not satisfy the necessary rate assumptions. See Remark \ref{plane_hyp}

\subsubsection{Analysis of the flow of  \eqref{eq:slow02}.}%on $s\equiv\tfrac{1+m}{1+ \alpha}$}

In the phase space $K$, the flow of \eqref{eq:slow02} can be completely analyzed. We visualize the overall flow by first analyzing the flow when restricted to the invariant plane $s\equiv\tfrac{1+m}{1+\alpha}$, and then noting that off the invariant plane the flow amounts to a stable  relaxation process towards the invariant plane at
$s\equiv\tfrac{1+m}{1+\alpha}$,  see Fig. \ref{fig:n0pqs}. The flow in the plane $s\equiv\tfrac{1+m}{1+\alpha}$ is characterized by using planar dynamical systems theory.

\begin{figure}[ht]
 \centering
%  \subfigure[Flow on and around $s\equiv\tfrac{1+m}{1+ \alpha}$ seen in $pqr$-space]{
 \psfrag{p}{$p$}
 \psfrag{q}{$q$}
 \psfrag{s}{$s-\tfrac{1+m}{1+ \alpha}$}
 \psfrag{x0}{\hskip -2pt $M_0$}
 \psfrag{x1}{\hskip -4pt $M_1$}
 \psfrag{B}{}
 \includegraphics[width=4cm]{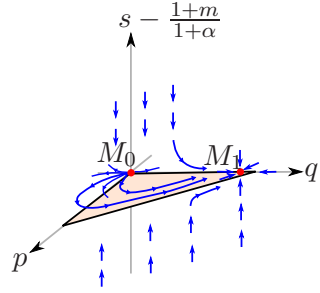}
%  }
 \caption{Flow on and around $s\equiv\tfrac{1+m}{1+ \alpha}$} \label{fig:n0pqs}
\end{figure}

\begin{remark} \label{plane_hyp}
The property that the invariant plane $s\equiv\tfrac{1+m}{1+ \alpha}$ has one stable direction in $K$  may lead one to believe that the plane is normally hyperbolic.
Indeed, the plane $s\equiv\tfrac{1+m}{1+ \alpha}$ admits a splitting $T\mathbb{R}^3=T\bar\Lambda \oplus E^s$ with the normal direction decaying.
However, the notion of normal hyperbolicity requires stronger properties than merely admitting a splitting: Note that in \eqref{eq:A} and \eqref{eq:B} upon a given splitting we demand $\sigma^s<1/r$ and $\rho^u<1/r'$. The latter for example leads to demanding that the negative eigenvalue $\mu_{14}$ of $M_1$ is strictly less than all other negative eigenvalues, which is not always the case. As a consequence, the persistence Theorem does not apply and we do not assert its persistence under perturbations.
\end{remark}

We turn to the reduced linear stability of $M_0$ and $M_1$ in $pqs$-space. To summarize, the expressions in Section \ref{sec:equil} also hold for $n=0$, with the exception that the third eigenvalue $\mu_{03}$ and the third eigenvector $X_{03}$ of $M_0$ ($\mu_{13}$ and $X_{13}$ respectively of $M_1$) are no longer used. Indeed when $n=0$, the flow is restricted on the three dimensional set $K$. %In the reduced phase space $K$,

The following Lemma utilizes planar dynamical systems theory to characterizes the flow on a triangle $T$ in the $pq$-plane. %first quadrant of the $pq$-plane.

\begin{figure}[ht]
 \centering
  \psfrag{X01}{\scriptsize$X_{02}$}
  \psfrag{X02}{\scriptsize$X_{01}$}
%   \psfrag{A}{$A$}
%   \psfrag{T}{$T$}
%   \psfrag{x0}{\scriptsize $M_0$}
%   \psfrag{x1}{\scriptsize $M_1$}
%   \psfrag{r0}{\scriptsize $\hat{r}(p,q,s,0)=r_0$}
%   \psfrag{r1}{\scriptsize $\hat{r}(p,q,s,0)=r_1$}
  \psfrag{p}{\scriptsize $p$}%=\frac{\gamma}{\sigma}$}
  \psfrag{q}{\scriptsize~~~$q$}%=n\frac{v}{\sigma}$}
  \psfrag{s}{\scriptsize $s-\frac{1+m}{1+\alpha}$}%=n\frac{v}{\sigma}$}
%   \subfigure[Affine level sets $\hat{r}(p,q,s)=R$ in $pqs$-space]{
%   \psfrag{r}{\scriptsize$r$}%=\big(\sigma\gamma^{(1-n)}\big)^{\frac{1}{n}}$}
  \psfrag{CC}{\scriptsize\hskip 5pt$\frac{\alpha-m}{\lambda(1+\alpha)}\big(-\frac{1}{\lambda},a\big)$}
  \includegraphics[width=8cm]{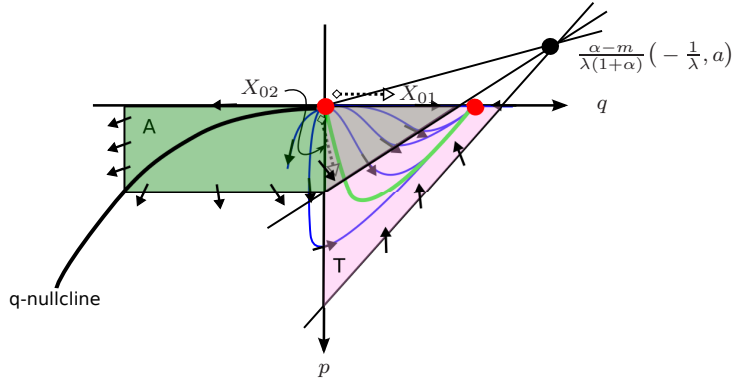} \label{fig:flow0}
%   \flushleft
%   }
%   \quad \quad
%   \subfigure[Flow in $pqs$-space for $n=0$]{
%   \psfrag{r}{\scriptsize$s-\frac{1+m}{1+\alpha}$}%=\big(\sigma\gamma^{(1-n)}\big)^{\frac{1}{n}}$}
%   \includegraphics[width=6cm]{flow0pqs.eps}\label{fig:flow0b}
%   }
  \caption{The schematic sketch of the flow on the invariant plane $s=\frac{1+m}{1+\alpha}$.  Restricted on $s=\frac{1+m}{1+\alpha}$, $M_0$ is an unstable node and $M_1$ is a stable node. Two directions of unstable subspaces of $M_0$ are denoted by $X_{01}$ and $X_{02}$; the straight lines emanating from the
  point $\frac{\alpha-m}{\lambda(1+\alpha)}\big(-\frac{1}{\lambda},a\big)$ are the intersections of the level sets of $\hat{r}$
  with the plane $s=\frac{1+m}{1+\alpha}$; the curve in the fourth quadrant is the nullcline of the equation $\eqref{eq:slow02}_2$; the triangle $T$ is a 2-dimensional positively invariant set; the trapezoid $A$ is a 2-dimensional negatively invariant set; the $\omega$-limit set of any point in $T$ is $M_1$;
  the $\alpha$-limit set of any point in $A$ is $M_0$; in particular there is a heteroclinic orbit (green one) that emanates from $M_0$ in the direction $X_{01}$ lying
  in the strongly unstable manifold of $M_0$.}
\end{figure}

\begin{lemma} \label{lem:T}
 Let $T$ be the closed triangle on $s\equiv \frac{1+m}{1+\alpha}$ enclosed by $p=0$, $q=0$, and the level set $\hat{r}(p,q,\frac{1+m}{1+\alpha})= \frac{1}{2}\min\{1,r_1\}$ that is intersected by $s\equiv \frac{1+m}{1+\alpha}$.
 Then $T\setminus M_0 \subset W^s(M_1)$.
\end{lemma}
\begin{proof}
$T$ is a two dimensional compact positively invariant set: (1) on $p=0$, $\nu = (1,0)$ and $X_R\cdot\nobreak\nu = \dot{p}=0$, where $X_R$ stands for the reduced vector field of \eqref{eq:slow02};
 (2) on $q=0$, $\nu = (0,1)$ and $X_R\cdot\nobreak\nu=\dot{q} = bp\hat{r}\ge0$ ($b$ in \eqref{eq:exponents} is always positive); lastly on the hypotenuse, let $\underbar{r}=\frac{1}{2}\min\{1,r_1\}$. The inward normal vector is $\nu = (-\lambda\underbar{r}, -1)$. We compute
  \begin{align}
  X_R\cdot\nu=-\lambda\underbar{r}\dot{p} -\dot{q}&= -\lambda \underbar{r}p \Big(1-\lambda \underbar{r}p -q + \frac{1}{\lambda}(\underbar{r}-a)+1\Big) - q(1-\lambda \underbar{r}p -q\big) - b \underbar{r}p \nonumber\\
  &= (1-\lambda \underbar{r}p -q)(-\lambda \underbar{r}p -q) -\underbar{r}p\Big((\underbar{r}-a)+\lambda+b\Big)\nonumber\\
  &= \left(\frac{\alpha-m}{\lambda(1+\alpha)}\right)^2(\underbar{r}-r_0)(\underbar{r}-r_1)+\underbar{r}p(1-\underbar{r})\ge \delta>0. \label{eq:affine}
 \end{align}

Let $\Omega$ be an $\omega$-limit set of the orbit from $x_0\in T\setminus M_0$. It is non-empty because $T$ is compact.
 It cannot contain $M_0$, because for the flow restricted in $T$, $M_0$ does not have  a stable manifold. It cannot contain a periodic orbit; if it did then there would be a fixed point in the interior of $T$ and this is not the case.
It cannot contain a separatrix cycle because $T$ has only two fixed points $M_0$ and $M_1$ and again $M_0$ does not have  a stable manifold.  By Poincar\'e-Bendixson Theorem, the $\omega$-limit set is $M_1$.
\end{proof}

Now we are able to state: We call $\mathcal{F}^u_{M_0}\subset W^u_0$ the strongly unstable manifold of $M_0$ satisfying \eqref{eq:rapid} (the green line in Fig. \ref{fig:HF}) that is characterizable by the Unstable manifold theorem for the hyperbolic fixed point. That $\mathcal{F}^u_{M_0}$ ends up arriving at $M_1$
follows by Lemma \ref{lem:T}, and this gives the proof for $n=0$ of Theorem \ref{thm1}. The following proposition shows that the one dimensional manifold $\mathcal{F}^u_{M_0}\subset W^u_0$ intersects the three dimensional manifold $W_1^s(=W^s(M_1))$ transversally (see Fig. \ref{fig:HF}).

\begin{proposition} \label{prop:singular}
 Let $\mathcal{N}_0=M_0$, $\mathcal{N}_1=M_1$, $\mathcal{F}^u_{M_0}\subset W^u_0$  the strongly unstable manifold of $M_0$ satisfying \eqref{eq:rapid}, $W^s_1=\Phi_{-t_0}(W^s_{loc}(M_1))$, the time $-t_0$ image of the local stable manifold of $M_1$ for large enough $t_0<\infty$. Then $\mathcal{F}^u_{M_0}$ intersects $W^s_1$ transversally in $pqs$-space.
\end{proposition}
\begin{proof}[Proof of Proposition \ref{prop:singular}]
% Let $\mathcal{F}^u_{M_0}\subset W^u(M_0)$ be the most rapidly escaping orbit from $M_0$ satisfying \eqref{eq:rapid} that is characterizable by the Unstable manifold theorem for the hyperbolic fixed point. By Lemma \ref{lem:T}, $\mathcal{F}^u_{M_0}$ survives to end up arriving at $M_1$. This proves the existence of the heteroclinic orbit for $n=0$.
%
% Now let  Let $W^s_{loc}(M_1)$ be the local three dimensional stable manifold of $M_1$ given by the Stable manifold theorem. Let $W^s_1=\Phi_{-t_0}(W^s_{loc}(M_1))$ be the time $-t_0$ image of $W^s_{loc}(M_1)$.
For large enough $t_0<\infty$, by Lemma \ref{lem:T} the orbit point $x\in \mathcal{F}^u_{M_0}$ must be attained in $W^s_1$ as an interior point. Therefore the tangent space ${T}_x W^s_1$ is the whole of ${T}_x \mathbb{R}^3$. Then the intersection with $\mathcal{F}^u_{M_0}$ is trivially transversal.
\end{proof}

\subsection{Persistence for $n>0$} \label{sec:thmproof}
Having set forth the critical manifold $K$ in $\mathcal{S}_H$ and the reduced vector field $X_R$ on $K$, the theorem of Fenichel holds in $K$; the family $K^n$ of {\it slow manifolds} persistently exist provided $n$ is sufficiently small. Now we show the finer hyperbolic structure of $\mathcal{F}^u_{M_0} \hookrightarrow K$. %It is common idea in invariant manifold theory to characterize the strongly unstable manifold as the unstable manifold of the complementary part in $W^u(M_0)$.

\begin{lemma} \label{lem:rapid}
 Let $\mathcal{N}_0=M_0$, $\mathcal{F}^u_{M_0}\subset W_0^u$  the strongly unstable manifold of $M_0$ satisfying \eqref{eq:rapid}. Then, for sufficiently small $n$, $\mathcal{F}^u_{M_0}$ perturbs in a $C^{r-1}$ manner to $\mathcal{F}^{u,n}_{M_0^n}$ the strongly unstable manifold of $M_0^n$ satisfying \eqref{eq:rapid}.
\end{lemma}
\begin{proof}
\tcb {
In the $pqs$-space we select the line segment
$\hat{\mathcal{N}}^{0}$ that is the transversal intersection of the invariant plane $\Big\{(p,q,s)~|~ p=0 \text{ and }q\in [- \frac{1}{2}, \frac{1}{2}] \Big\}$ and the unstable manifold $W_0^u$. $\hat{\mathcal{N}}^{n}$ is defined as the intersection of the same plane with $W_0^{u,n}$. We see that $\hat{\mathcal{N}}^{0}=\Big\{(p,q,s)~|~ p=0, ~ q\in [- \frac{1}{2}, \frac{1}{2}] \text{ and } s= \frac{1+m}{1+\alpha}\Big\}.$
% $q\in [- \frac{1}{2}, \frac{1}{2}]$, $p=0$ on the plane , which we will denote
%  Consider a family of invariant manifolds  $\hat{\mathcal{N}}^{n}$ that is the transversal intersection of  $W^{u,n}_0$ for each $0\le n\le n_0$ with $n_0$ sufficiently small. When $n=0$, that $W^{u,0}_0$ contains the set $q\in [- \frac{2}{3}, \frac{2}{3}]$, $p=0$ and $s= \frac{1+m}{1+\alpha}$ is straightforward. Since the Unstable manifold is persistent, taking $q\in [- \frac{1}{2}, \frac{1}{2}]$ is the restriction from the properly larger set.
 }

  $\hat{\mathcal{N}}^{0}$ is the one dimensional orbit in $W_0^u$ that is not strongly unstable. We claim that $\hat{\mathcal{N}}^{0}$ is an overflowing invariant manifold as in Definition \ref{def:over} of the reduced problem. More precisely, it satisfies \eqref{eq:A} and \eqref{eq:B} with $r'=r-1$ and $E$ the tangent $pq$-plane.

 From \eqref{eq:slow03}, $\dot{q}=q(1-q)$ on $q$-axis, it is clear that $\hat{\mathcal{N}}^{0}$ is overflowing invariant. Let $E$ be $pq$-plane along $\hat{\mathcal{N}}^{0}$ and $E'$ be the lines parallel to the $s$-axis. Then, $T \mathbb{R}^3|\hat{\mathcal{N}}^{0}$ splits into three one dimensional bundles $T\hat{\mathcal{N}}^{0}\oplus N \oplus E'$ with $N$ complementary to $T\hat{\mathcal{N}}^{0}$ in $E$ such that $N_{M_0}$ is parallel to $X_{01}$.  %defined by the three coordinate basis in $pqs$-space, i.e., $\mathbf{e}_q\in T\Lambda$, $\mathbf{e}_p\in N$, and $\mathbf{e}_s\in E'$.
 The asymptotic rates are determined at $M_0$ by the eigenvalues of $M_0$. At $M_0$, $E'_{M_0}$ is the stable subspace with eigenvalue $-\mu_{04}$ and $N_{M_0}$ and $T_{M_0}\hat{\mathcal{N}}^{0}$ are the unstable ones with $\mu_{01}=2$ and $\mu_{02}=1$ respectively. From these, we compute
 $$ \nu^s = e^{-\mu_{04}}, \quad\sigma^s = 0, \quad\alpha^u = e^{-2}, \quad\rho^u=e^{-1}, \quad\tau^u=0.$$

 Therefore, for the given family of overflowing manifolds $\hat{\mathcal{N}}^{n}$, the strongly unstable manifold and its foliations $\mathcal{F}^{u}(x,n):=\mathcal{F}^{u,n}_x$ exist as a $C^{r-1}$ family in both arguments $x$ and $n$. In turn, the foliation
 $\mathcal{F}^u(M_0^n,n)$ that passes through $M_0^n$ is a $C^{r-1}$ map in $n$.
\end{proof}

Persistence of the stable manifold $W_1^s (= W^s(M_1))$ is a consequence of the classical stable manifold theorem. Theorem \ref{thm1} follows in the same way as in \cite[Theorem 3.1]{Sz1991} by the transversal intersection.

\begin{proof}[Proof of Theorem \ref{thm1}]
%  By \cite[Theorem 3.1]{Sz1991}, In particular, $\mathcal{F}^u_{M_0}$ perturbs to $\mathcal{F}^u_{M_0^n}$, the most rapidly escaping orbit from $M_0^n$ by Lemma \ref{lem:rapid}
 By the theorem of Fenichel, for given $(\lambda,\alpha,m,0)$ satisfying \eqref{eq:paramrange} and \eqref{eq:lambda}, $n_0$ can be taken sufficiently small so that  if $n \in [0, n_0)$ then $(\lambda,\alpha,m,n)$ satisfies \eqref{eq:paramrange} and \eqref{eq:lambda} and  the system  \eqref{eq:slow} admits a transversal heteroclinic orbit joining equilibrium $M_0^{n}$ to equilibrium $M_1^{n}$: $\mathcal{F}^u_{M_0}$ perturbs to $\mathcal{F}^{u,n}_{M_0^n}$ by Lemma \ref{lem:rapid} and $W_1^s$ perturbs to $W_1^{s,n}$ and the transversal intersection is stable under the perturbation.
\end{proof}

\section{Emergence of localization}
\label{sec:localization}
By transforming back using \eqref{eq:ORItoCAP}, \eqref{eq:CAPtoBAR}, \eqref{eq:BARtoTIL}, and \eqref{eq:pqrdef}, we recover the profile
$\big(\Gamma(\xi),V(\xi),\Theta(\xi),\Sigma(\xi)\big)$ and $U(\xi)$ by \eqref{eqnsigma} and the associated solution.  %\$\big(\gamma(x,t),v(x,t),\theta(x,t),\sigma(x,t),u(x,t)\big)$ are recovered.
We replace $t \rightarrow t+1$ to obtain the final expression:
\begin{equation*}
\begin{aligned}
 \gamma(t,x) &= (t+1)^a\Gamma((t+1)^\lambda x), & v(t,x) &= (t+1)^b V((t+1)^\lambda x), &\theta(t,x) &= (t+1)^c \Theta((t+1)^\lambda x),\\
 \sigma(t,x) &= (t+1)^d \Sigma((t+1)^\lambda x), & u(t,x) &= (t+1)^{b+\lambda} U((t+1)^\lambda x) \, .
\end{aligned}
\end{equation*}
We interpret $\big(\Gamma(\xi),V(\xi),\Theta(\xi),\Sigma(\xi)\big)=\big(\gamma(0,x),v(0,x),\theta(0,x),\sigma(0,x)\big)|_{x=\xi}$ as the initial state.
For given material parameters $(\alpha, m, n)$, there are two available degrees of freedom giving rise to
a two-parameters family of solutions.
As noted in Section \ref{sec:twoparam}, the choices of $U_0$ and $\Gamma_0$ determine the self-similar profile  while the remaining boundary values $(\Theta_0$, $\Sigma_0)$
and the rate $\lambda$
are induced by them. The range of $U_0$ and $\Gamma_0$ is such that
 $$\frac{2(1+\alpha) -n}{D} < \frac{U_0}{\Gamma_0} < \frac{2(1+\alpha) -n}{D} + \frac{4(1+\alpha)(\alpha-m-n)(1+m)}{D(1+m+n)^2}.$$
The localizing rate $\lambda$ satisfies \eqref{eq:lambda} and takes values $0< \lambda < \frac{2(\alpha-m-n)}{1+m+n}\left(\frac{1+m}{1+m+n}\right).$
In the sequel, we establish properties of the profiles and the emergence of localization, in the sense of definition \eqref{deflocal1}, \eqref{deflocal2}.

\subsection{Properties of the self-similar profiles}
We first list some  information on the behavior of the profiles near $\xi = 0$ and as $\xi \to \infty$. The latter determines the behavior of the
induced solutions off the localization zone.

% Now, we fix the two primary parameters $U(0)$ and $\Theta(0)$ and look into the
%
%
% In this section we contrast the initial non-uniformities $\big(\gamma(0,x),v(0,x),\theta(0,x),\sigma(0,x),u(0,x)\big)=\\ \big(\Gamma(\xi),V(\xi),\Theta(\xi),\Sigma(\xi),U(\xi)\big)|_{\xi=x}$ from the one snapshot at a certain time of the uniform shearing motion.
\begin{proposition} \label{prop:ss}
Let $\big(\Gamma(\xi),V(\xi),\Theta(\xi),\Sigma(\xi)\big)$ be the self-similar profiles defined by transformations of \eqref{eq:CAPtoBAR}, \eqref{eq:BARtoTIL}, \eqref{eq:pqrdef}
from the heteroclinic orbit $\chi(\eta)=\big(p(\eta),q(\eta),r(\eta),s(\eta)\big)$ constructed in Theorem \ref{thm1} in the range of  parameters $\Gamma(0)=\Gamma_0$ and $U(0)=U_0$
depicted by \eqref{eq:restriction}. $U(\xi)$ is defined by \eqref{eqnsigma}. Then,
 \begin{enumerate}
  \item[(i)] The self-similar profile achieves the boundary condition at $\xi=0$,
    \begin{equation*}
    {V}(0) = \Gamma_\xi(0) = \Theta_\xi(0)=\Sigma_\xi(0) = {U}_\xi(0)=0, \quad \Gamma(0)=\Gamma_0, \quad U(0)=U_0.
  \end{equation*}
  \item[(ii)] Its asymptotic behavior as $\xi \rightarrow 0$ is given by
  \begin{equation} \label{eq:ss_asymp0}
  \begin{aligned}
    \Gamma(\xi) -\Gamma_0 &= \Gamma^{''}(0)\frac{\xi^2}{2} + o(\xi^2), & \Gamma^{''}(0)&<0,\\
    \Theta(\xi) - c^{-\frac{1}{1+\alpha}}\Gamma_0^{\frac{m}{1+\alpha}} U_0^{\frac{1+n}{1+\alpha}} &= \Theta^{''}(0)\frac{\xi^2}{2} + o(\xi^2), & \Theta^{''}(0)&<0,\\
    \Sigma(\xi) - c^{\frac{\alpha}{1+\alpha}}\Gamma_0^{\frac{m}{1+\alpha}} U_0^{-\frac{\alpha-n}{1+\alpha}} &= \Sigma^{''}(0)\frac{\xi^2}{2} + o(\xi^2), & \Sigma^{''}(0)&>0, \\
    U(\xi) - U_0 &= U^{''}(0)\frac{\xi^2}{2} + o(\xi^2), & U^{''}(0)&<0,\\
    V(\xi) - U_0\xi &= U^{''}(0)\frac{\xi^3}{6} + o(\xi^3), & U^{''}(0)&<0.
  \end{aligned}
  \end{equation}
  \item[(iii)] Its asymptotic behavior as $\xi \rightarrow \infty$ is given by\\
  if $\mu_{11}\ne-1$, or $\mu_{11}=-1$ but $b=\lambda$,
  \begin{equation} \label{eq:ss_asymp1}
  \begin{aligned}
    \Gamma(\xi) &= \BO\big(\xi^{-\frac{1+\alpha}{\alpha-m-n}}), & V(\xi) &= \BO\big(1), &    \Theta(\xi) &= \BO\big(\xi^{-\frac{1+m+n}{\alpha-m-n}}),\\
   \Sigma(\xi) &= \BO\big(\xi), &   U(\xi) &= \BO\big(\xi^{-\frac{1+\alpha}{\alpha-m-n}})
  \end{aligned}
  \end{equation}
  otherwise
    \begin{equation} \label{eq:ss_asymp2}
  \begin{aligned}
    \Gamma(\xi) &= \BO\big(\xi^{-\frac{1+\alpha}{\alpha-m-n}}\big(\log\xi\big)^{\frac{1+\alpha}{D}}\big), & V(\xi) &= \BO\big(\big(\log\xi\big)^{-\frac{\alpha-m-n}{D}}\big),
    \\
        \Theta(\xi) &= \BO\big(\xi^{-\frac{1+m+n}{\alpha-m-n}}\big(\log\xi\big)^{\frac{1+m+n}{D}}\big),\\
   \Sigma(\xi) &= \BO\big(\xi\big(\log\xi\big)^{-\frac{\alpha-m-n}{D}}\big), &   U(\xi) &= \BO\big(\xi^{-\frac{1+\alpha}{\alpha-m-n}}\big(\log\xi\big)^{\frac{1+\alpha}{D}}\big)
  \end{aligned}
  \end{equation}
 \end{enumerate}
\end{proposition}

\begin{proof}
The proof of the Proposition \ref{prop1} and Remark \ref{rem:signs} contains $(i)$ and $(ii)$ and thus we are left to prove $(iii)$. In a similar fashion to \eqref{eq:alpha-expan}, any orbit
$\psi(\eta)$ in the local stable manifold of $W^s(M_1)$ is characterized by a triple $(\kappa_1',\kappa_2',\kappa_3')$ in association with the asymptotic expansion
\begin{equation}
\begin{aligned}
 &\psi(\eta) -M_1\\
 &= \begin{cases} \kappa_1'e^{\mu_{11}\eta}X_{11} + \kappa_2'e^{\mu_{12}\eta}X_{12} + \kappa_4'e^{\mu_{14}\eta}X_{14} + \text{high order terms} & \text{if $\mu_{11}\ne-1$, or $\mu_{11}=-1$ but $b=\lambda$,}\\
 \kappa_1'\eta e^{\mu_{11}\eta}X_{11}' + \kappa_2'e^{\mu_{12}\eta}X_{12} + \kappa_4'e^{\mu_{14}\eta}X_{14} + \text{high order terms} & \text{if $\mu_{11}= \mu_{12} = -1$ and $b \ne \lambda$}
 \end{cases}
\end{aligned}
\end{equation}
as $\eta \rightarrow \infty$. The second formula reflects the presence of a generalized eigenvector.

Now, we have $q \rightarrow 1$, $r \rightarrow r_1$, $s \rightarrow s_1$ but $p \rightarrow 0$ and the leading order of $p$ is to be found. We can determine the coefficient
of $X_{11}$, above, because the $p$-component of the vectors $X_{12}$ and $X_{14}$ is $0$.
Since the plane $p\equiv0$ is an invariant plane for a non-linear flow, triplets of the form $(0,\kappa_2',\kappa_3')$ spans this invariant plane. Because our heteroclinic orbit $\chi(\eta)$ ventures out from the plane $p\equiv0$, $\kappa_1'$ for the expansion of $\chi(\eta)$ cannot be $0$.
This implies that the leading order of $p(\log\xi)$ is
$$p(\log\xi) = \begin{cases} \BO(\xi^{\mu_{11}}) & \text{if $\mu_{11}\ne-1$ or $\mu_{11}=-1$ but $b=\lambda$,}\\
 \BO(\xi^{\mu_{11}}\log\xi) & \text{otherwise}
 \end{cases}
 $$
as $\xi \rightarrow \infty$.

Asymptotics \eqref{eq:ss_asymp1} and \eqref{eq:ss_asymp2} are the straightforward calculations obtained from the reconstruction formulas
\begin{align*}
 \tg&=p^{\frac{1+\alpha}{D}}r^{\frac{n}{D}}s^{\frac{\alpha}{D}}, & \tv &= \frac{1}{b} p^{-\frac{\alpha-m-n}{D}}qr^{\frac{n}{D}}s^{\frac{\alpha}{D}}, & \tth&=p^{\frac{1+m+n}{D}}r^{\frac{2n}{D}}s^{-\frac{1-m-n}{D}}, \\ \ts&=p^{-\frac{\alpha-m-n}{D}}r^{\frac{n}{D}}s^{\frac{\alpha}{D}},  & \tu&=p^{\frac{1+\alpha}{D}}r^{\frac{n}{D}+1}s^{\frac{\alpha}{D}},
\end{align*}
via \eqref{eq:CAPtoBAR} and \eqref{eq:BARtoTIL}.
%
% $p \sim \xi^{-\frac{1+n}{1+\alpha}}$
%
% $$ \tv \sim \xi^{-\frac{1+n}{D}}, \quad \tth \sim \xi^{-\frac{(1+n)^2}{(\alpha-n)D}}, \quad \ts \sim \xi^{\frac{1+n}{D}}, \quad \tu \sim \xi^{-\frac{(1+n)(1+\alpha)}{(\alpha-n)D}}$$
%
% $$ V \sim O(1), \quad \Theta \sim \xi^{-\frac{1+n}{\alpha-n}}, \quad \Sigma \sim \xi, \quad U \sim \xi^{-\frac{1+\alpha}{\alpha-n}}.$$
\end{proof}

\subsection{Emergence of localization}
\label{sec:emloc}

As time proceeds, the initial nonuniformity evolves into localization. This section is devoted to describing the behavior
of the various fields as time advances. We only present the generic case $-\frac{1+m+n}{\alpha-m-n}\ne-1$;
in the non-generic case we would have to add a logarithmic correction according to Proposition \ref{prop:ss}.
\begin{itemize}
 \item Strain : The strain keeps increasing as time proceeds. The growth at the origin is faster than the growth rate at  all other points:
\begin{align*}
 \gamma(t,0) &= (1+t)^{\frac{2+2\alpha-n}{D} + \frac{2+2\alpha}{D}\lambda}\Gamma(0),
 \\
 \gamma(t,x) &\sim t^{\frac{2+2\alpha-n}{D} - \frac{(1+\alpha)(1+m+n)}{D(\alpha-m-n)}\lambda}|x|^{-\frac{1+\alpha}{\alpha-m-n}}, \quad \text{as $t \rightarrow \infty$, $x\ne0$.}
\end{align*}
Recall that the condition  $\frac{2+2\alpha-n}{D} - \frac{(1+\alpha)(1+m+n)}{D(\alpha-m-n)}\lambda > 0$ was the ground for imposing \eqref{eq:lambda-range},
placed to guarantee that the plastic strain is growing even outside the localization zone. On the other hand, the difference between the rate of growth of $\gamma$ at $x=0$
and the rate at $x \ne 0$ is easily computed as $\frac{1+\alpha}{\alpha - m -n} \lambda > 0$, which indicates localization of the profile of $\gamma$ around $x=0$.
\item Temperature : For the  temperature , the growth at the origin is again faster than other points,
\begin{align*}
 \theta(t,0) &= (1+t)^{\frac{2(1+m)}{D} + \frac{2(1+m+n)}{D}\lambda}\Theta(0),
 \\
 \theta(t,x) &\sim t^{\frac{2(1+m)}{D} - \frac{(1+m+n)^2}{D(\alpha-m-n)}\lambda}|x|^{-\frac{1+m+n}{\alpha-m-n}}, \quad \text{as $t \rightarrow \infty$, $x\ne0$.}
\end{align*}
Again, the positivity of the growth rate $\frac{2(1+m)}{D} - \frac{(1+m+n)^2}{D(\alpha-m-n)}\lambda$ is a consequence of \eqref{eq:lambda-range}.
\item Strain rate : The growth rates of the strain-rate is by definition less by one to those of the strain, again illustrating localization.
\begin{align*}
 u(t,0) &= (1+t)^{\frac{1+m}{D} + \frac{2+2\alpha}{D}\lambda}U(0),
 \\
 u(t,x) &\sim t^{\frac{1+m}{D} - \frac{(1+\alpha)(1+m+n)}{D(\alpha-m-n)}\lambda}|x|^{-\frac{1+\alpha}{\alpha-m-n}}, \quad \text{as $t \rightarrow \infty$, $x\ne0$.}
\end{align*}
\item Stress : The stress decays with time at all points, but the decay at  $x=0$ is much faster than the decay in other places,
indicating stress-collapse in the interior of the band:
\begin{align*}
 \sigma(t,0) &= (1+t)^{\frac{-2\alpha+2m+n}{D} + \frac{-2\alpha+2m+2n}{D}\lambda}\Sigma(0),
 \\
 \sigma(t,x) &\sim t^{\frac{-2\alpha+2m+n}{D} +\frac{1+m+n}{D}\lambda}|x|, \quad \text{as $t \rightarrow \infty$, $x\ne0$,}
\end{align*}
The difference of the two rates is $ \big ( \frac{1+m+n}{D} - \frac{-2\alpha+2m+n}{D} \big ) \lambda= \lambda$.

\item Velocity : The velocity is an odd function of $x$. At fixed $t$, $v(t,x)$ is an increasing function of $x$ ranging from $-v_\infty(t)$ to $v_\infty(t)$,
where  $v_\infty(t)\triangleq \lim_{x \rightarrow \infty} v(t,x)$. The velocity field is contrasted with the linear field of uniform shear motion. The self-similar scaling  $\xi=(1+t)^\lambda x$ implies
that most of the transition takes place around the origin leading eventually to step-function behavior as time goes to infinity.
The asymptotic velocity is
$$v_\infty(t)=(1+t)^{b}V_\infty = (1+t)^{\frac{1+m}{D} + \frac{1+m+n}{D}\lambda}V_\infty, \quad V_\infty \triangleq \lim_{\xi \rightarrow \infty} V(\xi) <\infty.$$
Note that the far field loading condition is different from the linear profile of uniform shearing. This deviation is a consequence of our simplifying assumption of self-similarity.
\end{itemize}

\section{Numerical computation of the heteroclinic orbit}
\label{sec:numerics}

In this section we present in detail the process we followed to capture numerically the heteroclinic orbit connecting $M_0$ and $M_1$. This is a
challenging computational task since both $M_0$ and $M_1$ are saddle points and the heteroclinic orbit connecting
them is the intersection of two 3-dimensional manifolds $W^u(M_0)$ and $W^s(M_1)$ in $\mathbb{R} ^4$.
Here, we use the software package \emph{AUTO}, \cite{Doedel_1981}, \cite{DK_1986}, \cite{DCFKSW_1999} to compute the heteroclinic orbit connecting $M_0$ and $M_1$. One of the main capabilities of \emph{AUTO} is that it can perform limited bifurcation analysis for parametric systems of ordinary differential equations of the form : $\displaystyle u^{\prime}(t) = f(u(t),\chi)$
where $f(\cdot,\cdot), \ u(\cdot) \in \mathbb{R}^d$ and $\chi$ could be one or a set of free parameters.

A direct application of \emph{AUTO} for solving system \eqref{eq:slow} and computing the desired  heteroclinic orbit will fail. A more careful approach has to be considered, starting from some well prepared data and continuing by exploiting the continuation capabilities of \emph{AUTO}.
Indeed, we start with an exact solution of system  \eqref{eq:slow} available for  a specific value of the parameter $\alpha$ and variable $p$, followed by a projection and two continuation steps escaping from these particular choices for $\alpha$ and $p$ allowing us to  compute the heteroclinic orbit.
We proceed by describing these four steps in detail.

\subsection{Continuation by AUTO}\label{sec:contauto}

{\bf Step 1.} (Exact solution) The system \eqref{eq:slow} admits an explicit solution for certain values of the parameters $\alpha, n$ and the variable $p$. First for $\alpha=0$  the equations for $p, q, r$ in \eqref{eq:slow} decouple from the equation of $s$. This reduced system carries only the three dimensional unstable manifold of $M_0$ characterizing the heteroclinic orbit as a node-saddle connection. Hence, in principle, by running time backwards and using a shooting argument in a small neighbourhood of $M_1$, any heteroclinic orbit can be computed as accurate as the numerical time integrator allows. However we can be more precise and prepare the data even better by
noticing that for $p\equiv 0$ the equation for $q$ in \eqref{eq:slow} decouples completely from the rest and can be solved explicitly. Further, using this analytic value of $q$ an exact solution can be also derived for $r$ in the case when $\displaystyle n=\frac{1}{k}$, $k \ge 1, \ k\in \mathbb{Z}$ :
\begin{equation}
\label{exa0p0}
\begin{aligned}
 &\alpha=0, \quad p\equiv0, \quad q(\eta) = \frac{1}{1+e^{-\eta}},\\
 & r(\eta) = \frac{r_0 \left(1 + e^{\eta}\right)^k}{ \displaystyle \sum_{j=0}^k \frac{kW_0}{kW_0 -j}\begin{pmatrix} k\\j\end{pmatrix} e^{j\eta}} , \quad \text{where}\ W_0= - \frac{(m+n)r_0}{\lambda}.
\end{aligned}
\end{equation}

{\bf Step 2.} (Projection step, $\alpha=0, \ p\equiv 0$) At this step we integrate numerically the equation for $s$ using the  exact values of $q(\eta), r(\eta)$ found in the previous step. The integration can be performed by either \emph{AUTO} or any other numerical integrator. The integration timespan is chosen to be $\eta\in[-\eta_{max},\eta_{max}]$ so that
the starting point $(p,q,r,s)|_{\eta=-\eta_{max}}$ and the ending point $(p,q,r,s)|_{\eta=\eta_{max}}$ both fall in small neighbourhoods of $M_0$ and $M_1$ respectively. In particular we choose the starting point so that $(p,q,r,s)|_{\eta=-\eta_{max}}=M_0 + \epsilon_0 \nu_0$, $\nu_0 = X_{02}$ and $\epsilon_0$ as small parameter. Using this an initial value we integrate numerically
the following non-autonomous equation for $s$
$$
\dot{s} =s\Big(\frac{-m-n}{\lambda}(r(\eta)-a) + q(\eta) - \frac{1}{\lambda}r(\eta)\big(s- (1+m+n)\big) - \frac{n}{\lambda}\Big).
$$
At the end of the calculation we project the vector $(p,q,r,s)|_{\eta=\eta_{max}}- M_1$ to the stable {\it subspace} of $M_1$.  Indeed,  we can find explicitly $\epsilon_1\ll 1$ and $\nu_1 \in \underset{}{ \textrm{Span}}\{X_{11},X_{12},X_{14}\}$, $|\nu_1|=1$ such that
$\pi\big((p,q,r,s)|_{\eta=\eta_{max}}- M_1\big) =\epsilon_1 \nu_1$, where $\pi$ denotes the projection.
At the completion of Step 2 we have the solution $(p,q,r,s)(\eta), \eta\in[-\eta_{max},\eta_{max}]$ at discrete levels $\eta_i, i=0,\dots N$ for $\alpha=0$ and lying in the plane $p\equiv 0$.

{\bf Step 3.} (Continuation with $\alpha\neq 0, \ p=0$) The goal in this step is to create a set of orbits in the plane $p\equiv 0$ but with $\alpha$ not any more trivial. To that effect, we use the well prepared data obtained in Step 2 and we run \emph{AUTO} with $\nu_0$ fixed but allowing $\alpha$, $m$, $\lambda$, $\epsilon_0$, $\epsilon_1$, $\nu_1$ be continued. The continuation process performed by \emph{AUTO} creates a family of orbits with the following characteristics : a) emanate from a small neighbourhood of size $\epsilon_0$ of $M_0$ in the direction of  $\nu_0$, b) terminate in a small neighbourhood of size $\epsilon_1$ of $M_1$, c) lie in the plane $p\equiv 0$ but with $\alpha\neq 0$.

{\bf Step 4.} (Continuation with $\alpha\neq 0, \ p\neq 0$) In this step we capture the desired heteroclinic orbit connecting $M_0$ and $M_1$. From the family of orbits obtained in Step 3 we select one according to the physical relevance of the parameters $\alpha$, $m$, $n$ and $\lambda$. We run \emph{AUTO} again allowing  $\nu_0$ in $\underset{}{ \textrm{span}}\{X_{01},X_{02},X_{03}\}$ to be continued, thus leaving the plane $p\equiv 0$. \emph{AUTO} generates a family of orbits emanating from $M_1$, terminating in a neighbourhood of size $\epsilon_0$ of $M_0$ and is the 2-surface of heteroclinic orbits of $W^u(M_0)\cap W^s(M_1)$. One of these orbits is the desired heteroclinic orbit  with $\nu_0=X_{01}$.

\begin{remark}
We note that the exact solution of $r$ in \eqref{exa0p0} is valid only for values of $n$ of the form $n=\frac{1}{k}, \ k\ge 1, k\in\mathbb{Z}$. In the case that $n$ is not of this form then one can rely on an numerical integrator for solving as accurately as possible the equation for $r$ using the exact value of $q$.
\end{remark}

\subsection{Numerical Results}\label{sec:numres}
In this section we illustrate the computation of the heteroclinic orbit of \eqref{eq:slow}, following the steps described in detail above. Further, using the heteroclinic orbit $(p,q,r,s)$ of system \eqref{eq:slow} we compute the associated self-similar solution in terms of the original variables $v(x,t), \ u(x,t),\ \theta(x,t),\ \sigma(x,t)$.

We begin by giving the explicit relation of the variables $(p(\eta),q(\eta),r(\eta),s(\eta))$ to the original variables $v(x,t), \ u(x,t)$, $\theta(x,t),\ \sigma(x,t)$. Indeed, collecting the transformations and change of variables described in \eqref{eq:ORItoCAP}, \eqref{eq:CAPtoBAR} and \eqref{eq:BARtoTIL}, we obtain

\begin{equation}
\label{eq:pqrsTOvugs}
\begin{aligned}
v(x,t) &= \frac{1}{b} t^b\ \xi^{-b_1}\ p^{\frac{(-\alpha+m+n)}{D}} s^{\frac{\alpha}{D}} r^{\frac{n}{D}}\quad\qquad u(x,t) = t^{b+\lambda}\ \xi^{-b_1-1}\ p^{\frac{(1+\alpha)}{D}} s^{\frac{\alpha}{D}} r^{1+\frac{n}{D}} \\
\theta(x,t) &= t^c\ \xi^{-c_1}\ p^{\frac{(1+m+n)}{D}} s^{\frac{m+n-1}{D}} r^{\frac{2n}{D}},\qquad
\sigma(x,t) = t^d\ \xi^{-d_1}\ p^{\frac{(-\alpha+m+n)}{D}} s^{\frac{\alpha}{D}} r^{\frac{n}{D}}
%\gamma(x,t) &=  t^a\  \xi^{-a_1}\ p^{\frac{1+\alpha}{D}} s^{\frac{\alpha}{D}}r^{\frac{n}{D}},
\end{aligned}
\end{equation}
where $\eta=\log\xi,\ \xi=t^{\lambda} x$ and $a, b,  c,  d,  D$ as in \eqref{eq:exponents}, \eqref{defD}.   We present now the results of  three numerical experiments. All the computations where performed with $\eta_{max}=10$ and $\lambda = \frac{1}{2}\lambda_{max}$ where $\lambda_{max}$ is the upper bound of $\lambda$ in \eqref{eq:lambda-range}. The initial values of $\alpha$, $m$ are taken as $\alpha=0$ and $m=-0.6, \ -0.5, \ -0.5$, respectively.
The corresponding values of $n$ remained fixed throughout the process and were $n=0.025, \ 0.0125, \ 0.01$.
Following the process  described in Section  \ref{sec:contauto}, the software \emph{AUTO} was able to perform the continuation process
and capture the desired heteroclinic orbit. The resulting values for
$\alpha$ and $m$ are shown in the figures, along with the value of the parameter $L_p = -\alpha + m + n$.  The change of sign of $L_p$ from positive to negative signals the onset of localization.

Figures \ref{fig_n40}, \ref{fig_n80} and \ref{fig_n100} illustrate the emergence of
localization by depicting the profiles of the original variables $v,\ u, \ \theta, \ \sigma$ at a few time instances. The vertical axes, except for the velocity $v$, are in logarithmic scale, however the corresponding $y-$range of values for each variable, is the same in all figures.   In part (a) of each figure the velocity profile is depicted,  which eventually, attains the shape of a step function. Parts
(b) and (c) of the figures present the localization in strain rate and temperature respectively. In both cases the initial profile is a small perturbation of a constant state which at later time localizes at the origin. On the other hand, part (d) of figures
shows the collapse of the stress to zero.
The rate of localization at the origin  differs for each numerical experiment and depends on the values of the materials parameters $\alpha, \ m, \ n$ and $L_p$.
In particular, the onset of localization is characterized by the parameter $L_p$ taking a negative value. Further,  the rate of localization is determined by the magnitude of this negative value, with larger negative values indicating faster localization, as it is observed in Figures \ref{fig_n40}-\ref{fig_n100}.

\begin{figure}
\centering\includegraphics[width=14cm, height=10cm]{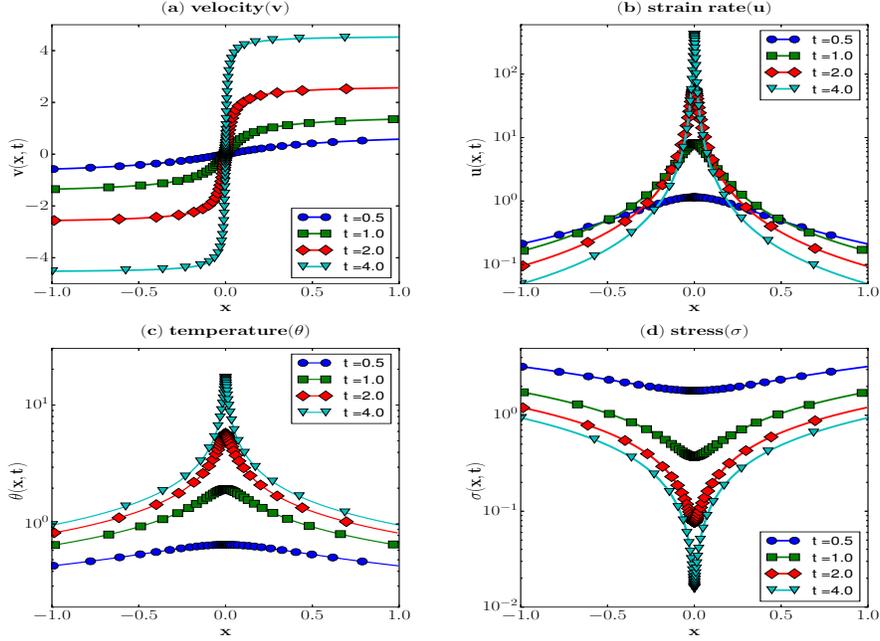}
\caption{$\alpha=1.572, \ m=0.02246, \ n=0.025, \ L_p = -1.52454$.}
\label{fig_n40}
\end{figure}

\begin{figure}
\includegraphics[width=14cm, height=10cm]{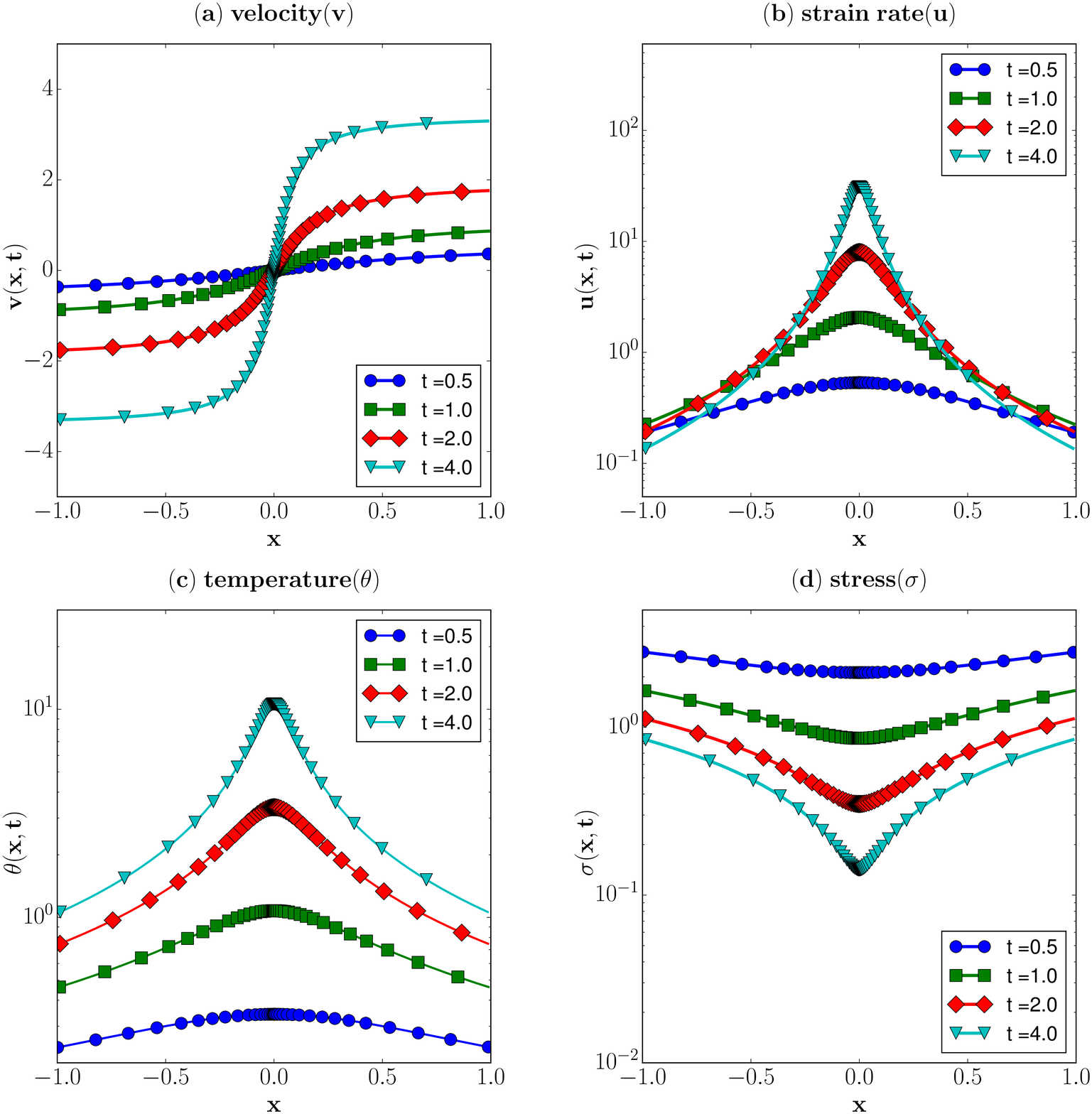}
\centering\caption{$\alpha=1.1698, \ m=0.2057, \ n=0.0125, \ L_p = -0.9516$.}
\label{fig_n80}
\end{figure}

\begin{figure}
\includegraphics[width=14cm, height=10cm]{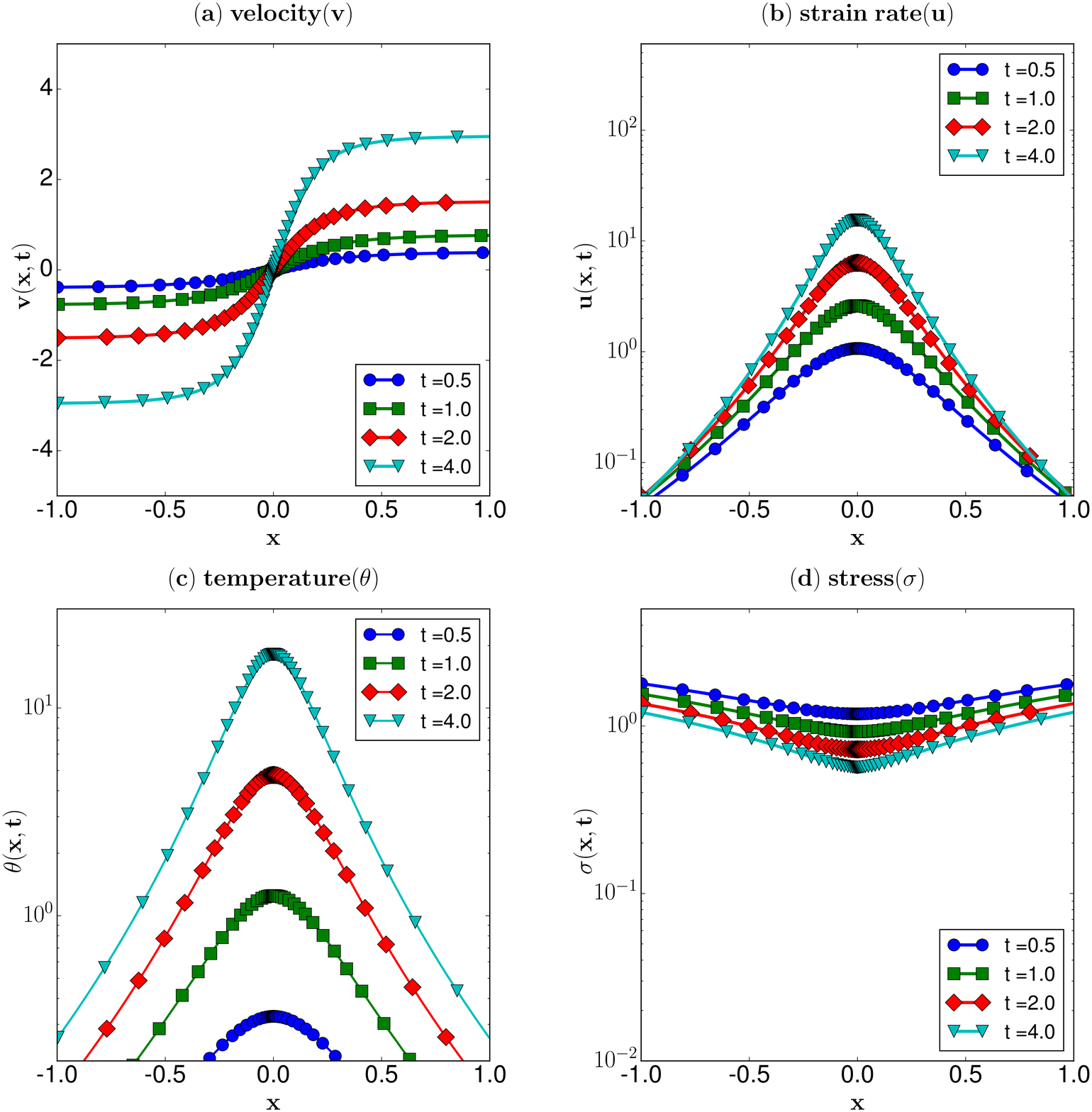}
\centering\caption{$\alpha=0.5957, \ m=0.3437, \ n=0.01, \ L_p = -0.242$.}
\label{fig_n100}
\end{figure}

There is the possibility of constructing the heteroclinic orbits via a shooting method, but this works only in the special cases $\alpha = 0$, $m <0$ and  $m=0$, see \cite{KLT_2016}
and \cite{KLT_HYP2016} respectively. The shooting method does not work in the general case where all parameters are nonzero.

\appendix
\section{The loss of hyperbolicity for $n = 0$} \label{append:hadamard}
Consider the system \eqref{intro-system1} when $n=0$, that is the viscoplastic effects are neglected. Then \eqref{consti} reads
\begin{equation}
\label{constisp}
\sigma = \tau (\theta, \gamma) = \theta^{-\alpha}\gamma^m
\end{equation}
and  \eqref{intro-system1} is written as a first order system
\begin{equation} \label{eq:transport}
 \begin{pmatrix} v_t \\ \theta_t \\ \gamma_t \end{pmatrix} = \underbrace{
 \begin{pmatrix}
 0 & \tau_\theta (\theta, \gamma) & \tau_\gamma (\theta, \gamma) \\
 \tau(\theta, \gamma) & 0 & 0 \\
  1 & 0 & 0 \\
 \end{pmatrix}}_\text{$\triangleq B (\theta, \gamma)$}
  \begin{pmatrix} v_x  \\ \theta_x \\  \gamma_x  \end{pmatrix}.
\end{equation}
We check hyperbolicity for  \eqref{eq:transport}. The characteristic speeds are the roots of
\begin{align*}
 \det\big(B-\lambda \textrm{I}\big) &= -\lambda \big(\lambda^2 - (\tau_\theta \tau + \tau_\gamma )\big)  = 0
 \end{align*}
 The system is thus hyperbolic when $\tau_\theta \tau + \tau_\gamma > 0$ and elliptic in the $t$-direction when $\tau_\theta \tau + \tau_\gamma < 0$.
 Observe that along the evolution of \eqref{eq:transport} and under the conditions for loading of interest in our problem, we have that
 $\gamma$ is increasing; the equation
 $$
 \theta_t = \tau(\theta, \gamma) \gamma_t
 $$
 implies that
 $$
 \tau_\theta \tau + \tau_\gamma = \frac{d}{d\gamma} \tau (\theta, \gamma) \, .
 $$
 We conclude that the system is hyperbolic  before  the maximum of the stress-strain curve,  and elliptic beyond the maximum point.
 For the constitutive law \eqref{constisp} a computation shows
 $$
 \begin{aligned}
 \tau_\theta  \tau + \tau_\gamma &= \theta^{-\alpha } \gamma^{m-1} \big ( -\alpha \frac{\gamma^{m+1}}{\theta^{1 + \alpha}} + m \big )
 \\
 &= \frac{-\alpha + m}{1 + \alpha} + \frac{\alpha (1+m)}{\theta^{1+\alpha}}
 \big (  \frac{\theta_0(x)^{1+\alpha}}{1 + \alpha} - \frac{ \gamma_0 (x)^{1+m}}{1+m} \big )
 \end{aligned}
 $$
In the region $\alpha > m$ the stress-strain curve may be initially increasing (depending on the data)
 but eventually decreases.

The system  \eqref{eq:transport} admits the class of uniform shearing solutions
\begin{equation} \label{app-uss}
v_s (x) = x \, , \quad \gamma_s (t) = t + \gamma_0 \, , \quad \theta_s (t)  \; \mbox{is determined by solving}  \;
\begin{cases} \frac{d\theta_s}{dt} = \tau (\theta_s, \gamma_s) & \\
\theta_s (0) = \theta_0 & \\
\end{cases}
\end{equation}
where $\gamma_0$, $\theta_0 >$ are the initial strain and temperature, respectively. We linearize around the uniform shearing solution by setting
$$
v = x + \hat V(t,x) \, , \quad \theta = \theta_s (t) + \hat\Theta (t,x) \, , \quad \gamma = \gamma_s (t) + \hat\Gamma (t,x)
$$
and obtain the linearized system satisfied by the perturbation $(\hat V, \hat\Theta, \hat\Gamma )$,
\begin{equation}  % \label{eq:transport}
 \begin{pmatrix} \hat V_t \\ \hat\Theta_t \\ \hat\Gamma_t \end{pmatrix} =  B (\theta_s(t) , \gamma_s (t) )  \begin{pmatrix} \hat V_x  \\ \hat\Theta_x \\  \hat\Gamma_x  \end{pmatrix}
 +
 \begin{pmatrix}
  0 & 0 & 0 \\
 0 & \tau_\theta (\theta_s , \gamma_s ) & \tau_\gamma (\theta_s , \gamma_s ) \\
 0 & 0 & 0 \\
 \end{pmatrix} \;
  \begin{pmatrix} \hat V  \\ \hat\Theta \\  \hat\Gamma  \end{pmatrix}.
\end{equation}
The above calculation shows that, when $\alpha > m$,  the linearized system
loses hyperbolicity in finite time, past the maximum of the curve $\sigma_s (t) - t$.

\section{The equilibria of the system \eqref{eq:slow} }\label{append:equi_reject}
We discussed in section \ref{sec:equil} the equilibria $M_0$ and $M_1$ of \eqref{eq:slow}.
The remaining equilibria of \eqref{eq:slow} are listed below, and they are  all functions of $(\alpha,m,n,\lambda)$ that
lie outside the the sector
$$
\mathcal{P} = \{(p,q,r,s) \; | \; p\ge0, \,  q\ge0, \,  r>0, \, s>0 \}
$$
in the parameter range \eqref{eq:paramrange}.
The reader will find underlined  the components indicating that the equilibrium lies outside the sector of interest:
We recall  the notations
$$
a=\frac{2+2\alpha-n}{D} + \frac{2(1 + \alpha)}{D}\lambda \, , \quad
b=\frac{1+m}{D} + \frac{1+m+n}{D}\lambda \, , \quad
D=1+2\alpha-m-n
$$
while  $t$, $t_1$ and $t_2$ denote arbitrary real numbers.
\allowdisplaybreaks
\begin{align*}
 \begin{array}{lllllll}
  (1) &\Big(0,&0,& \underline{0},&\underline{0}&\Big), \\%&\text{and } \lambda = 111,\\
  (2) &\Big(0,&0,& \underline{0},&t&            \Big) &\text{provided } \lambda = \frac{-2\alpha+2m+n}{2(\alpha-m-n)},\\
  (3) &\Big(0,&0,& \frac{n\alpha - a(\alpha-m-n)}{(1+\alpha)(m+n)}, &\underline{0}&\Big),\\
  (4) &\Big(0,&1,& \underline{0},&\underline{0}&\Big), \\
  (5) &\Big(0,&1,& \underline{0},&t&           \Big) &\text{provided } \lambda = \frac{2\alpha-2m-n}{1+m+n},\\
  (6) &\Big(0,&1,& \frac{n\alpha - a(\alpha-m-n)}{(1+\alpha)(m+n)}+\frac{\lambda}{m+n}, &\underline{0}&\Big),\\
  (7) &\Big(t,&0,& \underline{0},&\underline{0}&\Big) &\text{provided } \lambda=\frac{2+2\alpha-n}{2(\alpha-m-n)}, \\
  (8) &\Big(t,&1,& \underline{0},&\underline{0}&\Big) &\text{provided } \lambda=\frac{-2-2\alpha+n}{1+m+n}, \\
  (9) &\Big(t_1,&0,& \underline{0},&t_2&           \Big) &\text{provided $1+2\alpha-m-n=0$ and $\lambda=\frac{-1-m}{1+m+n}$}, \\
 (10) &\Big(t_1,&1,& \underline{0},&t_2&           \Big) &\text{provided $1+2\alpha-m-n=0$ and $\lambda=\frac{-1-m}{1+m+n}$}, \\
 \end{array}
\end{align*}
% \begin{align*}
%  \begin{array}{lllllll}
% (11) &\Big(\underline{ -\frac{(\alpha-m-n)(1+m+n)}{(1+\alpha)(1+m)}},&\frac{2(\alpha-m-n)}{1+m}b, &r=\frac{2(1+\alpha)}{1+2\alpha-m-n}, &s=\frac{1+m+n}{1+\alpha} - \frac{n(1+2\alpha-m-n)}{2(1+\alpha)^2}&\Big),\\
% \end{array}
% \end{align*}
\begin{align*}
(11)\qquad  \Big(&\underline{ -\tfrac{(\alpha-m-n)(1+m+n)}{(1+\alpha)(1+m)}}, \:\tfrac{2(\alpha-m-n)}{1+m}b, \:\tfrac{2(1+\alpha)}{1+2\alpha-m-n}, \:\tfrac{1+m+n}{1+\alpha} - \tfrac{n(1+2\alpha-m-n)}{2(1+\alpha)^2}\Big),
\\[4pt]
(12)\qquad  \bigg(&\Big ( \tfrac{2\alpha(1+m)}{D(1-m-n)} + \tfrac{2(\alpha-m-n)}{D}\lambda\Big)\Big(\tfrac{2\alpha(1+m)}{D(1-m-n)} - \tfrac{1+m+n}{D}\lambda\Big)\tfrac{1-m-n}{\lambda(2-n)}\tfrac{1-m-n}{\lambda(1+m)}, \\
 &\qquad \left( \tfrac{2\alpha(1+m)}{D(1-m-n)} + \tfrac{2(\alpha-m-n)}{D}\lambda\right)\left(\tfrac{1+m}{D} + \tfrac{1+m+n}{D}\lambda\right)\tfrac{1-m-n}{\lambda(1+m)}, \ \tfrac{2-n}{1-m-n}, \ \underline{0}\bigg).
\end{align*}
The generic equilibria in $\mathcal{P}$ are $M_0$, $M_1$, (1), (3-4), (6), and (11 -12); the rest are valid for specific parameter values.

\section{The linearized problems around $M_0$ and $M_1$}\label{append:lin}
The coefficient matrix for the linearized system \eqref{eq:slow} around the equlibrium $M_0$ is
\begin{align*}
 \begin{pmatrix}
          2 & 0 & 0 & 0 \\
          br_0 & 1 & 0 & 0\\
          \frac{r_0}{n}(\lambda r_0) & \frac{r_0}{n} & \frac{r_0}{n}\Big(\frac{\alpha-m-n}{\lambda(1+\alpha)} - \frac{n\alpha}{\lambda(1+\alpha)r_0}\Big) & \frac{r_0}{n}(\frac{\alpha r_0}{\lambda})\\
          s_0(\lambda r_0) & s_0 & s_0\Big(\frac{\alpha-m-n}{\lambda(1+\alpha)} + \frac{n}{\lambda(1+\alpha)r_0}\Big) & s_0(-\frac{r_0}{\lambda})
         \end{pmatrix}
        =\begin{pmatrix}
          2 & 0 & 0 & 0 \\
          br_0 & 1 & 0 & 0\\
          \frac{r_0}{n}(\lambda r_0) & \frac{r_0}{n} & \frac{r_0}{n}\frac{1}{\lambda}\Big(1-s_0-\frac{n}{r_0}\Big) & \frac{r_0}{n}(\frac{\alpha r_0}{\lambda})\\
          s_0(\lambda r_0) & s_0 & s_0\frac{1}{\lambda}(1-s_0) & s_0(-\frac{r_0}{\lambda})
         \end{pmatrix}
\end{align*}
The corresponding eigenvectors $X_{0j}$ are collected in the matrix $S_0$ as $j$-th column vector, $j=1,2,3,4$.
\begin{equation} \label{eq:S0}
\begin{aligned}
 S_0&=
 \begin{pmatrix}
    1 & 0 & 0 & 0\\
    br_0 & 1 & 0 & 0\\
    y_1 & y_2 & 1 & y_4\\
    z_1 & z_2 & z_3 &1
 \end{pmatrix},
 \quad \quad
 \begin{array}{l}
\begin{pmatrix}
 y_1\\z_1
\end{pmatrix}
=-(\lambda+b)r_0\begin{pmatrix}
  \frac{ \frac{1+\alpha}{\lambda}r_0 + \frac{2}{s_0} }{ \Delta_1 }\\
  \frac{ \frac{n}{r_0}\big(\frac{1}{\lambda} + 2\big) }{ \Delta_1 }
  \end{pmatrix},
  \quad
 \begin{pmatrix}
 y_2\\z_2
\end{pmatrix}
=-\begin{pmatrix}
  \frac{ \frac{1+\alpha}{\lambda}r_0 + \frac{\mu_{02}}{s_0} }{ \Delta_2 }\\
  \frac{ \frac{n}{r_0}\big(\frac{1}{\lambda} + \mu_{02}\big) }{ \Delta_2 }
  \end{pmatrix}\\
%    y_1=-\frac{(\lambda+b)r_0}{\frac{1-s_0}{\lambda} - nA}, \quad A=\frac{\big(\frac{r_0}{\lambda}+\frac{2}{s_0}\big)\big(\frac{1}{\lambda}+2\big)\frac{1}{r_0}}{ \frac{1+\alpha}{\lambda}r_0 + \frac{2}{s_0} }, \\
%  z_1=n\bigg(\frac{\big(\frac{1}{\lambda}+2\big)\frac{1}{r_0}}{ \frac{1+\alpha}{\lambda}r_0 + \frac{2}{s_0} }\bigg)y_1, \\
%  y_2=-\frac{1}{\frac{1-s_0}{\lambda} - nB }, \quad B=\frac{\big(\frac{r_0}{\lambda}+\frac{1}{s_0}\big)\big(\frac{1}{\lambda}+1\big)\frac{1}{r_0}}{ \frac{1+\alpha}{\lambda}r + \frac{1}{s_0} }\\
%  z_2=n\bigg(\frac{\big(\frac{1}{\lambda}+1\big)\frac{1}{r_0}}{ \frac{1+\alpha}{\lambda}r_0 + \frac{1}{s_0} }\bigg)y_2, \\
 z_3=n\bigg(\frac{\frac{1-s_0}{\lambda}}{\frac{n r_0}{\lambda} + \frac{n\mu_{0}^+}{s_0}}\bigg),\quad y_4=\frac{\frac{r_0}{\lambda}+\frac{\mu_0^-}{s_0}}{\frac{1-s_0}{\lambda}},
 \end{array}
\end{aligned}
\end{equation}
where $\Delta_1 = \frac{1-s_0}{\lambda}\big(\frac{1+\alpha}{\lambda}r_0 + \frac{2}{s_0}\big) -\frac{n}{r_0} \big( \frac{1}{\lambda} + 2\big)\big(\frac{r_0}{\lambda} + \frac{2}{s_0}\big)$
%=\frac{-n}{r_0s_0}\det \left[\begin{pmatrix} \frac{r_0}{n}\big(\frac{1-s_0}{\lambda}-\frac{n}{\lambda r_0}\big) & \frac{r_0}{n}\frac{\alpha r_0}{\lambda}\\ s_0\frac{1-s_0}{\lambda} & -s_0\frac{r_0}{\lambda} \end{pmatrix} -2\textrm{I}\right]\ne0$
and $\Delta_2 = \frac{1-s_0}{\lambda}\big(\frac{1+\alpha}{\lambda}r_0 + \frac{1}{s_0}\big) -\frac{n}{r_0} \big( \frac{1}{\lambda} + 1\big)\big(\frac{r_0}{\lambda} + \frac{1}{s_0}\big)$.
%=\frac{-n}{r_0s_0}\det \left[\begin{pmatrix} \frac{r_0}{n}\big(\frac{1-s_0}{\lambda}-\frac{n}{\lambda r_0}\big) & \frac{r_0}{n}\frac{\alpha r_0}{\lambda}\\ s_0\frac{1-s_0}{\lambda} & -s_0\frac{r_0}{\lambda} \end{pmatrix} -1\textrm{I}\right]\ne0$ respectively for the corresponding cases.
We find that $y_1,y_2,y_4<0$; $z_1,z_2,z_3 \sim\BO(n)$, provided $n$ is sufficiently small.

Next, the coefficient matrix for the linearized system around $M_1$ is
\begin{align*}
 \begin{pmatrix}
          -\frac{1+m+n}{\alpha-m-n} & 0 & 0 & 0\\
          (b-\lambda)r_1 & -1 & 0 & 0\\
          \frac{r_1}{n}(\lambda r_1) & \frac{r_1}{n} & \frac{r_1}{n}\Big(\frac{\alpha-m-n}{\lambda(1+\alpha)} - \frac{n\alpha}{\lambda(1+\alpha)r_1}\Big) & \frac{r_1}{n}(\frac{\alpha r_1}{\lambda})\\
          s_1(\lambda r_1) & s_1 & s_1\Big(\frac{\alpha-m-n}{\lambda(1+\alpha)} + \frac{n}{\lambda(1+\alpha)r_1}\Big) & s_1(-\frac{r_1}{\lambda})
         \end{pmatrix}
         =\begin{pmatrix}
          -\frac{1+m+n}{\alpha-m-n} & 0 & 0 & 0\\
          (b-\lambda)r_1 & -1 & 0 & 0\\
          \frac{r_1}{n}(\lambda r_1) & \frac{r_1}{n} & \frac{r_1}{n}\frac{1}{\lambda}\Big(1-s_1-\frac{n}{r_1}\Big) & \frac{r_1}{n}(\frac{\alpha r_1}{\lambda})\\
          s_1(\lambda r_1) & s_1 & s_1\frac{1}{\lambda}(1-s_1) & s_1(-\frac{r_1}{\lambda})
         \end{pmatrix}
\end{align*}

%It turned out that the coefficient matrix for $M_1$ possibly possesses the Jordan block when $\mu_{11}$ and $\mu_{12}$ coincide. We first demonstrate the four eigenvectors for the generic case where $\mu_{11}\ne\mu_{12}$ and then specify the generalized eigenvectors for this special case.

In what follows we examine all possible cases:  Except for the case $\mu_{11}=\mu_{12}=-1$, four linearly independent eigenvectors are attained.
In the exceptional case $\mu_{11}=\mu_{12}=-1$ the repeated eigenvalue $-1$ has geometric multiplicity which is strictly less that its  algebraic multiplicity.

As to the eigenvectors, notice that the eigenvalues for $M_1$  (differently from those for $M_0$) have the chance to be repeated. The analysis below
shows  that, unless $\mu_{11}=\mu_{12}=-1$, four linearly independent eigenvectors are attained. If the exceptional case takes place then
we will supplement precisely one generalized eigenvector for the repeated eigenvalue $-1$.

{\bf Case 1. $-\frac{1+m+n}{\alpha-m-n}\ne -1$; or $-\frac{1+m+n}{\alpha-m-n}= -1$ but $b=\lambda$. } This case yields four linearly independent eigenvectors. The eigenvectors $X_{1j}$ are collected in the matrix $S_1$ as $j$-th column vector, $j=1,2,3,4$, and in the case of repeated eigenvalues
the corresponding eigenvectors are understood as a basis for the associated subspace:
\begin{align*}
 S_1&=
 \begin{pmatrix}
    1 & 0 & 0 & 0\\
    x_1 & 1 & 0 & 0\\
    y_1 & y_2 & 1 & y_4\\
    z_1 & z_2 & z_3 &1
 \end{pmatrix}, \quad \quad
 \begin{array}{l}
  x_1=
 \begin{cases}
  \frac{(b-\lambda)r_1}{1+\mu_{11}} & \text{if $\mu_{11}\ne -1$,}\\
  0 & \text{otherwise,}
 \end{cases}\\
 z_3=n\bigg(\frac{\frac{1-s_1}{\lambda}}{\frac{n r_1}{\lambda} + \frac{n\mu_{1}^+}{s_1}}\bigg), \quad y_4=\frac{\frac{r_1}{\lambda}+\frac{\mu_1^-}{s_1}}{\frac{1-s_1}{\lambda}},\\
 \end{array}
\end{align*}
\begin{equation} \label{eq:S1-1}
\begin{aligned}
\begin{pmatrix}
 y_1\\z_1
\end{pmatrix}
=\begin{cases}
  -(\lambda r_1 + x_1)\begin{pmatrix}
  \frac{\lambda}{1-s_1}\\0
  \end{pmatrix} & \text{if $\mu_{14}=\mu_{11}$,}\\
  -(\lambda r_1 + x_1)
  \begin{pmatrix}
  \frac{ \frac{1+\alpha}{\lambda}r_1 + \frac{\mu_{11}}{s_1} }{ \Delta_3 }\\
  \frac{ \frac{n}{r_1}\big(\frac{1}{\lambda} + \mu_{11}\big) }{ \Delta_3 }
  \end{pmatrix} & \text{otherwise,}
 \end{cases}
 \quad
 \begin{pmatrix}
 y_2\\z_2
\end{pmatrix}
=\begin{cases}
 -\begin{pmatrix}
  \frac{\lambda}{1-s_1}\\0
  \end{pmatrix} & \text{if $\mu_{14}=\mu_{12}$,}\\
  -\begin{pmatrix}
  \frac{ \frac{1+\alpha}{\lambda}r_1 + \frac{\mu_{12}}{s_1} }{ \Delta_4 }\\
  \frac{ \frac{n}{r_1}\big(\frac{1}{\lambda} + \mu_{12}\big) }{ \Delta_4 }
  \end{pmatrix} & \text{otherwise,}
 \end{cases}
\end{aligned}
\end{equation}
where
\begin{align*}
 \Delta_3 &= \frac{1-s_1}{\lambda}\big(\frac{1+\alpha}{\lambda}r_1 + \frac{\mu_{11}}{s_1}\big) -\frac{n}{r_1} \big( \frac{1}{\lambda} + \mu_{11}\big)\big(\frac{r_1}{\lambda} + \frac{\mu_{11}}{s_1}\big)
 \\
 &=\frac{-n}{r_1s_1}\det \left[\begin{pmatrix} \frac{r_1}{n}\big(\frac{1-s_1}{\lambda}-\frac{n}{\lambda r_1}\big) & \frac{r_1}{n}\frac{\alpha r_1}{\lambda}\\ s_1\frac{1-s_1}{\lambda} & -s_1\frac{r_1}{\lambda} \end{pmatrix} -\mu_{11}\textrm{I}\right]\ne0,
 \\
 \Delta_4 &= \frac{1-s_1}{\lambda}\big(\frac{1+\alpha}{\lambda}r_1 + \frac{\mu_{12}}{s_1}\big) -\frac{n}{r_1} \big( \frac{1}{\lambda} + \mu_{12}\big)\big(\frac{r_1}{\lambda} + \frac{\mu_{12}}{s_1}\big)
 \\
 &=\frac{-n}{r_1s_1}\det \left[\begin{pmatrix} \frac{r_1}{n}\big(\frac{1-s_1}{\lambda}-\frac{n}{\lambda r_1}\big) & \frac{r_1}{n}\frac{\alpha r_1}{\lambda}\\ s_1\frac{1-s_1}{\lambda} & -s_1\frac{r_1}{\lambda} \end{pmatrix} -\mu_{12}\textrm{I}\right]\ne0
\end{align*}
respectively for the corresponding cases.
% $\Delta_3 = \frac{1-s_1}{\lambda}\big(\frac{1+\alpha}{\lambda}r_1 + \frac{\mu_{11}}{s_1}\big) -\frac{n}{r_1} \big( \frac{1}{\lambda} + \mu_{11}\big)\big(\frac{r_1}{\lambda} + \frac{\mu_{11}}{s_1}\big)=\frac{-n}{r_1s_1}\det \left[\begin{pmatrix} \frac{r_1}{n}\big(\frac{1-s_1}{\lambda}-\frac{n}{\lambda r_1}\big) & \frac{r_1}{n}\frac{\alpha r_1}{\lambda}\\ s_1\frac{1-s_1}{\lambda} & -s_1\frac{r_1}{\lambda} \end{pmatrix} -\mu_{11}\textrm{I}\right]\ne0$ and $\Delta_4 = \frac{1-s_1}{\lambda}\big(\frac{1+\alpha}{\lambda}r_1 + \frac{\mu_{12}}{s_1}\big) -\frac{n}{r_1} \big( \frac{1}{\lambda} + \mu_{12}\big)\big(\frac{r_1}{\lambda} + \frac{\mu_{12}}{s_1}\big)=\frac{-n}{r_1s_1}\det \left[\begin{pmatrix} \frac{r_1}{n}\big(\frac{1-s_1}{\lambda}-\frac{n}{\lambda r_1}\big) & \frac{r_1}{n}\frac{\alpha r_1}{\lambda}\\ s_1\frac{1-s_1}{\lambda} & -s_1\frac{r_1}{\lambda} \end{pmatrix} -\mu_{12}\textrm{I}\right]\ne0$

{\bf Case 2. $-\frac{1+m+n}{\alpha-m-n}= -1$ and $b\ne\lambda$: }
For this case $\mu_{11} = \mu_{12} = -1$ has algebraic multiplicity two but its geometric multiplicity is one, so we replace the first column of $S_1$ by the generalized eigenvector
$\big(\frac{1}{(b-\lambda)r_1}, 0, y_1', z_1'\big)^T$, where
\begin{equation} \label{eq:S1-2}
\begin{aligned}
\begin{pmatrix}
 y_1'\\z_1'
\end{pmatrix}
=\begin{cases}
  \begin{pmatrix}
  -\frac{\lambda}{1-s_1}\big(\frac{\lambda}{b-\lambda} -\frac{n}{r_1}z_2\big)\\0
  \end{pmatrix} & \text{if $\mu_{14}=-1$,}\\
  -\frac{\lambda}{b-\lambda}
  \begin{pmatrix}
  \frac{ \frac{1+\alpha}{\lambda}r_1 + \frac{\mu_{11}}{s_1} }{ \Delta_3 }\\
  \frac{ \frac{n}{r_1}\big(\frac{1}{\lambda} + \mu_{11}\big) }{ \Delta_3 }
  \end{pmatrix} +
  \frac{n}{r_1}
  \begin{pmatrix}
  \frac{ y_2\big(\frac{r_1}{\lambda} + \frac{\mu_{11}}{s_1}\big) + z_2\frac{\alpha r_1}{\lambda} }{ \Delta_3 }\\
  \frac{ y_2\big(\frac{1-s_1}{\lambda}\big) + z_2\big(-\frac{1-s_1}{\lambda}+\frac{n}{r_1}\big(\frac{1}{\lambda}+\mu_{11}\big)\big) }{ \Delta_3 }
  \end{pmatrix} & \text{otherwise.}
 \end{cases}
\end{aligned}
\end{equation}

\begin{align*}
 0&=Mat_1 \begin{pmatrix} w\\x\\y\\z \end{pmatrix} -\mu \begin{pmatrix} w\\x\\y\\z \end{pmatrix}=
\begin{pmatrix}
 (\mu_{11}-\mu) w\\
 (b-\lambda)r_1 w +(\mu_{12}-\mu)x\\
 \frac{r_1}{n} \left[\lambda r_1 w + x + \big(\frac{1-s_1}{\lambda} - \frac{n}{r_1}\big(\frac{1}{\lambda}+\mu\big)\big)y + \frac{\alpha r_1}{\lambda}z\right]\\
 s_1 \left[ \lambda r_1 w + x + \big(\frac{1-s_1}{\lambda}\big)y -\big(\frac{r_1}{\lambda}+\frac{\mu}{s_1}\big)z\right]
 \end{pmatrix},\\
 A&\triangleq
 \begin{pmatrix}
 \frac{1-s_1}{\lambda} - \frac{n}{r_1}\big(\frac{1}{\lambda}+\mu\big) & \frac{\alpha r_1}{\lambda}\\
 \frac{1-s_1}{\lambda} & -\big(\frac{r_1}{\lambda}+\frac{\mu}{s_1}\big)
 \end{pmatrix}
 \begin{pmatrix} y\\z \end{pmatrix} = -(\lambda r_1 w +x)\begin{pmatrix} 1\\1\end{pmatrix}\\
 A^{-1} &= \frac{1}{\Delta}
 \begin{pmatrix} \big(\frac{r_1}{\lambda}+\frac{\mu}{s_1}\big) & \frac{\alpha r_1}{\lambda}\\
 \frac{1-s_1}{\lambda} & -\frac{1-s_1}{\lambda} + \frac{n}{r_1}\big(\frac{1}{\lambda}+\mu\big)
\end{pmatrix}, \quad
\Delta=\frac{1-s_1}{\lambda}\big( \frac{1+\alpha}{\lambda}r_1 + \frac{\mu}{s_1}\big) - \frac{n}{r_1}\big(\frac{1}{\lambda} +\mu\big)\big(\frac{r_1}{\lambda}+\frac{\mu}{s_1}\big)
\end{align*}

%
% \begin{align*}
%  y_1&=-\frac{(\lambda+b)r_0}{\frac{\alpha-m-n}{\lambda(1+\alpha)} - \frac{2n}{r_0(1+\alpha)}\Big(1 + \frac{(\frac{1}{\lambda}+2)\frac{\alpha}{s}}{ \frac{1+\alpha}{\lambda}r + \frac{2}{s} }\Big) }\\
%  y_2&=-\frac{1}{\frac{\alpha-m-n}{\lambda(1+\alpha)} - \frac{n}{r_0(1+\alpha)}\Big(1 + \frac{(\frac{1}{\lambda}+1)\frac{\alpha}{s}}{ \frac{1+\alpha}{\lambda}r + \frac{1}{s} }\Big) }\\
%  y_3&=\frac{\mu_0^- +\frac{r_0s_0}{\lambda}}{s\Big(\frac{\alpha-m-n}{\lambda(1+\alpha)} + \frac{n}{\lambda(1+\alpha)r_0}\Big)}\\
%  z_1&=n\bigg(\frac{\big(\frac{1}{\lambda}+2\big)\frac{1}{r}}{ \frac{1+\alpha}{\lambda}r + \frac{2}{s} }\bigg)y_1 \\
%  z_2&=n\bigg(\frac{\big(\frac{1}{\lambda}+1\big)\frac{1}{r}}{ \frac{1+\alpha}{\lambda}r + \frac{1}{s} }\bigg)y_2 \\
%  z_3&=n\bigg(\mu_0^+-\frac{r_0}{n}\Big(\frac{\alpha-m-n}{\lambda(1+\alpha)} - \frac{n\alpha}{\lambda(1+\alpha)r_0}\Big)\bigg)r_0(\frac{\alpha r_0}{\lambda})\bigg)
% \end{align*}

\bigskip
\noindent
{\bf Acknowledgement.}
The authors thank {\sc Prof.~Peter Szmolyan} for valuable discussions on the use of geometric singular perturbation theory.

\end{document}